\documentclass{amsart}

\usepackage{amssymb}
\usepackage{enumerate}

\makeatletter
\@namedef{subjclassname@2010}{%
  \textup{2010} Mathematics Subject Classification}
\makeatother

\newtheorem{theorem}{Theorem}[section] 
\newtheorem{claim}{Claim}[theorem]
\newtheorem{subclaim}{Subclaim}[claim]
\newtheorem{lemma}[theorem]{Lemma} 
\newtheorem{proposition}[theorem]{Proposition} 
\newtheorem{observation}[theorem]{Observation} 

\theoremstyle{definition}
\newtheorem{definition}[theorem]{Definition}
\newtheorem{example}[theorem]{Example}
\newtheorem{remark}[theorem]{Remark}

\newtheorem{context}[theorem]{Context}

\newcommand{\conc}{{}^\frown\!}
\newcommand{\lh}{{\rm lh}\/}
\newcommand{\rest}{{\restriction}}
\newcommand{\vtl}{\vartriangleleft}
\newcommand{\dom}{{\rm dom}} 
 
\newcommand{\rng}{{\rm rng}}
 
\newcommand{\suc}{{\rm succ}}
\newcommand{\mrot}{{\rm root}}
\newcommand{\rt}{{\rm rt}}

\newcommand{\cA}{{\mathcal A}}

\newcommand{\cB}{{\mathcal B}}

\newcommand{\bbC}{{\mathbb C}}

\newcommand{\bbD}{{\mathbb D}}

\newcommand{\cH}{{\mathcal H}}
\newcommand{\bbH}{{\mathbb H}}
\newcommand{\bbE}{{\mathbb E}}

\newcommand{\cI}{{\mathcal I}}

\newcommand{\bbP}{{\mathbb P}}

\newcommand{\bbQ}{{\mathbb Q}}

\newcommand{\cR}{{\mathcal R}}

\newcommand{\bbS}{{\mathbb S}}
\newcommand{\cS}{{\mathcal S}}

\newcommand{\cf}{{\rm cf}\/} 
 
\newcommand{\otp}{{\rm otp}\/} 
\newcommand{\st}{{\bf st}} 

\newcommand{\vare}{\varepsilon}

\newcommand{\forces}{\Vdash} 
\newcommand{\bV}{{\bf V}} 
\newcommand{\lesdot}{\mathrel{\mathord{<}\!\!\raise 
0.8 pt\hbox{$\scriptstyle\circ$}}} 

\newcommand{\pr}{{\rm pr}}

\newcommand{\qthree}{{{\mathbb Q}^3_\lambda}}

\newcommand{\qell}{{{\mathbb Q}^\ell_\lambda}}

\newcount\skewfactor
\def\mathunderaccent#1#2 {\let\theaccent#1\skewfactor#2
\mathpalette\putaccentunder}
\def\putaccentunder#1#2{\oalign{$#1#2$\crcr\hidewidth
\vbox to.2ex{\hbox{$#1\skew\skewfactor\theaccent{}$}\vss}\hidewidth}}
\def\name{\mathunderaccent\tilde-3 }

\begin{document}

\title{The last forcing standing with diamonds}

\author{Andrzej Ros{\l}anowski}
\address{Department of Mathematics\\
 University of Nebraska at Omaha\\
 Omaha, NE 68182-0243, USA}
\email{roslanow@member.ams.org}

\author{Saharon Shelah}
\address{Einstein Institute of Mathematics\\
Edmond J. Safra Campus, Givat Ram\\
The Hebrew University of Jerusalem\\
Jerusalem, 91904, Israel\\
 and \\ 
Department of Mathematics\\
Rutgers University\\
Hill Center - Busch Campus\\
110 Frelinghuysen Road\\
Piscataway, NJ 08854-8019, USA}

\email{shelah@math.huji.ac.il}
\urladdr{http://shelah.logic.at}

\date{May 2017}

\begin{abstract}
  This article continues Ros{\l}anowski and Shelah \cite{RoSh:655,
    RoSh:942}. We introduce a new property of $({<}\lambda)$--strategically
  complete forcing notions which implies that their $\lambda$--support
  iterations do not collapse $\lambda^+$.
\end{abstract}

\maketitle

\section{Introduction}
While there are still a lot of open problems left in the theory of forcing
iterated with finite and/or countable supports and we still need to expand
our preservation theorems, there is a sense that we {\em understand these
  iterations pretty well}. Therefore it is natural to look at iterations
with uncountable supports and ask for parallel theorems. The first attempt
could be to do nothing special and just repeat what has been done for CS
iterations. We could start in the way suggested already in Shelah
\cite{Sh:100} (but not used there).

\begin{definition}
\label{properdef}
Let $\lambda=\lambda^{<\lambda}$.  A notion of forcing $\bbP$ is said to
be {\em $\lambda$--proper in the standard sense\/} (or just: {\em
  $\lambda$--proper\/}) if for all sufficiently large regular cardinals
$\chi$, there is some $x\in\cH(\chi)$ such that whenever $M$ is an
elementary submodel of $\cH(\chi)$ satisfying  
\[|M|=\lambda,\qquad \bbP, x\in M\qquad M^{<\lambda}\subseteq M\]
and $p$ is an element of $M\cap\bbP$, then there is a condition $q\geq p$
such that 
\begin{equation*}
q\forces\mbox{``} M[\name{G}_\bbP]\cap{\rm Ord} = M\cap{\rm Ord}\mbox{''}.
\end{equation*}
\end{definition}

The $\lambda$--properness may seem to be a straightforward generalization of
``proper'', with the right consequences in place. For instance: 

\begin{theorem}
[Folklore; cf.~Hyttinen and Rautila {\cite[Section 3]{HyRa01}}]
\label{basic}
Assume $\lambda^{<\lambda}=\lambda$ is an uncountable cardinal.
\begin{enumerate}
\item If a forcing notion $\bbP$ is either strategically
  $({\leq}\lambda)$--complete or it satisfies the $\lambda^+$--chain condition, 
  then $\bbP$ is $\lambda$--proper. 
\item If $\bbP$ is $\lambda$--proper, $p\in\bbP$, $\name{A}$ is a
  $\bbP$--name for a set of ordinals and $p\forces |\name{A}|\leq\lambda$,
  then there are a condition $q\in\bbP$ stronger than $p$ and a set $B$ of
  size $\lambda$ such that $q\forces\name{A}\subseteq B$. 
\item If $\bbP$ is $\lambda$--proper, then 
\[\forces_{\bbP}\mbox{`` }(\lambda^+)^{\bV}\mbox{ is a regular cardinal
  ''.}\]  
Moreover, if $\bbP$ is also strategically $({<}\lambda)$--complete, then the 
forcing with $\bbP$ preserves stationary subsets of $\lambda^+$.  
\end{enumerate}
\end{theorem}

Also chain condition results look similarly:

\begin{theorem}
[Folklore; cf.~Eisworth {\cite[Proposition 3.1]{Ei03}} and Abraham
{\cite[Theorems 2.10, 2.12]{Ab10}}]  
\label{cc}
Assume $\lambda^{<\lambda}=\lambda$, $2^{\lambda}=\lambda^+$,  and let
$\bar{\bbP}=\langle \bbP_i, \name{\bbQ}_i:i<\lambda^{++}\rangle$ be a
$\lambda$--support iteration such that  
\begin{enumerate}
\item[(a)] $\bbP_i$ is $\lambda$--proper for $i\leq\lambda^{++}$
\item[(b)] $\forces_{\bbP_i}\mbox{`` }|\name{\bbQ}_i|\leq\lambda^+\mbox{
    ''}$. 
\end{enumerate}
Then 
\begin{enumerate}
\item $\bbP_{\lambda^{++}}$ satisfies the $\lambda^{++}$--chain condition,
  and 
\item for each $i<\lambda^{++}$ we have $\forces_{\bbP_i}
  2^\lambda=\lambda^+$.  
\end{enumerate}
\end{theorem}

\begin{proof}
  Abraham \cite[Theorem 2.10, Theorem 2.12]{Ab10} gives the full proof of
  the theorem for CS iterations. Eisworth \cite[Proposition 3.1]{Ei03}
  presents careful justification of \ref{cc}(1). Let us sketch arguments
  justifying \ref{cc}(2)

  To argue that $\forces_{\bbP_i}$`` $2^\lambda=\lambda^+$ '' suppose
  towards contradiction that $p\in \bbP_i$, $i<\lambda^{++}$, and
  $p\forces_{\bbP_i}$`` $\name{x}:\lambda^{++}\longrightarrow 2^\lambda$ is
  one-to-one ''. For each $\xi<\lambda^{++}$ fix a model $M_\xi\prec
  \cH(\chi)$ such that
\[\xi,\bar{\bbP},i,p,\name{x}\in M_\xi,\quad |M_\xi|=\lambda\quad \mbox{ and
}\quad {}^{<\lambda}M_\xi\subseteq M_\xi.\]
Applying ``a cleaning procedure'' as in the proof of \cite[Proposition
3.1]{Ei03} we may choose a set $I\subseteq \lambda^{++}$ of size
$\lambda^{++}$, a set $H\subseteq\lambda^{++}$ and mappings
$\pi_{\xi,\zeta}$ (for $\xi,\zeta\in I$ satisfying $\xi<\zeta$) such that:  
\begin{enumerate}
\item[(a)] $\pi_{\xi,\zeta}:M_\xi\stackrel{\rm onto}{\longrightarrow}
  M_\zeta$ is an $\in$--isomorphism, $\pi_{\xi,\zeta}(\xi)=\zeta$,
  $\pi_{\xi,\zeta}(\name{x})=\name{x}$, $\pi_{\xi,\zeta}(p)=p$,
\item[(b)] $M_\xi\cap M_\zeta\cap \lambda^{++}=H$,
\item[(c)] $H\subseteq \min(M_\xi\cap\lambda^{++}\setminus H)$ and
  $M_\xi\cap \lambda^{++}\subseteq \min(M_\zeta\cap\lambda^{++}\setminus
  H)$, and 
\item[(d)] $i\in H$, $\pi_{\xi,\zeta}\rest H={\rm id}_H$, $\lambda\subseteq
  H$.  
\end{enumerate}
The following claim will immediately complete the proof. 

\begin{claim}
If $r\in\bbP_i$ is $(M_\xi,\bbP_i)$--generic and $r\geq p$ and
$\xi<\zeta$ are from $I$, then $r\forces_{\bbP_i}\name{x}(\xi) 
=\name{x}(\zeta)$.      
\end{claim}

\begin{proof}[Proof of the Claim]
 Suppose towards contradiction that $r^*\geq r$, $r^*\forces
 \name{x}(\xi)\neq \name{x}(\zeta)$. Take $r^+\geq r^*$ and $\ell<2$,
 $\alpha<\lambda$ such that
\[r^+\forces \mbox{`` }\name{x}(\xi)(\alpha)=\ell\mbox{ and
}\name{x}(\zeta)(\alpha) =1-\ell\mbox{ ''.}\]
The condition $r^+$ is $(M_\xi,\bbP_i)$--generic, so there is $s\in
M_\xi\cap \bbP_i$ such that $s\forces_{\bbP_i} \name{x}(\xi)(\alpha)=\ell$
and $r^+,s$ are compatible, say $q\geq r^+,s$. Then $q$ is
$(M_\xi,\bbP_i)$--generic so by \cite[Claim 3.4]{Ei03} we also have
$q\geq\pi_{\xi,\zeta}(s)$. But $\pi_{\xi,\zeta}(s)\forces
\pi_{\xi,\zeta}(\name{x})
(\pi_{\xi,\zeta}(\xi))(\pi_{\xi,\zeta}(\alpha))=\ell$, so
$\pi_{\xi,\zeta}(s)\forces \name{x}(\zeta)(\alpha)=\ell$. Hence $q\forces
\name{x}(\xi)(\alpha)= \name{x}(\zeta)(\alpha)$, contradicting the choice of
$r^+$ and $q\geq r^+$. 
\end{proof}
\end{proof}

What is missing? The main point of properness is the preservation
theorem for CS iterations. If one tries to repeat the proof of the
preservation theorem for $\lambda$--support iterations of $\lambda$--proper
forcing notions, then one faces difficulties at limit stages of cofinality
fewer than $\lambda$ caused by the fact that it is inconvenient to
diagonalize $\lambda$ objects in less than $\lambda$ steps. This is a more
serious obstacle than just a technicality. Let us consider the following
forcing notion.   

\begin{example}
[Shelah {\cite[Appendix]{Sh:b}}]
\label{exam1}
Assume that $\lambda=\lambda^{<\lambda}$ is an uncountable cardinal and let
$\cS^{\lambda^+}_\lambda\stackrel{\rm  def}{=}\{
\delta<\lambda^+: \cf(\delta)=\lambda\}$. Suppose that a sequence $\langle 
A_\delta, h_\delta:\delta\in\cS^{\lambda^+}_\lambda\rangle$ is such
that for each $\delta\in\cS^{\lambda^+}_\lambda$: 
\begin{enumerate}
\item[(a)] $A_\delta\subseteq\delta$, $\otp(A_\delta)=\lambda$ and
$A_\delta$ is a club of $\delta$, and  
\item[(b)] $h_\delta:A_\delta\longrightarrow 2$.
\end{enumerate}
The forcing notion $\bbQ^*=\bbQ^*(\langle A_\delta,h_\delta:\delta\in
\cS^{\lambda^+}_\lambda\rangle)$ is defined as follows:

\noindent{\bf a condition in $\bbQ^*$}\quad is a tuple $p=(u^p,v^p,
\bar{e}^p,h^p)$ such that 
\begin{enumerate}
\item[(a)] $u^p\in [\lambda^+]^{<\lambda}$, $v^p\subseteq 
\cS^{\lambda^+}_\lambda\cap u^p$,   
\item[(b)] $\bar{e}^p=\langle e^p_\delta:\delta\in v^p\rangle$, where each
$e^p_\delta$ is a closed bounded non-empty subset of $A_\delta$, and
$e^p_\delta\subseteq u^p$, and
\item[(c)] if $\delta\in v^p$, then $\max(e^p_\delta)=\sup(u^p\cap\delta) >
  \sup(v^p\cap\delta)$, 
\item[(d)] $h^p:u^p\longrightarrow 2$ is such that for each $\delta\in
v^p$ we have that $h^p\restriction e_\delta\subseteq h_\delta$. 
\end{enumerate}
\noindent{\bf The order $\leq$ of $\bbQ^*$}\quad is such that $p\leq
q$ if and only if $u^p\subseteq u^q$, $h^p\subseteq h^q$, $v^p\subseteq
v^q$, and for each $\delta\in v^p$ the set $e^q_\delta$ is an end-extension
of $e^p_\delta$. 
\end{example}

Plainly, under the assumptions of \ref{exam1}, the following holds true. 

\begin{observation}
The forcing notion $\bbQ^*$ is $({<}\lambda)$--complete and
$|\bbQ^*|=\lambda^+$. It satisfies the $\lambda^+$--chain 
condition, so it is also $\lambda$--proper.
\end{observation}

If $\lambda=\lambda^{<\lambda}$ is not inaccessible, $2^{\lambda^+}=
\lambda^{++}$ and $2^\lambda= \lambda^+$, then some $\lambda$--support
iterations of forcing notions like $\bbQ^*$ are not $\lambda$--proper, as a
matter of fact this bad effect happens quite often. The problem comes from
the fact that if $\lambda$--support iterations of forcings of type $\bbQ^*$
were $\lambda$--proper, we could use Theorem \ref{cc} and a suitable
bookkeeping device to build a forcing notion which forces ``$\lambda= 
\lambda^{<\lambda}$ is not inaccessible and uniformization for
continuous ladder systems holds true''. However, this is not possible: 

\begin{theorem}
[Shelah {\cite{Sh:b}, \cite[Appendix, Theorem 3.6(2)]{Sh:f}}]
Assume $\theta<\lambda={\rm cf}(\lambda)$, $2^\theta=2^{<\lambda}=\lambda$.  
Furthermore suppose that for each $\delta\in \cS^{\lambda^+}_\lambda$ we
have a club $A_\delta$ of $\delta$. Then we can find  a sequence $\langle
d_\delta:\delta\in\cS^{\lambda^+}_\lambda\rangle$ of colorings such that 
\begin{itemize}
\item $d_\delta:A_\delta\longrightarrow 2$ and 
\item for any $h:\lambda^+\longrightarrow \{0,1\}$ for stationarily
many $\delta\in\cS^{\lambda^+}_\lambda$, the set $\{i\in A_\delta:
d_\delta(i)\neq h(i)\}$ is stationary in $A_\delta$.
\end{itemize}
\end{theorem}
\medskip

Many positive results concerning not collapsing cardinals in iterations with
uncountable supports are presented in literature already. For instance,
Kanamori \cite{Ka80} considered iterations of $\lambda$--Sacks forcing
notion (similar to the forcing $\bbQ^{2,\bar{E}}$; see Definition
\ref{defFilter} and Remark \ref{oldfor}) and he proved that under some
circumstances these iterations preserve $\lambda^+$. Several conditions
ensuring that $\lambda^+$ is not collapsed in $\lambda$--support iterations
were introduced in Shelah \cite{Sh:587,Sh:667}. Fusion properties of
iterations of tree--like forcing notions were used in Friedman and Zdomskyy
\cite{FrZd10} and Friedman, Honzik and Zdomskyy \cite{FrHoZd13}.

Numerous strong versions of $\lambda$--properness were studied by the
authors in a series of articles Ros{\l}anowski and Shelah \cite{RoSh:655,
  RoSh:860, RoSh:777, RoSh:888, RoSh:890, RoSh:942}. Each of those
conditions was meant to be applicable to some natural forcing notions adding
a new member of ${}^\lambda\lambda$ without adding new elements of
${}^{{<}\lambda} \lambda$. In some sense, they explained why the relevant
forcings can be iterated (without collapsing cardinals).

Also Eisworth \cite{Ei03} gave another (simpler) relative of the 
preservation theorem of \cite{RoSh:655}. 

The properties introduced in \cite{RoSh:655} and subsequent works we
typically applicable to either very bounding forcing notions or forcings
with trees with splittings on a stationary co-stationary set of levels.
While those properties were quite general, we had problems to include the
forcing notion $\bbQ^2_\lambda$ (see \ref{defFilter} below) in that
framework. In \cite{RoSh:942} we managed to formulate a suitable property
and show a relevant iteration theorem covering the forcing notions we were
interested in (including $\bbQ^2_\lambda$) in the case when $\lambda$ is
strongly inaccessible. Here we present a framework covering those forcings,
including our ``last forcing standing'', under a very mild demand on
$\lambda$: it admits diamonds\footnote{This should explain the title of this
  paper.}.  In a sense, we generalize (and correct, see Remark
\ref{correct655}) the property considered in \cite{RoSh:655}.  Our property,
{\em the purely sequential properness over semi diamonds}, is also close to
the fuzzy properness over quasi--diamonds of \cite[Definition
A.3.6]{RoSh:777}.

\medskip

\noindent {\bf Notation and Terminology.}\qquad Our notation is rather
standard and compatible with that of classical textbooks (like Jech
\cite{J}). However, in forcing we keep the convention that {\em a stronger
  condition is the larger one}.  

\begin{enumerate}
\item Ordinal numbers will be denoted be the lower case initial letters of
the Greek alphabet $\alpha,\beta,\gamma,\delta,\vare$ and $\zeta$, and also
by $\xi,i,j$ (with possible sub- and superscripts).  

Cardinal numbers will be called $\kappa,\lambda$; {\bf $\lambda$ will
  be always assumed to be a regular uncountable cardinal such that
  $\lambda^{<\lambda}=\lambda$}. 

Also, $\chi$ will denote a {\em sufficiently large\/} regular cardinal;
$\cH(\chi)$ is the family of all sets hereditarily of size less than
$\chi$. Moreover, we fix a well ordering $<^*_\chi$ of $\cH(\chi)$.

\item We will consider several games of two players. One player will be
  called {\em Generic\/} or {\em Complete\/}, and we will refer to this
  player as ``she''. Her opponent will be called {\em Antigeneric\/} or {\em
    Incomplete} and will be referred to as ``he''.

\item For a forcing notion $\bbP$, all $\bbP$--names for objects in
  the extension via $\bbP$ will be denoted with a tilde below (e.g.,
  $\name{\tau}$, $\name{X}$), and $\name{G}_\bbP$ will stand for the
  canonical $\bbP$--name for the generic filter in $\bbP$. The weakest
  element of $\bbP$ will be denoted by $\emptyset_\bbP$ (and we will always
  assume that there is one, and that there is no other condition  equivalent
  to it). 

  By ``$\lambda$--support iterations'' we mean iterations in which domains
  of conditions are of size $\leq\lambda$. However, we will pretend that
  conditions in a $\lambda$--support iteration $\bar{\bbQ}=\langle
  \bbP_\zeta, \name{\bbQ}_\zeta:\zeta< \zeta^* \rangle$ are total functions
  on $\zeta^*$ and for $p\in\lim(\bar{\bbQ})$ and $\alpha\in\zeta^*
  \setminus \dom(p)$ we will stipulate $p(\alpha)=
  \name{\emptyset}_{\name{\bbQ}_\alpha}$. 

\item A filter on $\lambda$ is a non-empty family of subsets of $\lambda$
  closed under supersets and intersections and not containing $\emptyset$. A
  filter is $({<}\kappa)$--complete if it is closed under intersections of
  ${<}\kappa$ members. (Note: we do allow principal filters or even
  $\{\lambda\}$.)

For a filter $D$ on $\lambda$, the family of all $D$--positive subsets
of $\lambda$ is called $D^+$. (So $A\in D^+$ if and only if $A\subseteq
\lambda$ and $A\cap B\neq\emptyset$ for all $B\in D$.) By a normal filter on
$\lambda$ we mean {\em proper uniform\/} filter closed under diagonal
intersections.  

\item By a {\em sequence\/} we mean a function whose domain is a set of
  ordinals. For two sequences $\eta,\nu$ we write $\nu\vtl\eta$ whenever
  $\nu$ is a proper initial segment of $\eta$, and $\nu \trianglelefteq\eta$
  when either $\nu\vtl\eta$ or $\nu=\eta$.  The length of a sequence $\eta$
  is the order type of its domain and it is denoted by $\lh(\eta)$.

The set of all sequences with domain $\alpha$ and with values in $A$ is
denoted by ${}^\alpha A$ and we set ${}^{<\alpha}
A=\bigcup\limits_{\beta<\alpha} {}^\beta A$. 

\item A {\em tree\/} is a $\vtl$--downward closed set of sequences. A {\em
    complete $\lambda$--tree\/} is a tree $T\subseteq {}^{<\lambda}\lambda$
  such that every $\vtl$-chain of size less than $\lambda$ has a
  $\vtl$-bound in $T$ and for each $\eta\in T$ there is $\nu\in T$ such that 
  $\eta\vtl\nu$.

Let $T\subseteq {}^{<\lambda}\lambda$ be a tree. For $\eta\in T$ we let 
\[\suc_T(\eta)=\{\alpha<\lambda:\eta\conc\langle\alpha\rangle\in T\}\quad
\mbox{ and }\quad (T)_\eta=\{\nu\in T:\nu\vtl\eta\mbox{ or }\eta
\trianglelefteq \nu\}.\]
We also let $\mrot(T)$ be the shortest $\eta\in T$ such that
$|\suc_T(\eta)|>1$ and $\lim_\lambda(T)=\{\eta\in{}^\lambda\lambda: (\forall
\alpha<\lambda)(\eta\rest\alpha\in T)\}$.
\end{enumerate}

\section{Preliminaries}

\subsection{Iterations of strategically complete forcing notions} 
\begin{definition}
\label{comp}
Let $\bbP$ be a forcing notion.
\begin{enumerate}
\item For an ordinal $\alpha$, let $\Game_0^\alpha(\bbP)$ be the following
  game of two players, {\em Complete} and  {\em Incomplete}:    
\medskip

\noindent the game lasts at most $\alpha$ moves and during a play the
players attempt to construct a sequence $\langle (p_i,q_i): i<\alpha\rangle$
of pairs of conditions from $\bbP$ in such a way that  
\[(\forall j<i<\alpha)(p_j\leq q_j\leq p_i)\]
and at the stage $i<\alpha$ of the game, first Incomplete chooses $p_i$ and
then Complete chooses $q_i$.   
\medskip

\noindent Complete wins if and only if for every $i<\alpha$ there are
legal moves for both players. 
\item A winning strategy $\st^0$ of Complete in $\Game_0^\alpha(\bbP)$ is
  {\em regular\/} if it instructs Complete to play $\emptyset_\bbP$ as long
  as Incomplete plays $\emptyset_\bbP$. 
\item The forcing notion $\bbP$ is {\em strategically
$({<}\lambda)$--complete\/} ({\em strategically
$({\leq}\lambda)$--complete,} respectively) if Complete has a winning
strategy in the game  $\Game_0^\lambda(\bbP)$ ($\Game_0^{\lambda+1}(\bbP)$,
respectively). 

Note that then Complete has also a regular winning strategy.
\item Let a model $N\prec (\cH(\chi),\in,<^*_\chi)$ be such that
${}^{<\lambda} N\subseteq N$, $|N|=\lambda$ and $\bbP\in N$. We say that a 
condition $p\in\bbP$ is {\em $(N,\bbP)$--generic in the standard sense\/}
(or just: {\em $(N,\bbP)$--generic\/}) if for every $\bbP$--name
$\name{\tau}\in N$ for an ordinal we have $p\forces$`` $\name{\tau}\in N$
''. 
\end{enumerate}
\end{definition}

{\em For the rest of this subsection we assume that:}
\begin{itemize}
\item $\bar{\bbQ}=\langle\bbP_\zeta, \name{\bbQ}_\zeta:\zeta<\zeta^*
  \rangle$ is a $\lambda$--support iteration of strategically
  $({<}\lambda)$--complete forcing notions, and 
\item for $\xi<\zeta^*$, $\name{\st}^0_\xi$ is the $<^*_\chi$--first 
$\bbP_\xi$--name for a regular winning strategy of Complete in
$\Game^\lambda_0(\name{\bbQ}_\xi)$, and 
\item a model $N\prec  ({\mathcal H}(\chi),{\in},{<^*_\chi})$ is such that
  $|N|=\lambda$, ${}^{<\lambda}N\subseteq N$ and $\bar{\bbQ},\ldots\in N$. 
\end{itemize}

\begin{observation}
$\bbP_{\zeta^*}$ is strategically $({<}\lambda)$--complete. 
\end{observation}

\begin{observation}
\label{tostart}
Suppose that $\zeta\in(\zeta^*+1)\cap N$, $\langle q_\alpha:\alpha<\delta
\rangle\subseteq N\cap \bbP_\zeta$, $\delta<\lambda$, $r\in\bbP_\zeta$ is 
$(N,\bbP_\zeta)$--generic and $q_\beta\leq r$ for all $\beta<\delta$. Then
there are conditions $q\in N\cap \bbP_\zeta$ and $r^+\in\bbP_\zeta$ such
that $q\leq r^+$, $r\leq r^+$ and $(\forall\beta<\delta)(q_\beta\leq q)$. 
\end{observation}

\begin{lemma}
\label{655.3.3}
Suppose that $\zeta\in(\zeta^*+1)\cap N$ is a limit ordinal of cofinality
$\cf(\zeta)<\lambda$ and $r\in\bbP_\zeta$ is such that 
\[(\forall\vare\in\zeta\cap N)\big(r\rest\vare\mbox{ is
$(N,\bbP_\vare)$--generic}\big).\]
Assume that conditions $q_\alpha\in N\cap\bbP_\zeta$ (for $\alpha<\delta$,
$\delta<\lambda$) satisfy
\begin{enumerate}
\item[(a)] if $\alpha<\beta<\delta$, then $q_\alpha\leq q_\beta\leq r$, and 
\item[(b)] if $\vare\in N\cap \zeta$, $s\in \bbP_\vare$ and $(\forall
  \alpha<\delta)(q_\alpha\rest \vare\leq s)$, then
\[s\forces_{\bbP_\vare}\mbox{`` the sequence }\langle q_\alpha(\vare):
\alpha< \delta\rangle\mbox{ has an upper bound in $\name{\bbQ}_\vare$ ''.}\] 
\end{enumerate}
Then there are conditions $q\in N\cap\bbP_\zeta$ and $r^+\in\bbP_\zeta$ such
that $q\leq r^+$, $r\leq r^+$ and $(\forall\beta<\delta)(q_\beta\leq q)$. 
\end{lemma}

\begin{proof}
The arguments here are similar to \cite[Proposition 3.3]{RoSh:655}, except
that our iterands are {\em strategically\/} $({<}\lambda)$--complete only
(and not necessarily $({<}\lambda)$--complete). This introduces some
additional complexity.

Let $\langle i_\gamma:\gamma<\cf(\zeta)\rangle\subseteq N\cap \zeta$ be a
strictly increasing continuous sequence cofinal in $\zeta$ with $i_0=0$. By
induction on $\gamma<\cf(\zeta)$ we choose $r^-_\gamma,r_\gamma$ and
$r^*_\gamma$ such that
\begin{enumerate}
\item[(i)] $r^-_\gamma\in\bbP_{i_\gamma}\cap N$ is above (in
  $\bbP_{i_\gamma}$) of all $q_\beta\rest i_\gamma$ for $\beta<\delta$; 
\item[(ii)] $r_\gamma,r_\gamma^*\in\bbP_{i_\gamma}$, $r^-_\gamma\leq
  r_\gamma\leq r^*_\gamma$ and $r\rest i_\gamma\leq r_\gamma$;
\item[(iii)] if $\gamma<\vare<\cf(\zeta)$, then $r^-_\gamma\leq
  r^-_\vare$, $r^*_\gamma\leq r_\vare$, and 
\item[(iv)] for each $\vare<\zeta$ and $\gamma<\cf(\zeta)$ we have
\[\begin{array}{ll}
r^*_\gamma\rest \vare\forces&\mbox{`` }\langle r_\alpha(\vare),
r^*_\alpha(\vare): \alpha\leq\gamma\rangle\mbox{ is a legal partial play of
} \Game^\lambda_0(\name{\bbQ}_\vare)\\
&\mbox{ in which Complete uses her regular winning strategy }
\name{\st}^0_\vare\mbox{ ''.}
\end{array}\]
\end{enumerate}
So, at stage $\gamma+1$ of the construction we apply Observation
\ref{tostart} to the sequence $\langle r_\gamma^-\conc q_\beta\rest
[i_\gamma, i_{\gamma+1}): \beta<\delta\rangle\subseteq N\cap
\bbP_{i_{\gamma+1}}$ and the $(N,\bbP_{i_{\gamma+1}})$--generic condition
$r^*_\gamma\conc r\rest [i_\gamma,i_{\gamma+1})$. This will give us
conditions $r_{\gamma+1}^-\in N\cap \bbP_{i_{\gamma+1}}$ and
$r_{\gamma+1}\in \bbP_{i_{\gamma+1}}$ so that the demands (i)--(iii) are
satisfied. Then $r^*_{\gamma+1}\in \bbP_{i_{\gamma+1}}$ is chosen according
to demand (iv), remembering that $\name{\st}^0_\vare$ are
$\bbP_\vare$--names for {\em regular\/} winning strategies. Next, at a limit
stage $\gamma$ is limit then we first pick $r'\in\bbP_{i_\gamma}$ that is
stronger than all $r^*_\alpha$ for $\alpha<\gamma$ (exists by (iv)). This
$r'$ is $(N,\bbP_{i_\gamma})$--generic, stronger than all $r^-_\alpha$ for 
$\alpha<\gamma$ and stronger than all $q_\beta\rest i_\gamma$ for
$\beta<\delta$. So we may use Observation \ref{tostart} to find $r^-_\gamma$
and $r_\gamma\geq r'$ as needed for (i)--(iii) and then pick $r^*_\gamma$ to
fullfil (iv).
\medskip

Let $r^+\in\bbP_\zeta$ be an upper bound of $\langle r_\gamma:\gamma
<\cf(\zeta) \rangle$ (remember clause (iv) above); then also $r\leq
r^+$. Now we are going to define a condition $q\in \bbP_\zeta\cap N$. We let 
\[\dom(q)=\bigcup\{\dom(r^-_{\gamma+1})\cap [i_\gamma, i_{\gamma+1}):
\gamma< \cf(\zeta)\},\]
and for $\vare\in \dom(q)$, $i_\gamma\leq \vare< i_{\gamma+1}$, we let
$q(\vare)$ be the $<^*_\chi$--first $\bbP_\vare$--name for the following
object in $\bV[G_{\bbP_\vare}]$. 
\begin{enumerate}
\item[(A)] If $r^-_{\gamma+1}(\vare)[G_{\bbP_\vare}]$ is an upper bound to
  $\{q_\beta(\vare)[G_{\bbP_\vare}]: \beta<\delta\}$ in
  $\name{\bbQ}_\vare[G_{\bbP_\vare}]$, then $q(\vare)[G_{\bbP_\vare}]=
  r^-_{\gamma+1}(\vare)[G_{\bbP_\vare}]$. 
\item[(B)] If not (A) but  $\{q_\beta(\vare)[G_{\bbP_\vare}]:
  \beta<\delta\}$ has an upper bound in $\name{\bbQ}_\vare[G_{\bbP_\vare}]$,
  then  $q(\vare)[G_{\bbP_\vare}]$ is such a bound. 
\item[(C)] If neither (A) nor (B), then $q(\vare)[G_{\bbP_\vare}]=
  q_0(\vare) [G_{\bbP_\vare}]$. 
\end{enumerate}
It should be clear that $q\in \bbP_\zeta\cap N$. Now,
\begin{itemize}
\item $q\leq r^+$.
\end{itemize}
Why? By induction on $\vare\in \zeta\cap N$ we show that $q\rest \vare\leq
r^+\rest \vare$. Steps ``$\vare=0$'' and ``$\vare$ is limit'' are clear, so
suppose that we have proved $q\rest \vare\leq r^+\rest \vare$, $i_\gamma\leq
\vare <i_{\gamma+1}$ (and we are interested in restrictions to
$\vare+1$). Assume that $G_{\bbP_\vare}\subseteq \bbP_\vare$ is a generic
filter over $\bV$ such that $r^+\rest\vare \in G_{\bbP_\vare}$. Since
$q_\beta\rest i_{\gamma+1}\leq r^-_{\gamma+1}\leq r_{\gamma+1}\leq r^+$, we
also have $q_\beta\rest \vare\in G_{\bbP_\vare}$ (for $\beta<\delta$) and
$r^-_{\gamma+1} \rest \vare\in G_{\bbP_\vare}$. Hence
$r^-_{\gamma+1}(\vare)[G_{\bbP_\vare}]$ is an upper bound of
$\{q_\beta(\vare) [G_{\bbP_\vare}]:\beta<\delta\}$. Therefore 
\[q(\vare)[G_{\bbP_\vare}]=r^-_{\gamma+1}(\vare)[G_{\bbP_\vare}] \leq
r_{\gamma+1}(\vare)[G_{\bbP_\vare}] \leq r^+(\vare)[G_{\bbP_\vare}]\]
(see (A) above) and we are done.

The proof of the Lemma will be finished once we show that
\begin{itemize}
\item $(\forall \beta<\delta)(q_\beta\leq q)$.
\end{itemize}
Why does this hold? By induction on $\vare\in \zeta\cap N$ we show that
$q_\beta\rest \vare \leq q\rest \vare$ for all $\beta<\delta$. Steps
``$\vare=0$'' and ``$\vare$ is limit'' are as usual clear, so suppose that
we have proved $q_\beta\rest\vare \leq q\rest\vare$ for $\beta<\delta$,
$i_\gamma\leq \vare <i_{\gamma+1}$ (and we are interested in the
restrictions to $\vare+1$). Assume that $G_{\bbP_\vare}\subseteq \bbP_\vare$
is a generic filter over $\bV$ such that $q\rest \vare\in
G_{\bbP_\vare}$. Then, by the inductive hypothesis and the assumption (b) of 
the Lemma, we know that the sequence $\langle q_\beta(\vare)
[G_{\bbP_\vare}]: \beta<\delta \rangle$ has an upper bound in
$\name{\bbQ}_\vare[G_{\bbP_\vare}]$. Therefore, by (A)+(B),
$q(\vare)[G_{\bbP_\vare}] \geq q_\beta(\vare)[G_{\bbP_\vare}]$ for all
$\beta<\delta$, and we are done.  
\end{proof}

\begin{lemma}
\label{emplystrat}
Suppose that $\zeta\in (\zeta^*+1)\cap N$ and conditions $r\in
\bbP_\zeta$ and $p_\alpha,q_\alpha\in \bbP_\zeta\cap N$ (for
$\alpha<\delta$, $\delta<\lambda$) satisfy: 
\begin{enumerate}
\item[(i)] $p_\alpha\leq q_\alpha\leq p_\beta\leq q_\beta\leq r$ for all
  $\alpha<\beta<\delta$,  
\item[(ii)] for each $\xi\in \zeta\cap N$ and $\beta<\delta$ we have 
\[\begin{array}{ll}
q_\beta\rest \xi\forces_{\bbP_\xi}&\mbox{`` }\langle p_\alpha(\xi),
q_\alpha(\xi): \alpha\leq\beta\rangle\mbox{ is a legal partial play of }
\Game^\lambda_0(\name{\bbQ}_\xi)\\
&\mbox{ in which Complete uses her regular winning strategy
  $\name{\st}^0_\xi$ '',}
\end{array}\]
\item[(iii)] either 
\begin{enumerate}
\item[$(\alpha)$] $r$ is $(N,\bbP_\zeta)$--generic,
\end{enumerate}
or
\begin{enumerate}
\item[$(\beta)$] $\zeta$ is a limit ordinal of cofinality $\cf(\zeta)<
  \lambda$ and for each $\xi\in \zeta\cap N$ the condition $r\rest \xi$ is
  $(N,\bbP_\xi)$--generic. 
\end{enumerate}
\end{enumerate}
Let $\cI\in N$ be an open dense subset of $\bbP_\zeta$ and $p^-\in N\cap
\bbP_\zeta$, $p^-\leq r$. Then there are conditions $p',q'\in
\bbP_\zeta\cap N$ and $r'\in \bbP_\zeta$ such that 
\begin{enumerate}
\item[(a)] $r\leq r'$ and $q_\alpha\leq p'\leq q'\leq r'$ for all
  $\alpha<\delta$, and 
\item[(b)] for each $\xi\in \zeta\cap N$ we have 
\[\begin{array}{ll}
q'\rest \xi\forces_{\bbP_\xi}&\mbox{`` }\langle p_\alpha(\xi),
q_\alpha(\xi): \alpha<\delta\rangle\conc\langle p'(\xi),q'(\xi)\rangle
\mbox{ is a partial play of } \Game^\lambda_0(\name{\bbQ}_\xi)\\
&\mbox{ in which Complete uses her regular winning strategy
  $\name{\st}^0_\xi$ '',}
\end{array}\]
\item[(c)] if we are under the assumption of {\rm (iii)}$(\alpha)$, then
  also $p'\in \cI$,
\item[(d)] if we are under the assumption of {\rm (iii)}$(\beta)$ and 
  $q_\alpha\leq p^-$ for all $\alpha<\delta$, then we still may require that
  $p^- \leq p'$. 
\end{enumerate}
\end{lemma}

\begin{proof}
If we are under the assumption of (iii)$(\alpha)$, then the conclusion of
the Lemma should be clear (including clause (c)).

So suppose that we are in the case described in (iii)$(\beta)$ and let
$\langle i_\gamma:\gamma<\cf(\zeta)\rangle \subseteq N\cap \zeta$ be an
increasing continuous sequence cofinal in $\zeta$, $i_0=0$. First, let us
assume that the condition $p^-\in N\cap \bbP_\zeta$ satisfies $q_\alpha\leq
p^-\leq r$ for all $\alpha<\delta$. By induction on $\gamma<\cf(\zeta)$ we
will pick conditions $p^\gamma,q^\gamma,r^\gamma,s^\gamma$ so that:
\begin{enumerate}
\item[$(*)_1$] $p^\gamma,q^\gamma\in N\cap \bbP_{i_\gamma}$, $r^\gamma,
  s^\gamma\in \bbP_{i_\gamma}$, 
\item[$(*)_2$] $p^-\rest i_\gamma\leq p^\gamma\leq q^\gamma$, $r\rest
  i_\gamma\leq r^\gamma\leq s^\gamma$ and $q^\gamma\leq r^\gamma$, 
\item[$(*)_3$] $q^{\gamma'}\leq p^\gamma\rest i_{\gamma'}$ and $s^{\gamma'}
  \leq r_\gamma\rest i_{\gamma'}$ for $\gamma' <\gamma$, 
\item[$(*)_4$] if $i_\gamma\leq\xi<i_{\gamma+1}$ and $\gamma<\gamma^* <
  \cf(\zeta)$, then  
\[\begin{array}{l}
q^{\gamma^*}\rest \xi\forces_{\bbP_\xi}\mbox{`` the sequence }\langle
p_\alpha(\xi),q_\alpha(\xi)\!: \alpha<\delta\rangle\conc\langle
p^{\gamma'}(\xi), q^{\gamma'}(\xi)\!:\gamma<\gamma'\leq\gamma^*\rangle\\ 
\qquad\qquad\mbox{ is a legal partial play of }
\Game^\lambda_0(\name{\bbQ}_\xi) \mbox{  in which }\\ 
\qquad\qquad\mbox{ Complete uses her regular winning strategy
  $\name{\st}^0_\xi$ '',}
\end{array}\]
and
\[\begin{array}{ll}
s^{\gamma^*}\rest \xi\forces_{\bbP_\xi}&\mbox{`` }\langle r^{\gamma'}(\xi), 
s^{\gamma'}(\xi):\gamma'\leq\gamma^*\rangle\mbox{ is a legal partial play 
  of } \Game^\lambda_0(\name{\bbQ}_\xi)\\  
&\mbox{ in which Complete uses her regular winning strategy
  $\name{\st}^0_\xi$ ''.}
\end{array}\]
\end{enumerate}
When we arrive to a stage $\gamma<\cf(\zeta)$ of the construction, after
having determined $p^{\gamma'},q^{\gamma'},r^{\gamma'},s^{\gamma'}$ for
$\gamma'<\gamma$, we first pick a condition $r^*\in\bbP_{i_\gamma}$ stronger
than all $s^{\gamma'}$ for $\gamma'<\gamma$ and such that $r\rest i_\gamma
\leq r^*$. [There is such $r^*$ by the second part of the demand $(*)_4$ at
previous stages.] Then the condition $r^*$ is $(N,\bbP_{i_\gamma})$--generic
so we may use arguments as for the first part of the Lemma (case
(iii)$(\alpha)$) to pick $p^\gamma,q^\gamma$ and $s^\gamma\geq r^\gamma\geq
r^*$ so that demands in $(*)_1$--$(*)_4$ are satisfied. 

After the construction is carried out we define $p'\in\bbP_\zeta$ so
that $\dom(p')=\bigcup\limits_{\gamma<\cf(\zeta)} \dom(q^\gamma)$ and if
$\xi\in \dom(p')$, $i_\gamma\leq\xi<i_{\gamma+1}$, $\gamma<\cf(\zeta)$, 
then $p'(\xi)$ is the $<^*_\chi$--first $\bbP_\xi$--name such that the
condition $p'\rest\xi$ forces (in $\bbP_\xi$) that 
\[\mbox{`` if }p^-(\xi)\leq p^{\gamma+1}(\xi)
\mbox{ then } p'(\xi)=p^{\gamma+1}(\xi),\mbox{ otherwise }p'(\xi)= p^-(\xi)
\mbox{ ''.}\] 
Then $p'\in N$ is a condition stronger than $p^-$ (and so also stronger than
all $q_\alpha$ for $\alpha<\delta$). Let $q'\in \bbP_\zeta\cap N$ be such
that $\dom(q')=\dom(p')$ and for each $\xi\in \dom(q')$ we have  
\[\begin{array}{ll}
q'\rest \xi\forces_{\bbP_\xi}&\mbox{`` }\langle p_\alpha(\xi),
q_\alpha(\xi): \alpha<\delta\rangle\conc\langle p'(\xi),q'(\xi)\rangle
\mbox{ is a partial play of } \Game^\lambda_0(\name{\bbQ}_\xi)\\
&\mbox{ in which Complete uses her regular winning strategy
  $\name{\st}^0_\xi$ ''.}
\end{array}\]
Plainly, $p'\leq q'$. Let $r'\in\bbP_\zeta$ be an upper bound to
$\langle s^\gamma:\gamma<\cf(\zeta)\rangle$ (exists by the second part of
the demands in $(*)_4$). Then $r'\geq r$. Now we show inductively that
$r'\rest \xi\geq q'\rest \xi$ for $\xi<\zeta$. So suppose $\xi\in\dom(q')$, 
$i_\gamma\leq \xi< i_{\gamma+1}$, and $q'\rest\xi \leq r'\rest \xi$. Then
$r'\rest\xi\forces p'(\xi)=p^{\gamma+1}(\xi)$ (by $(*)_2$) and so also
$r'\rest\xi \forces q'(\xi)=q^{\gamma+1}(\xi)\leq s^{\gamma+1}(\xi)\leq
r'(\xi)$. Hence $p^-\leq p'\leq q'\leq r'$. 

If the condition $p^-$ does not satisfy the assumptions of clause (d), then
we may use Lemma \ref{655.3.3} to find $q\in N\cap \bbP_\zeta$ and
$r^+\in\bbP_\zeta$ such that $r\leq r^+$, $q\leq r^+$ and $q_\alpha\leq q$
for all $\alpha<\delta$. Now carry out the above arguments for $q, r^+$ in
place of $p^-, r$.
\end{proof}

\begin{definition}
\label{RSconditions}
\begin{enumerate}
\item An {\em RS--condition\footnote{the RS stands for ``revised support''} 
    in $\bbP_{\zeta^*}$} is a pair $(p,w)$ such that $w\in
  [(\zeta^*+1)]^{<\lambda}$ is a closed set, $0,\zeta^*\in w$, $p$ is a
  function with domain $\dom(p)\subseteq \zeta^*$, and 
\begin{enumerate}
\item[$(\otimes)$] for every two successive members $\vare'<
\vare''$ of the set $w$, $p\rest [\vare',\vare'')$
is a $\bbP_{\vare'}$--name of an element of $\bbP_{\vare''}$
whose support is included in the interval $[\vare',\vare'')$.
\end{enumerate}
The family of all RS--conditions in $\bbP_{\zeta^*}$ is denoted by
$\bbP_{\zeta^*}^{\rm RS}$.
\item If $(p,w)\in\bbP_{\zeta^*}^{\rm RS}$ and $G\subseteq
\bbP_{\zeta^*}$ is a generic filter over $\bV$, then we define 
\[(p,w)^G=\bigcup\big\{(p\rest [\vare',\vare''))[G\cap \bbP_{\vare'}]:
\vare'< \vare''\mbox{ are successive members of }w\big\}\] 
Note that $(p,w)^G\in \bbP_{\zeta^*}$. Also, we write $(p,w)\in' G$ whenever
$(p,w)^G\in G$.  
\item If $(p_1,w_1), (p_2,w_2)\in \bbP_{\zeta^*}^{\rm RS}$, then we write
$(p_1,w_1)\leq'(p_2,w_2)$ whenever 
\begin{enumerate}
\item[$(\oplus)$] for every generic $G\subseteq\bbP_{\zeta^*}$ over
  $\bV$,

if $(p_2,w_2)\in' G$ then $(p_1,w_1)^G \leq_{\bbP_{\zeta^*}} (p_2, w_2)^G$. 
\end{enumerate}
\end{enumerate}
\end{definition}

\begin{remark}
\label{RSremark}
In \ref{RSconditions}(1)$(\otimes)$, $p\rest [\vare',\vare'')$ is a sequence
of $\bbP_{\vare'}$--names, so it is a $\bbP_{\vare'}$--name for some
sequence. This (resulting) sequence is required to belong to
$\bbP_{\vare''}$, and just also to $\bV$. We may therefore think about an
RS--condition $(p,w)$ as follows. First, for every $\vare\in w\cap \zeta^*$
we have  
\begin{itemize}
\item a maximal antichain $\{p_i^\vare: i\in I_\vare\}$ of $\bbP_\vare$, and 
\item a function $f^\vare:I_\vare\longrightarrow \bbP_{\zeta^*}$ such that
  $f^\vare(i)$ is a condition with support included in the interval $\big[
  \vare,\min(w\setminus (\vare+1))\big)$. 
\end{itemize}
(They describe a $\bbP_\vare$--name for a condition in $\bbP_{\zeta^*}$ with
support in the right interval.) Next, if $\vare\leq\xi<\vare'$
where $\vare,\vare'$ are successive members of $w$, then we may fix a
$\bbP_\xi$--name $p(\xi)$ for a member of $\name{\bbQ}_\xi$ such that
$p^\vare_i\forces_{\bbP_\xi} p(\xi)=f^\vare(i)(\xi)$. Note however that if
we apply this approach to each $\xi$, we may not end up with a condition in
$\bbP_{\zeta^*}$ because of the support! Nevertheless we will think of $p$
as a function on $\zeta^*$ where each $p(\xi)$ is a $\bbP_\xi$--name for a
member of $\name{\bbQ}_\xi$.  
\end{remark}

\begin{lemma}
\label{RSbounds}
\begin{enumerate}
\item If $(p,w)\in\bbP_{\zeta^*}^{\rm RS}$ and $q\in\bbP_{\zeta^*}$, then
there is $q^*\in\bbP_{\zeta^*}$ stronger than $q$ and such that for each
successive members $\vare'<\vare''$ of $w$ the condition $q^*\rest \vare'$
decides $p\rest [\vare',\vare'')$ (i.e., $q^*\rest\vare'\forces$`` $p\rest
[\vare', \vare'')=p_{\vare',\vare''}$'' for some $p_{\vare', \vare''}\in
\bbP_{\zeta^*}$).  
\item For each $(p,w)\in\bbP_{\zeta^*}^{\rm RS}$ there is $q\in 
\bbP_{\zeta^*}$ such that $(p,w)\leq'(q,\{0, \zeta^*\})$. 
\item Let $(p_i,w_i)\in \bbP_{\zeta^*}^{\rm RS}\cap N$ (for $i<\delta<
\lambda$), and $q\in\bbP_{\zeta^*}\cap N$, $r\in\bbP_{\zeta^*}$ be such that 
\[q\leq r\quad\mbox{ and }\quad(\forall i<j<\delta)((p_i,w_i)\leq' (p_j,w_j)
\leq' (r,\{0,\zeta^*\})).\]
Assume that either $r$ is $(N,\bbP_{\zeta^*})$--generic, or $\zeta^*$ is a
limit ordinal of cofinality $\cf(\zeta^*)<\lambda$ and for every $\zeta\in 
\zeta^*\cap N$ the condition $r\rest\zeta$ is $(N,\bbP_\zeta)$--generic.

\noindent Then there are conditions $q'\in N\cap\bbP_{\zeta^*}$
and $r'\in\bbP_{\zeta^*}$ such that $q\leq q'\leq r'$, $r\leq r'$ and
$(\forall i<\delta)((p_i,w_i)\leq' (q',\{0,\zeta^*\}))$.  
\end{enumerate}
\end{lemma}

\begin{proof} 
It is a slight generalization of \cite[Proposition 3.6]{RoSh:655}. 
\medskip

\noindent (1), (2)\qquad Straightforward (as $\bbP_{\zeta^*}$ is
strategically $({<}\lambda)$--complete).   
\medskip

\noindent (3)\quad If $r$ is $(N,\bbP_{\zeta^*})$--generic, then the
conclusion is immediate. So let us consider the case when $\aleph_0 \leq
\cf(\zeta^*)<\lambda$ and $r\rest \zeta$ is $(N,\bbP_\zeta)$--generic for
each $\zeta\in \zeta^*\cap N$. 

Let $\langle i_\gamma:\gamma<\cf(\zeta^*)\rangle\subseteq N\cap\zeta^*$ be a 
strictly increasing continuous sequence cofinal in $\zeta^*$, $i_0=0$. For
$\gamma<\cf(\zeta^*)$ and  $i<\delta$ we put $p_i^\gamma=p_i\rest
i_\gamma$ and $w^\gamma_i=(w_i\cap i_\gamma)\cup\{i_\gamma\}$ (clearly
$(p^\gamma_i,w^\gamma_i)\in \bbP_{i_\gamma}^{\rm RS}$). 

Now, we inductively pick conditions $q_\gamma,q_\gamma^+,r_\gamma,
r^+_\gamma$ (for $\gamma<\cf(\zeta^*)$) so that  
\begin{enumerate}
\item[(i)] $q_\gamma,q^+_\gamma\in\bbP_{i_\gamma}\cap N$,
  $r_\gamma,r^+_\gamma\in\bbP_{i_\gamma}$, 
\item[(ii)] $q\rest i_\gamma\leq q_\gamma\leq q^+_\gamma\leq r_\gamma$,
  $r\rest i_\gamma\leq r_\gamma\leq r_\gamma^+$, and for each $\beta<\gamma$ we have  
\[q_\beta^+\leq q_\gamma\mbox{ and }r_\beta^+\leq r_\gamma,\]
\item[(iii)] for every $\xi<i_\gamma$ we have 
\[\begin{array}{ll}
q^+_\gamma\rest \xi\forces_{\bbP_\xi}&\mbox{`` }\langle q_\beta(\xi),
q_\beta^+(\xi): \beta\leq\gamma\rangle\mbox{ is a legal partial play of } 
\Game^\lambda_0(\name{\bbQ}_\xi)\\ 
&\mbox{ in which Complete uses her regular winning strategy
  $\name{\st}^0_\xi$ '',}  
\end{array}\]
 and
\[\begin{array}{ll}
r^+_\gamma\rest \xi\forces_{\bbP_\xi}&\mbox{`` }\langle r_\beta(\xi),
r^+_\beta(\xi): \beta\leq\gamma\rangle\mbox{ is a legal partial play of }
\Game^\lambda_0(\name{\bbQ}_\xi)\\ 
&\mbox{ in which Complete uses the strategy $\name{\st}^0_\xi$ '',}
\end{array}\]
\item[(iv)] if $i<\delta$ and $\vare'<\vare''$ are two successive members of
  $w_i$ with $\vare'<i_\gamma$, then the condition $q_\gamma\rest\vare'$
  decides $p_i\rest [\vare', \vare'')$ and $(p^\gamma_i,w^\gamma_i)\leq'
  (q_\gamma, \{0,i_\gamma\})$.
\end{enumerate}
At a limit stage $\gamma$ of the construction, we first pick an upper bound
$r^*\in \bbP_{i_\gamma}$ to $\langle r_\beta^+:\beta<\gamma\rangle$
(exists by clause (iii) at previous stages). By (ii) we know that $r\rest
i_\gamma\leq r^*$ so $r^*$ is $(N,\bbP_{i_\gamma})$--generic. Consequently,
we may use Lemma \ref{emplystrat} to choose $q_\gamma,q_\gamma^+\in
\bbP_{i_\gamma}\cap N$ and $r_\gamma\in\bbP_{i_\gamma}$ satisfying the
relevant demands in (ii)+(iii), and then we pick $r^+_\gamma\geq r_\gamma$
so that the second part of (iii) holds. Note that the demand in (iv) will
then follow by the inductive hypothesis.

Suppose now we are at a successor stage $\gamma+1$, so we assume
$q_\gamma,q^+_\gamma,r_\gamma,r^+_\gamma$ have been already chosen. We note
that a condition $r^*=r^+_\gamma\conc (r\rest [i_\gamma,i_{\gamma+1}))
\in\bbP_{i_{\gamma+1}}$ is stronger than $r$ so it is $(N,
\bbP_{i_{\gamma+1}})$--generic. Therefore we may apply Lemma
\ref{emplystrat} to $r^*$, the sequence $\langle q_\beta,q_\beta^+: \beta\leq 
\gamma\rangle$ considered as a sequence of conditions in
$\bbP_{i_{\gamma+1}}$ and a suitable open dense subset $\cI$ of
$\bbP_{i_{\gamma+1}}$ from $N$ to choose $q_{\gamma+1},q_{\gamma+1}^+$ and
$r_{\gamma+1}, r^+_{\gamma+1}$ so that the demands in (i)---(iv) are
satisfied. (For clause (iv) remember that by our assumptions we have
$(p^{\gamma+1}_i,w^{\gamma+1}_i)\leq' (r^*, \{0,i_{\gamma+1}\})$.)
 
Let $r^*\in \bbP_{\zeta^*}$ be stronger than all $r_\gamma$'s. Apply Lemma
\ref{655.3.3} to the sequence $\langle
q^+_\gamma:\gamma<\cf(\zeta^*)\rangle$ and the condition $r^*$ to choose
$q'\in N\cap \bbP_{\zeta^*}$ and $r'\in \bbP_{\zeta^*}$ such that 
\[q'\leq r',\quad r^*\leq r',\quad \mbox{ and }\quad (\forall\gamma<
\cf(\zeta^*))(q^+_\gamma\leq q').\] 
Then $r',q'$ are as required. 
\end{proof}

\begin{proposition}
\label{X.7}  
Assume that
\begin{enumerate}
\item $(p,w)\in \bbP^{\rm RS}_{\zeta^*}$, $\zeta\in w$ and $\sup(w\cap
    \zeta)<\xi<\zeta$, 
\item $w'\subseteq \zeta^*+1$ is a closed set such that $|w'|<\lambda$ and
  $w\cup\{\xi\}\subseteq w'$,
\item $\name{\tau}$ is a $\bbP_\zeta$--name for an ordinal.
\end{enumerate}
Then there are $p'$ and $\name{\tau}'$ such that 
\begin{enumerate}
\item[(a)] $(p',w')\in \bbP^{\rm RS}_{\zeta^*}$ and $(p,w)\leq' (p',w')$,
\item[(b)] $p'\rest\xi=p\rest\xi$ and $p'\rest [\zeta,\zeta^*)= p\rest
  [\zeta, \zeta^*)$,
\item[(c)] $\name{\tau}'$ is a $\bbP_\xi$--name, and 
\item[(d)] if $G\subseteq \bbP_{\zeta^*}$ is generic over $\bV$ and
  $(p',w')\in' G$, then $\name{\tau}'[G]=\name{\tau}[G]$. 
\end{enumerate}
\end{proposition}

\begin{proof}
  First we declare that $p'\rest\xi=p\rest\xi$ and $p'\rest [\zeta,\zeta^*)
  = p\rest [\zeta,\zeta^*)$. Now, $p\rest [\xi,\zeta)$ is a
  $\bbP_\xi$--name, so we may choose a maximal antichain
  $\cB\subseteq\bbP_\xi$ and a function $g:\cB\longrightarrow \bbP_\zeta$
  such that for each $r\in\cB$ we have
  \begin{enumerate}
\item[$(\alpha)$]  $\dom(g(r))\subseteq [\xi,\zeta)$,
\item[$(\beta)$] $r$ decides the value of $p\rest [\xi,\zeta)$ and
  $r\forces_{\bbP_\xi} p\rest [\xi,\zeta)\leq g(r)$, 
\item[$(\gamma)$] $r\conc g(r)$ forces a value to $\name{\tau}$, say $r\conc
  g(r)\forces \name{\tau}=\tau(r)$.
  \end{enumerate}
Let $\name{\tau}'$ be a $\bbP_\xi$--name such that $r\forces_{\bbP_\xi}
\name{\tau}'=\tau(r)$ for each $r\in\cB$, and let $p'\rest [\xi,\zeta)$ be a
$\bbP_\xi$--name such that $r\forces p'\rest [\xi,\zeta)=g(r)$ for $r\in
\cB$. Then $(p',w')$ and $\name{\tau}'$ are as required in (a)--(d). 
\end{proof}

It will be convenient for the proof of our main result \ref{n3.7} to look at
RS--conditions also in a slightly different way.

\begin{definition}
\label{X.1}
Let $(p,w)\in \bbP^{\rm RS}_{\zeta^*}$. {\em A standard representation of
  $(p,w)$} is a triple $(\cA,f,w)$ such that 
\begin{enumerate}
\item[(i)] $\cA\subseteq \bbP_{\zeta^*}$ is a maximal antichain,
\item[(ii)] $f:\cA\longrightarrow\bbP_{\zeta^*}$,
\item[(iii)] for each $r\in\cA$, either $f(r)\leq r$ or for some $\vare \in
  \dom(r)$ we have $f(r)\rest \vare\leq r\rest \vare$ and
  $r\rest\vare\forces_{\bbP_\vare}$`` $f(r)(\vare),r(\vare)$ are
  incompatible in $\name{\bbQ}_\vare$ '',
\item[(iv)] if $r\in\cA$, $\vare\in w\cap\zeta^*$ and
  $\vare'=\min(w\setminus (\vare+1))$, 

then $r\rest\vare\forces_{\bbP_\vare}$`` $p\rest\vare'=f(r)\rest \vare'$ ''.  
\end{enumerate}
\end{definition}

\begin{observation}
\label{X.2}
For each $(p,w)\in \bbP^{\rm RS}_{\zeta^*}$ and an open dense set
$\cI\subseteq \bbP_{\zeta^*}$ there is a standard representation $(\cA,f,w)$
of $(p,w)$ such that $\cA\subseteq \cI$.   
\end{observation}

\begin{observation}
\label{X.2A}
Assume that $(p,w)\in \bbP^{\rm RS}_{\zeta^*}$ and $(\cA,f,w)$ is a standard
representation of $(p,w)$.
\begin{enumerate}
\item[(v)] If $s,r\in\cA$, $\vare\in w\cap\zeta^*$, $\vare'=\min(w\setminus
  (\vare+1))$ and the conditions $s\rest \vare$ and $r\rest\vare$ are
  compatible in $\bbP_\vare$, then $f(s)\rest \vare'= f(r)\rest \vare'$. 
\item[(vi)] If $r\in\cA$, $\vare<\zeta^*$ and $f(r)\rest \vare\leq r\rest
  \vare$, then there is $s\in\cA$ such that $r\rest\vare$ and $s\rest \vare$
  are compatible and $f(s)\leq s$. 
\end{enumerate}
\end{observation}

\begin{proposition}
\label{X.3}
Let $w\in [\zeta^*+1]^{<\lambda}$ be a closed set containing
$0,\zeta^*$. Assume that $(\cA,f,w)$ satisfies (i), (ii) and (iii) of
Definition \ref{X.1} and (v) of Observation \ref{X.2A}. Then there is
$(p,w)\in \bbP^{\rm RS}_{\zeta^*}$ such that $(\cA,f,w)$ is a standard
representation of $(p,w)$.
\end{proposition}

\begin{proof}
  Define $(p,w)$ so that  
\[r\rest\vare\forces_{\bbP_\vare} p\rest [\vare,\vare')=f(r)\rest [\vare,
\vare')\]
for $\vare\in w\cap\zeta^*$, $\vare'=\min(w\setminus(\vare+1))$ and $r\in
\cA$. Now check. 
\end{proof}

\begin{definition}
  \label{X.4}
Let $(\cA_1,f_1,w_1),(\cA_2,f_2,w_2)$ be standard representations. We write
$(\cA_1,f_1,w_1) \preccurlyeq (\cA_2,f_2,w_2)$ whenever 
\begin{enumerate}
\item[$(\circledast)_1$] $(\forall r\in \cA_2)(\exists s\in \cA_1)(s\leq r)$
  and 
\item[$(\circledast)_2$] if $r\in\cA_2$, $s\in \cA_1$, $s\leq r$ and
  $f(r)\leq r$, then $f(s)\leq f(r)$. 
\end{enumerate}
\end{definition}

\begin{proposition}
\label{X.4A}  
The relation $\preccurlyeq$ is transitive on standard representations. 
\end{proposition}

\begin{proof}
Assume $(\cA_1,f_1,w_1)  \preccurlyeq (\cA_2,f_2,w_2)  \preccurlyeq
(\cA_3,f_3,w_3)$. Suppose that $r_i\in\cA_i$ (for $i=1,2,3$) are such that
$r_1\leq r_2\leq r_3$ and $f_3(r_3)\leq r_3$. Then also $f_2(r_2)\leq
f_3(r_3)$ and consequently $f_2(r_2)$ and $r_2$ are compatible. By
\ref{X.1}(iii) we have $f_2(r_2)\leq r_2$ and therefore also $f_1(r_1)\leq
f_2(r_2)$. Together $f_1(r_1)\leq f_2(r_2)\leq f_3(r_3)$, and the
transitivity of $\preccurlyeq$ follows. 
\end{proof}

\begin{observation}
\label{X.5}
Suppose that $(p_1,w_1), (p_2,w_2)\in \bbP^{\rm RS}_{\zeta^*}$ and let 
$(\cA_1,f_1,w_1)$ and $(\cA_2,f_2,w_2)$ be their standard
representations. Assume $(\cA_1,f_1,w_1)\preccurlyeq (\cA_2,f_2,w_2)$. Then
$(p_1,w_1)\leq' (p_2,w_2)$.   
\end{observation}

\begin{proposition}
  \label{X.6}
Let $(p_1,w_1),(p_2,w_2)\in \bbP^{\rm RS}_{\zeta^*}$ be such that
$(p_1,w_1) \leq' (p_2,w_2)$. Suppose that $(\cA_1,f_1,w_1)$ is a standard
representation of $(p_1,w_1)$. Then there is a standard representation
$(\cA_2,f_2,w_2)$ for $(p_2,w_2)$ such that $(\cA_1,f_1,w_1) \preccurlyeq
(\cA_2,f_2,w_2)$. 
\end{proposition}

\begin{proof}
Choose a standard representation $(\cA_2,f_2,w_2)$ of $(p_2,w_2)$ such that
$(\forall s\in \cA_2)(\exists r\in \cA_1)(r\leq s)$ (possible by
Observation \ref{X.2}). Then plainly the demand $(\circledast)_2$
of Definition \ref{X.4} is also satisfied, so $(\cA_1,f_1,w_1)\preccurlyeq
(\cA_2,f_2,w_2)$.  
\end{proof}

\subsection{The Forcings}
Let us recall the definitions of the main forcing notions we are interested 
in. These forcing notions generalize the well known classical forcings with
trees used in the set theory of the reals. Thus both $\bbQ^1_\lambda$ and
$\bbQ^2_\lambda$ (defined below in \ref{defFilter}) represent possible 
generalizations of Miller's rational perfect set forcing, see Miller
\cite{Mil84a}. The Laver forcing notion (see Laver \cite{L1}) is generalized
by forcings $\bbQ^4_\lambda$. Due to flexibility of our parameters, the
schemes presented here generalize the Cohen and the Sacks forcing notions as
well, see Remark \ref{oldfor}. The many ``bounded'' variations of the
classical forcings with trees have suitable nice generalizations if
$\lambda$ is inaccessible, see Definition \ref{boundedPQ}. In Definition
\ref{Silver}(1) we will introduce generalizations of the Silver forcing notion
(or perhaps rather the Grigorieff forcing, see \cite[\S3]{Gri71}). Finally,
relatives of the Hechler forcing (Hechler  \cite{Hechler74}) and Eventually
Different Real forcing (Miller \cite[\S 5]{Mi81}) are introduced in Definition
\ref{Silver}(2,3).   

\begin{definition}
\label{defFilter}
Suppose that $\bar{E}=\langle E_t:t\in {}^{{<}\lambda}\lambda\rangle$ is a
system of $({<}\lambda)$--complete filters on a regular cardinal $\lambda$
and $E$ is a normal filter on $\lambda$. We define forcing notions
$\bbQ^{\ell,\bar{E}}$ (for $\ell=1,2,3,4$) and $\bbQ^{1,\bar{E}}_E$ and
$\bbQ^{3,\bar{E}}_E$ as follows.
\begin{enumerate}
\item {\bf A condition} in $\bbQ^{2,\bar{E}}$ is a complete $\lambda$--tree
$T\subseteq {}^{{<}\lambda}\lambda$ such that  
\begin{enumerate}
\item[(a)] if $t\in T$, then either $|\suc_T(t)|=1$ or $\suc_T(t)\in E_t$, and  
\item[(b)] $(\forall t\in T)(\exists s\in T)(t\vtl s\ \&\ |\suc_T(s)|>1)$, and    
\item[(c)$^2$] if $j<\lambda$ and a sequence $\langle
  t_i:i<j\rangle\subseteq T$ is $\vtl$--increasing,  $|\suc_T(t_i)|>1$ for
  all $i<j$ and $t=\bigcup\limits_{i<j}t_i$, then ($t\in T$ and)
  $|\suc_T(t)|>1$, 
\end{enumerate}
{\bf the order} $\leq$ of $\bbQ^{2,\bar{E}}$ is the inverse inclusion, i.e.,
$T_1\leq T_2$ if and only if $T_1,T_2\in \bbQ^{2,\bar{E}}$ and
$T_2\subseteq T_1$. 
\item Forcing notions $\bbQ^{1,\bar{E}}, \bbQ^{3,\bar{E}}, \bbQ^{4,\bar{E}}$  
  are defined analogously, but the demand (c)$^2$ is replaced by the
  respective (c)$^\ell$:
\begin{enumerate}
\item[(c)$^1$] for every $\lambda$--branch $\eta\in\lim_\lambda(T)$ the set
  \[\{\alpha<\lambda: |\suc_T(\eta\rest\alpha)|>1\}\]
includes a closed unbounded set. 
\item[(c)$^3$] for some closed unbounded set $C\subseteq \lambda$ consisting
  of limit ordinals we have   
\[(\forall t\in T)(\lh(t)\in C\ \Leftrightarrow\ |\suc_T(t)|>1),\] 
\item[(c)$^4$] $(\forall t\in T)(\mrot(T)\vtl t\ \Rightarrow\
  |\suc_T(t)|>1)$. 
\end{enumerate}
\item The forcing
  notions $\bbQ^{1,\bar{E}}_E$ and $\bbQ^{3,\bar{E}}_E$ are defined like
  $\bbQ^{1,\bar{E}}$ and $\bbQ^{3,\bar{E}}$, but in demands (c)$^1$ and
  (c)$^3$ we replace ``a  closed unbounded set'' by ``a set of limit
  ordinals belonging to the filter $E$''. 
\item If each $E_t$ is the club filter of $\lambda$ (for all $t\in 
  {}^{{<}\lambda} \lambda$), then we omit $\bar{E}$ and we write 
  $\bbQ^\ell_\lambda$ instead of $\bbQ^{\ell,\bar{E}}$.
\end{enumerate}
\end{definition}

\begin{remark}
  \label{oldfor}
\begin{enumerate}
\item Note that our definition of $\qthree$ slightly differs from the one in
  \cite{RoSh:942}, however the forcing defined here is a dense subset of the
  one defined there. 
\item The forcing notion $\bbQ^{2,\bar{E}}$ was studied by Brown and
  Groszek \cite{BrGr06} who described when this forcing adds a generic
  of minimal degree. 
\item Remember that in Definition \ref{defFilter} we allow the filters $E_t$
  to be principal. Thus if $E_t=\{\lambda\}$ for each $t\in
  {}^{{<}\lambda}\lambda$, then $\bbQ^{4,\bar{E}}$ is the $\lambda$--Cohen
  forcing $\bbC_\lambda$ and $\bbQ^{2,\bar{E}}$ is the forcing
  $\bbD_\lambda$ from \cite[Proposition 4.10]{RoSh:655}. If for each 
$t\in {}^{{<}\lambda}\lambda$ we let $E_t$ be the filter of all subsets
of $\lambda$ including $\{0,1\}$, then the forcing notion
$\bbQ^{2,\bar{E}}$ will be equivalent with Kanamori's
$\lambda$--Sacks forcing of \cite[Definition 1.1]{Ka80}.  
\item A relative of $\bbQ^2_\lambda$ was used in iterations in Friedman
and Zdomskyy \cite{FrZd10} and Friedman, Honzik and Zdomskyy
\cite{FrHoZd13}. It was called ${\rm Miller}(\lambda)$ there and the main
difference between the two forcings is in condition \cite[Definition
2.1(vi)]{FrZd10}. 
\item The property introduced in this paper does not ``capture''
  $\bbQ^1_\lambda$. In a subsequent work we will modify it to include more
  forcings of the form $\bbQ^{1,\bar{E}}_E$. However one should note that
  the forcing notions $\bbQ^1_\lambda$ and $\bbQ^2_\lambda$ are very
  similar. If in the demand \ref{defFilter}(2)(c)$^1$ we replace ``includes
  a closed unbounded set'' with ``is closed unbounded'' then we clearly get
  an equivalent definition of $\bbQ^2_\lambda$. In \cite[Section
  4]{RoSh:942} we showed that consistently the forcing notions
  $\bbQ^1_\lambda$ and $\bbQ^2_\lambda$ are equivalent, but also
  consistently they are not equivalent.
\end{enumerate}
\end{remark}

\begin{definition}
\label{boundedPQ}
Assume that 
\begin{itemize}
\item $\lambda$ is weakly inaccessible, $\varphi:\lambda\longrightarrow
  \lambda$  is a strictly increasing function such that each
  $\varphi(\alpha)$ is a regular uncountable cardinal and
  $\alpha<\varphi(\alpha)$ (for $\alpha<\lambda$),   
\item $\bar{F}=\langle F_t:t\in\bigcup\limits_{\alpha<\lambda}
  \prod\limits_{\beta<\alpha} \varphi(\beta)\rangle$ where $F_t$ is a 
  ${<}\varphi(\alpha)$--complete filter on $\varphi(\alpha)$ whenever $t\in 
  \prod\limits_{\beta<\alpha}\varphi(\beta)$, $\alpha<\lambda$,
\item $E$ is a normal filter on $\lambda$. 
\end{itemize}
\begin{enumerate}
\item We define a forcing notion $\bbQ^2_{\varphi,\bar{F}}$ as follows.\\  
{\bf A condition} in $\bbQ^2_{\varphi,\bar{F}}$ is a complete
  $\lambda$--tree $T\subseteq \bigcup\limits_{\alpha<\lambda}
  \prod\limits_{\beta<\alpha}\varphi(\beta)$ such that  
\begin{enumerate}
\item[(a)] for every $t\in T$, either $|\suc_T(t)|=1$ or $\suc_T(t)\in  
  F_t$, and
\item[(b)] $(\forall t\in T)(\exists s\in T)(t\vtl s\ \&\ |\suc_T(s)|>1)$,
  and 
\item[(c)$^2$] if $j<\lambda$ and a sequence $\langle
  t_i:i<j\rangle\subseteq T$ is $\vtl$--increasing,  $|\suc_T(t_i)|>1$ for
  all $i<j$ and $t=\bigcup\limits_{i<j}t_i$, then ($t\in T$ and)
  $|\suc_T(t)|>1$. 
\end{enumerate}
\noindent {\bf The order} of $\bbQ^2_{\varphi,\bar{F}}$ is the reverse
inclusion.  
\item Forcing notions $\bbQ^\ell_{\varphi,\bar{F}}$ for $\ell=1,3,4$ are
  defined similarly, but the demand (c)$^2$ is replaced by the respective 
  (c)$^\ell$:  
\begin{enumerate}
\item[(c)$^1$] for every $\eta\in\lim_\lambda(T)$ the set $\{\alpha<\lambda:
  |\suc_T(\eta\rest\alpha)|>1\}$ contains a  closed unbounded subset of
  $\lambda$.  
\item[(c)$^3$] for some closed unbounded set $C\subseteq\lambda$ consisting
  of limit ordinals we have    
\[\big(\forall t\in T\big)\big(\lh(t)\in C\ \Leftrightarrow\ |\suc_T(t)|>1
\big).\] 
\item[(c)$^4$] $(\forall t\in T)(\mrot(T)\vtl t\ \Rightarrow\
  |\suc_T(t)|>1)$. 
\end{enumerate}
\item Replacing ``a  closed unbounded set'' in (c)$^1$ and (c)$^3$ by ``a
  set of limit ordinals belonging to the filter $E$'' will define forcing
  notions $\bbQ^1_{\varphi,\bar{F},E}$ and $\bbQ^3_{\varphi,\bar{F},E}$,
  respectively.    
\end{enumerate}
\end{definition}

\begin{remark}
If $\lambda$ is strongly inaccessible and $\varphi,\bar{F}$ are as in 
\ref{boundedPQ}, then $\bbQ^3_{\varphi,\bar{F}}$  is a dense subset of 
$\bbQ^2_{\varphi,\bar{F}}$. 
\end{remark}

\begin{definition}
\label{Silver}
\begin{enumerate}
\item Assume that $\psi:\lambda\longrightarrow(\lambda+1)\setminus \{0,1\}$ and 
$E$ is a normal filter on $\lambda$. We define a forcing notion 
$\bbS^\psi_E$ as follows.\\
{\bf A condition} in $\bbS^\psi_E$ is a function $p$ such that 
$\dom(p)\subseteq \lambda$, $\lambda\setminus \dom(p)\in E$ and 
$p(i)<\psi(i)$ for each $i\in\dom(p)$,\\
{\bf the order} $\leq$ of $\bbS^\psi_E$ is the inclusion, i.e., $p\leq q$ if 
and only if $p,q\in \bbS^\psi_E$ and $p\subseteq q$. 

For $p\in \bbS^\psi_E$ we also set $\rt(p)=\min\big(\lambda\setminus 
\dom(p)\big)$ (the notation rt points to analogy with ``the length of the {\bf 
  r}oo{\bf t}'').  
\item {\bf A condition} in a forcing notion $\bbH_\lambda$ is a pair $(s,g)$
  such that $s\in {}^{\lambda{>}}\lambda$ and $g\in {}^\lambda\lambda$.\\ 
{\bf The order} $\leq$ of $\bbH_\lambda$ is defined by: $(s,g)\leq (s',g')$
if and only if ($(s,g),(s',g')\in\bbH_\lambda$ and) $s\trianglelefteq s'$,
$g(\alpha)\leq s'(\alpha)$ for each $\alpha\in [\lh(s),\lh(s'))$ and
$g(\alpha)\leq g'(\alpha)$ for all $\alpha\in [\lh(s'),\lambda)$. 
\item {\bf A condtion} in a forcing notion $\bbE_\lambda$ is a pair
  $(s,\bar{G})$ such that $s\in {}^{\lambda{>}}\lambda$, $\bar{G}=\langle
  G_\beta:\beta<\lambda\rangle$ and for some $\mu<\lambda$, for every
  $\beta<\lambda$ we have $G_\beta\in [\lambda]^{\leq\mu}$.\\ 
{\bf The order} $\leq$ of $\bbE_\lambda$ is defined by: $(s,\bar{G})\leq
(s',\bar{G}')$ if and only if ($(s,\bar{G}),(s',\bar{G}')\in\bbE_\lambda$
and) $s\trianglelefteq s'$, $s'(\alpha)\notin G_\alpha$ for each $\alpha\in
[\lh(s),\lh(s'))$ and $G_\alpha\subseteq G'_\alpha$ for all $\alpha\in
[\lh(s'),\lambda)$.  
\end{enumerate}
\end{definition}

\begin{observation}
\label{easyComplete}
Assume $\lambda^{<\lambda}=\lambda$. 
\begin{enumerate}
\item For $\bar{E}$ as in \ref{defFilter} and $\ell\in\{1,2,3,4\}$, the
  forcing notion $\bbQ^{\ell,\bar{E}}$ is $({<}\lambda)$--lub--complete
  (i.e., increasing sequences of length ${<}\lambda$ have least upper
  bounds).   Likewise for $\bbQ^{1,\bar{E}}_E$ and $\bbQ^{3,\bar{E}}_E$. 
\item For $\varphi,\bar{F}$ as in \ref{boundedPQ} and $\ell\in\{1,2,3,4\}$, 
  the forcing notion $\bbQ^\ell_{\varphi,\bar{F}}$ is strategically
  $({<}\lambda)$--complete. Moreover, if $\bar{T}=\langle T_\alpha:\alpha<
  \delta\rangle\subseteq \bbQ^\ell_{\varphi,\bar{F}}$ is
  $\leq_{\bbQ^\ell_{\varphi,\bar{F}}}$--increasing and $\mrot(T_\alpha)\vtl
  \mrot(T_\beta)$ for $\alpha<\beta<\delta<\lambda$, then
  $\bigcap\limits_{\alpha<\delta}T_\alpha\in \bbQ^\ell_{\varphi,\bar{F}}$ is
  the least upper bound to $\bar{T}$.  Similarly for
  $\bbQ^1_{\varphi,\bar{F},E}$ and $\bbQ^3_{\varphi,\bar{F},E}$. 
\item For $E$ and $\psi$ as in \ref{Silver}, the forcing notion
  $\bbS^\psi_E$ is $({<}\lambda)$--lub--complete.
\item The forcing notions $\bbH_\lambda$ and $\bbE_\lambda$ are
  $({<}\lambda)$--lub--complete. 
\end{enumerate}
\end{observation}

Arguably, one of the most important properties of forcing notions with trees
in the set theory of the reals is the possibility to have fusion for
suitable $\omega$--sequences of trees. Similar properties hold for our
forcing notions (with respect to $\lambda$--sequences of conditions). The
next two lemmas are actually fusion lemmas exemplifying this similarity.  

\begin{lemma}
\label{pre2.5}
Assume that 
\begin{itemize}
\item either $\bar{E}$ and $E$ are as in \ref{defFilter}, $1\leq \ell\leq 4$, and
  $\bbP=\bbQ^{\ell,\bar{E}}$, or $\bbP=\bbQ^{1,\bar{E}}_E$, or
  $\bbP=\bbQ^{3,\bar{E}}_E$,  
\item or $\lambda,\varphi,\bar{F}$ and $E$ are as in \ref{boundedPQ},
  $1\leq\ell\leq 4$ and $\bbP=\bbQ^\ell_{\varphi,\bar{F}}$ or $\bbP=
  \bbQ^1_{\varphi,\bar{F},E}$ or $\bbP=\bbQ^3_{\varphi,\bar{F},E}$.  
\end{itemize}
Suppose that $\gamma\leq \lambda$ and $T^\delta\in\bbP$ for $\delta<\gamma$
are such that  
\begin{enumerate}
\item[(i)] $T^{\delta+1}\subseteq T^\delta$ and $T^\delta\cap
  {}^\delta\lambda = T^{\delta+1}\cap {}^\delta\lambda$,  
\item[(ii)] if $\delta$ is limit, then $T^\delta=\bigcap\limits_{i<\delta}
  T^i$, 
\item[(iii)] if $t\in T^\delta\cap {}^\delta\lambda$ and
  $|\suc_{T^\delta}(t)|>1$, then $|\suc_{T^{\delta+1}}(t)|>1$.  
\end{enumerate}
Then $T^\gamma\stackrel{\rm def}{=}\bigcap\limits_{\delta<\gamma}
T^\delta\in \bbP$. 
\end{lemma}

\begin{proof}
For $\qell$ it was proved in \cite[Lemma 3.5]{RoSh:942}; for other cases
the arguments are essentially the same. 
\end{proof}

\begin{lemma}
  \label{fusSil}
Assume that $\psi,E$ are as in \ref{Silver}. Suppose that $p_\delta\in
\bbS^\psi_E$ for $\delta<\lambda$ are such that 
\begin{enumerate}
\item[(i)] $p_\delta\subseteq p_{\delta+1}$ and $p_\delta\rest
  (\delta+1)=p_{\delta+1} \rest (\delta+1)$,  and 
\item[(ii)] if $\delta$ is limit, then $p_\delta=\bigcup\limits_{i<\delta}
  p_i$.
\end{enumerate}
Then $p_\lambda\stackrel{\rm def}{=} \bigcup\limits_{\delta<\lambda}
p_\delta\in \bbS^\psi_E$. 
\end{lemma}

\begin{proof}
  Clearly $p_\lambda$ is a function with $\dom(p)= \bigcup\limits_{\delta<
    \lambda} \dom(p_\delta)\subseteq \lambda$ and it satisfies
  $p_\lambda(i)<\psi(i)$ for each $i\in \dom(p_\lambda)$. We should to argue
  that $\lambda\setminus \dom(p_\lambda)\in E$. For this we just note that
  if $\delta \in \mathop{\triangle}\limits_{\alpha<\lambda}
  (\lambda\setminus \dom(p_\alpha))$ is a limit ordinal, then $\delta\notin
  \bigcup\limits_{\alpha<\delta}\dom(p_\alpha)=\dom(p_\delta)$ (by (ii)) and
  hence $\delta\notin \dom(p_{\delta+1})$ (by (i)). The assumption (i)
  implies also that $\delta\notin \dom(p_\lambda)$. 
\end{proof}

\section{Sequential purity with diamonds}
For the rest of the paper we assume the following Context. 

\begin{context}
\label{context}
\begin{enumerate}
\item $\lambda$ is a regular uncountable cardinal,
  $\lambda^{<\lambda}=\lambda$. 
\item $D$ is a normal filter on $\lambda$.
\item A set $\cS\in D^+$ contains all successor ordinals below $\lambda$,
$0\notin\cS$ and $\lambda\setminus\cS$ is unbounded in $\lambda$. For an
ordinal $\gamma<\lambda$ we define  $\cS[\gamma]=\cS\setminus
\{\delta\leq\gamma:\delta\mbox{ is limit }\}$. 
\item $\cR$ is the closure of $\lambda\setminus\cS$ and $\bar{\gamma}=
  \langle\gamma_\alpha:\alpha<\lambda\rangle$ is the increasing enumeration
  of $\cR$ (so the sequence $\bar{\gamma}$ is increasing continuous,
  $\gamma_0=0$ and all other terms of $\bar{\gamma}$ are limit ordinals). 
\end{enumerate}
\end{context}

\begin{definition}
\label{apprS}
A sequence $\bar{f}=\langle f_\delta:\delta\in\cS\rangle$ is a {\em
  $(D,\cS)$--diamond} if $f_\delta\in {}^\delta\delta$ for $\delta\in\cS$
and $(\forall\eta\in {}^\lambda\lambda)(\{\delta\in\cS:
f_\delta\vtl\eta\}\in D^+)$.
\end{definition}

\begin{observation}
\label{easyObs}
Let $\bbP$ be a strategically $({<}\lambda)$--complete forcing notion, $D$
be a normal filter on $\lambda$.
\begin{enumerate}
\item $\forces_\bbP$`` The family of all supersets of diagonal intersections
  of members of $D^\bV$ is a (proper) normal filter on $\lambda$ ''. 
(Abusing our notation, the (name for the) normal filter generated by $D$ in
$\bV^\bbP$ will also be denoted by $D$ or sometimes by
$D^{\bV[\name{G}_{\bbP}]}$.) 
\item If $\bar{f}$ is a $(D,\cS)$--diamond, then $\forces_\bbP$`` $\bar{f}$
  is a $(D^{\bV[\name{G}_\bbP]},\cS)$--diamond ''. 
\end{enumerate}
\end{observation}

A pure extension of a tree is a subtree with the same root. This concept
appears in several places in the classical forcing with trees, e.g., in the
Laver forcing every sentence can be decided by passing to a pure extension
of a condition. The relations $\leq_\pr$ of pure extensions were important
ingredients of the properties studied in \cite[Section 2]{RoSh:888} and
\cite{RoSh:942}. The property introduced in the present paper also involves
pure extensions, but for technical reasons they are considered for bounds of
increasing sequences of conditions. However, for our ``test tree forcing
notions'' the intuitions behind the sequential purity $R^\pr$ introduced in
Definition \ref{n1.0} below should be that it is essentially the same as
$\leq_\pr$. For instance, in $\bbQ^2_\lambda$ we say that $T\in
\bbQ^2_\lambda$ is {\em a pure bound\/} to $\bar{T}=\langle T_\alpha: 
\alpha<\delta\rangle$ if $\bar{T}$ is an increasing sequence of conditions
from $\bbQ^2_\lambda$, $\delta<\lambda$ is a limit ordinal, $\alpha\leq
\lh(\mrot(T_\alpha))<\delta$ and $T$ is a pure extension of
$\bigcap\limits_{\alpha<\delta} T_\alpha$. (We will write then $\bar{T}\;
R^\pr\; T$.)  Plainly, an increasing sequence of conditions in
$\bbQ^2_\lambda$ of length $<\lambda$, all of them with the same root $t$,
has an upper bound with the root $t$. Therefore, the relation $R^\pr$
introduced above is actually a $\lambda$--sequential$^+$ purity on
$\bbQ^2_\lambda$.  

\begin{definition}
\label{n1.0}
Let $\bbQ$ be a forcing notion. A binary relation $R^\pr$ is called {\em a
  $\lambda$--sequential purity\/} on $\bbQ$ whenever $\bar{r}\; R^\pr\; r$
implies 
\begin{enumerate}
\item[(a)] $\bar{r}=\langle r_\alpha:\alpha<\delta\rangle$ is a
  $\leq_\bbQ$--increasing sequence of conditions from $\bbQ$ of limit length 
  $\delta<\lambda$, and 
\item[(b)] $r\in\bbQ$ is an upper bound of $\bar{r}$ (i.e.,
  $r_\alpha\leq_\bbQ r$ for all $\alpha<\delta$). 
\end{enumerate}
If, additionally, the relation $R^\pr$ satisfies
\begin{enumerate}
\item[(c)] if $\bar{r}=\langle r_\alpha:\alpha<\delta\rangle$, $\bar{r}\;
  R^\pr\; s_\beta$ for $\beta<\xi$, $\xi< |\delta|^+$ and $s_\beta\leq
  s_\gamma$ for $\beta<\gamma<\xi$, then there is a condition $s\in\bbQ$
  stronger than all $s_\beta$ (for $\beta<\xi$) and such that $\bar{r}\;
  R^\pr\; s$, 
\end{enumerate}
then we say that $R^\pr$ is {\em a $\lambda$--sequential$^+$ purity\/} on
$\bbQ$. 
\end{definition}

The next definition introduces the main technical concept of the paper: the
game   $\Game^\cS_\gamma(r,N,h,\bbQ,R^\pr,\bar{f},\bar{q})$ between two
players, Generic and Antigeneric. Suppose that $\bbQ=\bbQ^2_\lambda$. In a
play of $\Game^\cS_\gamma$ the two players construct two descending (with
respect to the inclusion) sequences of trees $\langle T^-_i:i<\lambda
\rangle\subseteq\bbQ^2_\lambda\cap N$ and $\langle T_i:i<\lambda\rangle
\subseteq \bbQ^2_\lambda$ with the property that oftentimes $T_i\subseteq
T^-_i$. If $i\in S[\gamma]$ then the trees $T^-_i,T_i$ are chosen by
Generic, otherwise they are picked by Antigeneric. It is not unreasonable to
think that for $i<j$ we have 
\[\mrot(T^-_i)\trianglelefteq \mrot(T_i)\vtl \mrot(T^-_j)\trianglelefteq
\mrot(T_j).\] 
Then for many limit $\delta<\lambda$, a diamond sequence $\bar{f}$ guesses
(via some coding function $h$) the sequence $\langle T^-_i:i<\delta\rangle$
and $\bigcap\limits_{i<\delta} T^-_i\in \bbQ^2_\lambda\cap N$ is a tree with
the root  
\[\mrot(\bigcap\limits_{i<\delta} T^-_i)=\bigcup\limits_{i<\delta} \mrot(T^-_i)= 
\bigcup\limits_{i<\delta} \mrot(T_i).\]
The task of Generic will be to make sure that oftentimes for these limit
$\delta$ it holds that
\[T_\delta\subseteq q_\delta\ \mbox{ and }\ \mrot(T_\delta)= \mrot(q_\delta)=
\mrot(\bigcap\limits_{i<\delta} T^-_i).\] 

\begin{definition}
[Compare {\cite[Definition 2.1]{RoSh:655}}]
\label{n1.1} 
Let $(\bbQ,\leq)$ be a strategically $({<}\lambda)$--complete forcing notion
and $R^\pr$ be a $\lambda$--sequential purity on $\bbQ$. Suppose that a
model $N\prec ({\mathcal H}(\chi),{\in},{<^*_\chi})$ is such that
$|N|=\lambda$, ${}^{<\lambda}N\subseteq N$ and $\lambda,\bbQ,D,\cS,\ldots\in 
N$ (but note we do not demand $R^\pr\in N$) and a function $h:\lambda
\longrightarrow N$ is such that its range $\rng(h)$ includes $\bbQ\cap
N$. Also, let $\bar{\cI}=\langle\cI_\alpha:\alpha<\lambda \rangle$ list all
dense open subsets of $\bbQ$ belonging to $N$ and let $\gamma<\lambda$. 
\begin{enumerate}
\item We say that a sequence $\bar{f}=\langle f_\delta:\delta\in\cS\rangle$
  is {\em a $(D,\cS,h)$--semi diamond for $\bbQ$ over $N$\/} if $f_\delta\in
  {}^\delta\delta$ for $\delta\in\cS$ and 
\begin{enumerate}
\item[$(*)$] for every $\leq_\bbQ$--increasing sequence $\bar{p}=\langle 
p_\alpha:\alpha<\lambda\rangle\subseteq \bbQ\cap N$ we have
\[\{\delta\in\cS: (\forall\alpha<\delta)(h(f_\delta(\alpha))=p_\alpha)\}  
\in D^+.\]
\end{enumerate}
Below, let $\bar{f}$ be a $(D,\cS,h)$--semi diamond for $\bbQ$ over $N$.  
\item {\em An $(N,h,\bbQ,R^\pr,\bar{f},\bar{\cI})$--candidate\/} is a
  sequence $\bar{q}=\langle q_\delta:\delta\in\cS\mbox{ limit }
  \rangle$ of conditions from $N\cap\bbQ$ satisfying for each limit
  $\delta\in\cS$:   
\begin{enumerate}
\item[(a)] if $h\circ f_\delta=\langle h(f_\delta(\alpha)):
  \alpha<\delta\rangle\subseteq \bbQ\cap N$ and it has an upper bound in
  $\bbQ$, then $h(f_\delta(\alpha))\leq q_\delta$ for all $\alpha<\delta$,
  and  
\item[(b)] if, moreover, $h\circ f_\delta\in \dom(R^\pr)$, then also $h\circ
  f_\delta\; R^\pr\;q_\delta$, and
\item[(c)] if there is $q\in \bigcap\limits_{\alpha<\delta}\cI_\alpha$ such
  that $h\circ f_\delta\; R^\pr\; q$, then also $q_\delta\in
  \bigcap\limits_{\alpha<\delta}\cI_\alpha$. 
\end{enumerate}
If $N,h,\bbQ,R^\pr,\bar{f},\bar{\cI}$ are obvious from the context we may
just say {\em a candidate}. 
\item Let $\bar{q}=\langle q_\delta:\delta\in \cS\ \&\ \delta\mbox{ is limit
    }\rangle$ be a candidate and $r\in\bbQ$. We define a game
    $\Game^\cS_\gamma(r,N,h,\bbQ,R^\pr,\bar{f},\bar{q})$ of two players,
    {\em Generic\/} and {\em Antigeneric}, as follows. A play lasts 
    $\leq\lambda$ moves and in the $i^{\rm th}$ move the players try to
    choose conditions $r_i^-,r_i \in \bbQ$ and a set $C_i\in D$ so
    that\footnote{Remember Context \ref{context}(3,4)}:  

\begin{enumerate}
\item[(a)] $r\leq r_i$, and $r_i^-\in N$, and if $i\notin\cS[\gamma]\cap \cR$
  then  $r_i^-\leq r_i$,  
\item[(b)] $(\forall i<j<\lambda)(r_i\leq r_j\ \&\ r_i^-\leq r_j^-)$, and  
\item[(c)] Generic chooses $r_i^-,r_i,C_i$ if $i\in\cS[\gamma]$, and
  Antigeneric chooses $r_i^-,r_i,C_i$ if $i\notin\cS[\gamma]$.    
\end{enumerate}
At the end Generic wins the play whenever both players always had legal
moves (so the game lasted $\lambda$ steps) and 
\begin{enumerate}
\item[$(\circledast)$] if $\delta\in\cS[\gamma]\cap \bigcap\limits_{i<\delta}
  C_i$ is a limit ordinal and $h\circ f_\delta$ is an increasing sequence of 
  conditions in $\bbQ$ such that for all $\alpha<\delta$ we have
  $h(f_\delta(\alpha+1))= r^-_{\alpha+1}$, then $q_\delta\leq r_\delta$ and
  $h\circ f_\delta\; R^\pr\; r_\delta$. 
\end{enumerate}
\item Let $\bar{q}$ be a candidate (for $N,h,\bbQ,R^\pr,\bar{f}$ and
  $\bar{\cI}$). A  condition $r\in\bbQ$ is {\em generic for the candidate
    $\bar{q}$ over $N,h,\bbQ,R^\pr,\bar{f},\cS,\gamma$\/} if Generic has a
  winning strategy in the game $\Game^\cS_\gamma(r,N,h,\bbQ,R^\pr,\bar{f}, 
  \bar{q})$. 
\end{enumerate}
\end{definition}

\begin{definition}
[Compare {\cite[Definition 2.3]{RoSh:655}}]
\label{n1.3} 
Let $\bbQ$ be a strategically $({<}\lambda)$--complete forcing notion. 
\begin{enumerate}
\item We say that $\bbQ$ is {\em purely sequentially proper over
    $(D,\cS)$--semi diamonds\/} whenever the following condition $(\odot)$
  is satisfied.  
\begin{enumerate}
\item[$(\odot)$] Assume that $\chi$ is a large enough regular cardinal and 
  $N\prec (\cH(\chi),\in,<^*_\chi)$, $|N|=\lambda$, ${}^{<\lambda} N\subseteq
N$ and $\lambda,\bbQ,D,\cS,\ldots\in N$. {\em Then\/} there exists a
$\lambda$--sequential purity $R^\pr$ on $\bbQ$ such that for every ordinal
$\gamma<\lambda$, a condition $p\in \bbQ\cap N$ and every
$\bar{\cI},h,\bar{f},\bar{q}$ satisfying    
\begin{itemize}
\item $\bar{\cI}=\langle \cI_\alpha:\alpha<\lambda\rangle$ lists all open
  dense subsets of $\bbQ$ from $N$, 
\item a function $h:\lambda\longrightarrow N$ is such that $\bbQ\cap
  N\subseteq\rng(h)$, and 
\item a sequence $\bar{f}$ is a $(D,\cS,h)$--semi diamond for $\bbQ$, and 
\item $\bar{q}$ is an  $(N,h,\bbQ,R^\pr,\bar{f},\bar{\cI})$--candidate,
\end{itemize}
there is a condition $r\geq p$ generic for $\bar{q}$ over $N,h,\bbQ,
R^\pr,\bar{f},\cS,\gamma$.   
\end{enumerate}
\item If in the condition $(\odot)$ of (1) above 
  \begin{enumerate}
\item[(a)] the relation $R^\pr$ can be required to be a
  $\lambda$--sequential$^+$ purity, then we say that $\bbQ$ is {\em purely
    sequentially$^+$ proper over $(D,\cS)$--semi diamonds},
\item[(b)] if the relation $R^\pr$ does not depend on $N$, then we add the 
  adjective {\em uniformly\/} (so we say then {\em uniformly purely
    sequentially proper\/} etc).   
  \end{enumerate}
\end{enumerate}
\end{definition}

\begin{remark}
\label{correct655}
The game $\Game^\cS_\gamma(r,N,h,\bbQ,R^{\rm pr},\bar{f},\bar{q})$
is very similar to the game $\Game(r,N,h,\bbQ,\bar{f},\bar{q})$ 
of \cite[Definition 2.1(3)]{RoSh:655}. Some of the differences in the
definition of candidates are caused by the larger generality 
and the use of $R^{\rm pr}$. But one technical point is 
actually correcting an error in \cite{RoSh:655}. 
When comparing \cite[Definition 2.1(3)$(\circledast)$]{RoSh:655} and 
Definition \ref{n1.1} here, one notices that in the former paper 
the condition $(\circledast)$ is ``active'' if the semi-diamond guessed
$\langle r^-_i:i<\delta\rangle$, while here we ``activate'' this 
demand if {\em the successor terms of the sequence\/} were guessed 
correctly. In other words, current $(\circledast)$ is ``active'' 
more often. Now, in the proof of \cite[Main Claim 3.10]{RoSh:655} 
we actually use this more often ``active'' version of $(\circledast)$ -- 
on page 73, in the paragraph between (3.14) and (3.15), we set 
$\name{r}^\ominus_{j_0}(\vare_\gamma)=s^-(\vare_\gamma)$ and not 
$\name{r}^\ominus_{j_0}(\vare_\gamma)=r^-_{j_0}(\vare_\gamma)$. 
This makes ``guessing on coordinate $\vare_\gamma$'' different from 
``guessing in $\bbP_{\zeta^*}$''. To avoid the problem we make 
limit terms of semi--diamond guessing irrelevant. This change should 
be introduced to \cite[Definition 2.1(3)$(\circledast)$]{RoSh:655} and 
this condition should read:
\begin{enumerate}
\item[$(\circledast)^{\rm corrected}$] if $\delta\in 
S\cap\bigcap\limits_{i<\delta} C_i$ is a limit ordinal and $h\circ F_\delta$
is an increasing sequence of conditions satisfying $h(F_\delta(\alpha+1))
=r^-_{\alpha+1}$ for all $\alpha<\delta$, then $q_\delta\leq r_\delta$.
\end{enumerate}
Then {\em the properness over $(D,\cS)$--semi diamonds\/}
of \cite[Definition 2.3]{RoSh:655} becomes a potentially stronger 
property, but the examples presented there still satisfy it. 

Now, 
\begin{enumerate}
\item[$(\otimes)$] {\em if $\bbP$ is proper over $(D,\cS[\gamma])$--semi 
diamonds (see \cite[Definition 2.3]{RoSh:655} plus the correction 
stated above) for each $\gamma<\lambda$, then $\bbP$ is uniformly purely  
sequentially$^+$ proper over $(D,\cS)$--semi diamonds. }
\end{enumerate}
Why? First note that $\bbP$ is $({<}\lambda)$--complete and 
define a binary relation $R^{\rm pr}$ by:
\begin{enumerate}
\item[{}] $\bar{p}\; R^{\rm pr}\; p$\quad if and only if \quad 
$\bar{p}=\langle p_\alpha:\alpha<\delta\rangle$ is a 
$\leq_\bbP$--increasing sequence of conditions in $\bbP$,
$\delta<\lambda$ is a limit ordinal and $p$ is an upper 
bound to $\bar{p}$. 
\end{enumerate}
Then $R^{\rm pr}$ is a $\lambda$--sequential$^+$ purity.  Suppose
$N,h,\bar{f},\bar{\cI},\gamma,\bar{q}= \langle q_\delta:\delta\in\cS\mbox{
  is limit }\rangle$ and $p$ are as in \ref{n1.3}(1)$(\odot)$. For
$\delta\in \cS[\gamma]$ let $q^*_\delta\in\bbP$ be such that
\begin{enumerate}
\item[{}] if $\delta$ is limit, $h\circ f_\delta\subseteq 
\bbP\cap N$ is a $\leq_{\bbP}$--increasing sequence and $q_\delta$ is its
upper bound, then  $q^*_\delta=q_\delta$, \\
otherwise $q^*_\delta$ is the $<^*_\chi$--first 
member of $\bigcap\limits_{i<\delta}\cI_i$.
\end{enumerate}
Plainly, $\bar{q}^*=\langle q^*_\delta: \delta\in 
\cS[\gamma]\rangle$ is an $(N,h,\bbP)$--candidate over 
$\bar{f}$ in the sense of \cite[Definition 2.1(2)]{RoSh:655} (and if 
$\delta\in \cS[\gamma]$ is limit and $h\circ f_\delta$ is an
increasing sequence of conditions from $\bbP$, then 
$q^*_\delta=q_\delta$). Let $\st$ be a winning strategy of
Generic in the game $\Game(r,N,h,\bbP,\bar{f},\bar{q}^*)$
of \cite[Definition 2.1]{RoSh:655}. The same strategy can be used 
by Generic in $\Game^{\cS}_\gamma(r,N,h,\bbP,R^{\rm pr},
\bar{f},\bar{q})$ and easily it is a winning strategy in
this game too.
\end{remark}

\begin{proposition}
\label{n1.2}
Let $\bbQ,N,h,\bar{\cI},R^\pr,\gamma$ be as in Definition \ref{n1.1}. 
\begin{enumerate}
\item Every $(D,\cS)$--diamond is a $(D,\cS,h)$--semi diamond for $\bbQ$
  over $N$.\\ 
Below, let $\bar{f}$ be a $(D,\cS,h)$--semi diamond for $\bbQ$ over $N$. 
\item There exists an $(N,h,\bbQ,R^\pr,\bar{f},\bar{\cI})$--candidate.\\ 
Below, let $\bar{q}$ be an $(N,h,\bbQ,R^\pr,\bar{f},\bar{\cI})$--candidate. 
\item If $\bbQ$ is $({<}\lambda)$--lub--complete\footnote{i.e., increasing
    sequences of length ${<}\lambda$ have least upper bounds}, then for any
  condition $r\in\bbQ$ both players have always legal moves in the game
  $\Game^\cS_\gamma(r,N,h,\bbQ,R^\pr,\bar{f},\bar{q})$ satisfying $r_j^-\leq
  r_j$, provided that for each $i\in \cS[\gamma]\cap \cR\cap j$ conditions
  $r_i^-,r_i$ are compatible. 
\item If $\bbQ$ is $({<}\lambda)$--complete and $r\in\bbQ$ is
  $(N,\bbQ)$--generic in the standard sense, then also both players have
  always legal moves in the game $\Game^\cS_\gamma(r,N,h,\bbQ,R^\pr,\bar{f},
  \bar{q})$ satisfying $r_j^-\leq r_j$, provided that for each $i\in
  \cS[\gamma]\cap \cR\cap j$ conditions $r_i^-,r_i$ are compatible. 
\item If $r\in\bbQ$ is generic for $\bar{q}$ over $N,h,\bbQ, R^\pr,\bar{f},
  \cS,\gamma$, then $r$ is $(N,\bbQ)$--generic in the standard sense. 
\item Consequently, if a forcing notion $\bbQ$ is purely sequentially proper
  over $(D,\cS)$--semi diamonds and there exists a $(D,\cS)$--diamond, then
  $\bbQ$ is  $\lambda$--proper in the standard sense. 
\end{enumerate}
\end{proposition}

\begin{proof}
(5)\qquad Let $\bar{q}$ be an $(N,h,\bbQ,R^\pr,\bar{f},
\bar{\cI})$--candidate and $r\in \bbQ$ be generic for $\bar{q}$ over
$N,h,\bbQ,R^\pr,\bar{f},\cS,\gamma$. Suppose that $\cI\in N$ is an open
dense subset of $\bbQ$, say $\cI=\cI_{j_0}$ (remember, $\bar{\cI}=\langle   
\cI_i:i<\lambda\rangle$ lists all open dense subsets of $\bbQ$ belonging to
$N$). We want to argue that $\cI\cap N$ is predense above $r$, so suppose
$r_0\geq r$.

Consider a play of $\Game^\cS_\gamma(r,N,h,\bbQ,R^\pr,\bar{f},\bar{q})$ in 
which Generic follows her winning strategy and Antigeneric plays as
follows.\\ 
$\bullet$\quad At stage $i=0$, Antigeneric sets $C_0=\lambda$, $r_0^-=
\emptyset_\bbQ$ and $r_0$ is the one fixed above.\\ 
$\bullet$\quad At a stage $i\notin\cS[\gamma]$, $i>0$, Antigeneric first picks
any legal move\footnote{Note that since Generic uses her winning strategy,
there are always legal moves.} $C_i,r_i^-,r_i'$ and then ``corrects'' it by
choosing a condition $r_i\geq r_i'$ so that $r_i\in\bigcap\limits_{j<i}
\cI_j$ (remember, $\bbQ$ is strategically $({<}\lambda)$--complete).       

After the play is completed and a sequence $\langle C_i,r^-_i,r_i:
i<\lambda\rangle$ is constructed, we know that the condition $(\circledast)$
of Definition \ref{n1.1}(3) is satisfied. Also, since $\bar{f}$ is a
$(D,\cS,h)$--semi diamond for $\bbQ$ over $N$, we know that $\{\delta\in\cS:
(\forall\alpha< \delta)(h(f_\delta(\alpha))=r^-_\alpha)\}\in D^+$. Pick a
limit ordinal $\delta\in \cS[\gamma]\cap\mathop{\triangle}\limits_{i<
  \lambda} C_i$ such  that $\delta>j_0$, $\delta$ is a limit of elements of
$\lambda\setminus\cS$ and $h\circ f_\delta=\langle r^-_\alpha:\alpha<
\delta\rangle$. Then by $(\circledast)$ we have that $q_\delta\leq r_\delta$
and $h\circ f_\delta\; R^\pr\; r_\delta$. Moreover, since $r_\alpha\leq
r_\delta$ for all $\alpha<\delta$ and since $\delta$ is a limit of points
from $\lambda\setminus\cS$ we get $r_\delta\in\bigcap\limits_{j<\delta}
\cI_j$.  Therefore (by \ref{n1.1}(2)(c)) $q_\delta\in\bigcap\limits_{j<
  \delta}\cI_j$, so in particular $q_\delta\in \cI_{j_0}\cap N$. But the
condition $r_\delta$ is stronger than $q_\delta$ and it is also stronger
than $r_0$, so $r_0$ is compatible with $q_\delta$.   
\end{proof}

\begin{proposition}
\label{for24}
Assume that a forcing notion $\bbP$ is one of the following types:
\begin{itemize}
\item[(a)] $\bbP=\bbS^\psi_E$ where $\psi, E$ are as in Definition 
  \ref{Silver}(1), and  
\[A\in E\ \Rightarrow (\lambda\setminus \cS)\cup A\in D,\]
\item[(b)] $\bbP=\bbQ^{\ell,\bar{E}}$ where $\bar{E}$ is as in
Definition \ref{defFilter} and $1<\ell\leq 4$, or 
\item[(c)] $\bbP=\bbQ^{3,\bar{E}}_E$ where $\bar{E},E$ are as in 
Definition \ref{defFilter}, and 
\[A\in E\ \Rightarrow (\lambda\setminus \cS)\cup A\in D,\]
\item[(d)] $\bbP=\bbQ^\ell_{\varphi,\bar{F}}$ where
  $\lambda,\varphi,\bar{F}$ are as in Definition \ref{boundedPQ},
  $1<\ell\leq 4$, or  
\item[(e)] $\bbP=\bbQ^3_{\varphi,\bar{F},E}$ where $\bar{E},E$ are as in 
Definition \ref{boundedPQ}, and  
\[A\in E\ \Rightarrow (\lambda\setminus \cS)\cup A\in D,\]
\item[(f)] $\bbP=\bbS^\psi_E$ where $\psi, E$ are as in Definition 
  \ref{Silver}, and  $\lambda\setminus\cS\in E$, or 
\item[(g)] $\bbP=\bbQ^{\ell,\bar{E}}_E$ or $\bbP=\bbQ^\ell_{\varphi,
    \bar{F}, E}$, where $\ell=1$ or $\ell=3$, $\bar{E},E$ are as in 
Definition \ref{defFilter} or $\lambda,\varphi,\bar{F},E$ are as in
Definition \ref{boundedPQ}, respectively, and $\lambda\setminus\cS\in E$, or   
\item[(h)] $\bbP=\bbH_\lambda$ or $\bbP=\bbE_\lambda$ (see Definition
  \ref{Silver}(2,3)), 
\item[(i)] $\bbP$ is strategically $({\leq}\lambda)$--complete and
  $({<}\lambda)$--complete.  
\end{itemize}
Then the forcing notion $\bbP$ is uniformly purely sequentially$^+$ proper
over $(D,\cS)$--semi diamonds.  
\end{proposition}

\begin{proof}
(a)\quad The arguments here are very similar to those for
  $\bbQ^3_{\varphi,\bar{F},E}$ below, but much simpler. We will use the fact
  that the forcing notion $\bbS^\psi_E$ is $({<})$--lub--complete without
  explicitly saying this. Define a binary relation $R^\pr$ by:

$\bar{p}\; R^\pr\; p$\quad if and only if  
\begin{itemize}
\item $\bar{p}=\langle p_\alpha:\alpha<\delta\rangle$ is an increasing
  sequence of conditions from $\bbS^\psi_E$ such that $\delta<\lambda$ is a
  limit ordinal, and 
\item $p\in\bbS^\psi_E$ is an upper bound of the sequence $\bar{p}$
  with  $\rt(p)=\rt(\bigcup\limits_{\alpha<\delta} p_\alpha)$. 
\end{itemize}
Plainly, $R^\pr$ is a $\lambda$--sequential$^+$ purity on
$\bbS^\psi_E$. Now, let $N,h,\bar{f},\bar{\cI},\gamma$ be as in the
assumptions of \ref{n1.3}(1)$(\odot)$ and let $\bar{q}=\langle 
q_\delta: \delta\in\cS \mbox{ is limit}\;\rangle$ be an appropriate
candidate. Suppose $p\in\bbS^\psi_E\cap N$. We will define a condition
$p_\lambda\geq p$ generic for $\bar{q}$. For this, by induction on
$\alpha<\lambda$ we choose conditions $p_\alpha\in \bbS^\psi_E\cap N$ as
follows.      
\begin{enumerate}
\item[$(\blacktriangle)_1$] $p_0=p$ and if $\delta$ is a limit ordinal,
  then $p_\delta=\bigcup\limits_{\xi<\delta} p_\xi$.  
\item[$(\blacktriangle)_2$] Suppose that $\alpha=\delta+1$,
  $\delta\in\cS[\gamma]$ is a limit  ordinal, and $h\circ f_\delta$ is an
  increasing sequence of conditions from $\bbS^\psi_E\cap N$ such that 
$p_\delta\leq \bigcup\limits_{\beta<\delta}h(f_\delta(\beta))$.
Then we let $p_\alpha= \big(p_\delta\rest \alpha\big)\cup \big(q_\delta\rest 
[\alpha,\lambda)\big)$.
\item[$(\blacktriangle)_3$] If $\alpha=\delta+1$ and the clause
$(\blacktriangle)_2$ does not apply, then $p_\alpha=p_\delta$.
\end{enumerate}
Under the assumptions of $(\blacktriangle)_2$ we have $p_\delta\leq
\bigcup\limits_{\beta<\delta} h(f_\delta(\beta))\leq q_\delta$. Therefore,
it should be clear that the construction can be carried out and that the 
resulting sequence $\langle p_\alpha:\alpha<\lambda\rangle$ satisfies the
assumptions of Lemma \ref{fusSil}. Consequently $p_\lambda\stackrel{\rm
  def}{=} \bigcup\limits_{\alpha<\lambda} p_\alpha\in\bbS^\psi_E$. 

Let us show that $p_\lambda$ is generic for $\bar{q}$ over $N,h,
\bbS^\psi_E, R^\pr,\bar{f},\cS,\gamma$. Consider the following
strategy $\st$ of  Generic in the game $\Game^\cS_\gamma(p_\lambda,
N,h,\bbS^\psi_E,R^\pr,\bar{f}, \bar{q})$. At a stage $i\in \cS[\gamma]$,
when a sequence $\langle r_j^-,r_j, C_j:j<i\rangle$ has been already
constructed, Generic picks $r^-_i,r_i,C_i$ so that  
\begin{enumerate}
\item[$(\blacktriangle)_4$] if $i$ is limit, then $r^-_i=\bigcup\limits_{j<i}
r^-_j$, $r_i=\bigcup\limits_{j<i} r_j$ and $C_i=\bigcap\limits_{j<i} C_j$, 
\item[$(\blacktriangle)_5$] if $i$ is a successor, say $i=j+1$, then
  $r^-_i,r_i\in \bbS^\psi_E$ are chosen so that $r_j\leq r_i$,
  $\rt(r_j)<\rt(r_i^-)=\rt(r_i)$, $r^-_j\leq r^-_i$, $p_i\leq r^-_i$ and
  $r^-_i\leq r_i$. Also, $C_i=\lambda\setminus(\dom(r_i)\cap \cS)$.      
\end{enumerate}
Note that at a stage $i=j+1$, since $r^-_j\leq r_j$ and $p_i\leq
p_\lambda\leq r_j$, we know that $r^-_j\cup p_i\in \bbS^\psi_E\cap N$ and
$(r^-_j\cup p_i)\leq r_j$. Therefore Generic may easily pick $r_i^-,r_i$ as
required. Also, by our assumption on $E$, we are sure that $C_i\in D$. 

Suppose now that $\langle r^-_i,r_i,C_i:i<\lambda\rangle$ is a play of
$\Game^\cS_\gamma(T^*,N,h,\bbS^\psi_E,R^\pr,\bar{f},\bar{q})$ in which
Generic  follows the strategy $\st$ described above. Let
$\delta\in\cS[\gamma] \cap \bigcap\limits_{i<\delta} C_i$ be a limit ordinal
such that $h\circ f_\delta$ is an increasing sequence of conditions from
$\bbS^\psi_E$ satisfying  $h(f_\delta(i+1))=r_{i+1}^-$ for all
$i<\delta$. By $(\blacktriangle)_5+(\blacktriangle)_4$ we have that $i \leq  
\rt(h(f_\delta(i)))<\delta$ and $p_i\leq_{\bbS^\psi_E} h(f_\delta(i))$ (for
all successor $i<\delta$) and $\delta\notin \dom\big(
\bigcup\limits_{j<\delta} r_j^-\big)=\dom\big(\bigcup\limits_{j<\delta}
h(f_\delta(j)\big)$. Therefore  $q_\delta\rest [\delta+1,\lambda)\subseteq
p_{\delta+1}\subseteq p_\lambda \subseteq r_\delta$, $q_\delta\rest \delta=
\big(\bigcup\limits_{j<i} h(f_\delta(j)\big) \rest \delta=r_\delta\rest\delta$ and
$\delta=\rt(r_\delta)=\rt(q_\delta)= \rt\big(\bigcup\limits_{j<i}
h(f_\delta(j))\big)$. Consequently,   $q_\delta\leq r_\delta$ and $h\circ
f_\delta\; R^\pr\; r_\delta$.    
\bigskip

\noindent (b,c,d,e)\quad We know that the forcing notion $\bbP$ is
strategically $({<}\lambda)$--complete (see Observation
\ref{easyComplete}). Define a binary relation  $R^\pr$ by:

$\bar{T}\; R^\pr\; T$\quad if and only if  
\begin{itemize}
\item $\bar{T}=\langle T_\alpha:\alpha<\delta\rangle$ is a
  $\leq_{\bbP}$--increasing sequence of conditions from $\bbP$ such that
  $\delta<\lambda$ is a limit ordinal and for $\alpha<\delta$ we have
  $\alpha\leq \lh\big(\mrot(T_\alpha)\big)<\delta$, and 
\item $T\in\bbP$ is a $\leq_{\bbP}$--upper bound of the sequence $\bar{T}$
  with  $\mrot(T)=\mrot(\bigcap\limits_{\alpha<\delta} T_\alpha)$.  (Note
  that the previous bullet implies $\bigcap\limits_{\alpha<\delta}
  T_\alpha\in \bbP$. In cases (b) and (d) we also know that
  $\bigcup\limits_{\alpha<\delta} \mrot(T_\alpha)=
  \mrot(\bigcap\limits_{\alpha<\delta} T_\alpha)$.) 
\end{itemize}
Easily, $R^\pr$ is a $\lambda$--sequential$^+$ purity on $\bbP$. Suppose
that $N\prec (\cH(\chi),\in,<^*_\chi)$ and $h,\bar{f},\bar{\cI},\gamma$ are
as in the assumptions of Condition \ref{n1.3}(1)$(\odot)$. Let
$\bar{q}=\langle q_\delta: \delta\in\cS \mbox{ is limit}\;\rangle$ be an
$(N,h,\bbP,R^\pr, \bar{f}, \bar{\cI})$--candidate and let $T\in\bbP\cap
N$. We want to find a condition $T^*\geq T$ generic for $\bar{q}$ over
$N,h,\bbP, R^\pr, \bar{f},\cS, \gamma$. To this end we choose by induction
on $\alpha< \lambda$ conditions $T_\alpha\in \bbP\cap N$ as follows.
\begin{enumerate}
\item[$(\triangle)_1$] $T_0=T$ and if $\delta$ is a limit ordinal,
  then $T_\delta=\bigcap\limits_{\xi<\delta}T_\xi$.  
\item[$(\triangle)_2$] Suppose that $\alpha=\delta+1$,
  $\delta\in\cS[\gamma]$ is a limit  ordinal, and $h\circ f_\delta$ is a
  $\leq_\bbP$--increasing sequence of conditions from $\bbP\cap N$ such that
  for all successor $\beta<\delta$ we have 
\[\beta\leq\lh\big(\mrot(h(f_\delta(\beta)))\big)<\delta\ \mbox{ and }\
T_\beta\leq h(f_\delta(\beta)).\]   
Let $T_\alpha^*=\big\{t\in T_\delta: \mrot\big(\bigcap_{\beta<\delta}
h(f_\delta(\beta))\big)\trianglelefteq t\ \Rightarrow\ t\in q_\delta\big\}$.  
If $\ell\neq 3$ then $T^*_\alpha\in \bbP$ and we put $T_\alpha=
T^*_\alpha$. Otherwise, we choose $T_\alpha\in\bbP$ such that
$T_\alpha\subseteq T^*_\alpha$ and $T_\alpha\cap {}^\alpha\lambda= 
T^*_\alpha\cap {}^\alpha\lambda$.  
\item[$(\triangle)_3$] If $\alpha=\delta+1$ and the clause
$(\triangle)_2$ does not apply, then $T_\alpha=T_\delta$.
\end{enumerate}
Suppose that $\alpha=\delta+1$ satisfies the assumptions of
$(\triangle)_2$. Then $\bigcap\limits_{\beta<\delta}h(f_\delta(\beta))\in
\bbP$ (remember Observation \ref{easyComplete}) and $T_\delta
\leq\bigcap\limits_{\beta<\delta}h(f_\delta(\beta))\leq q_\delta$ and
$\mrot(q_\delta)= \mrot\big(\bigcap\limits_{\beta<\delta}
h(f_\delta(\beta))\big)\in T_\delta$. Let $t_\delta=\mrot
\big(\bigcap\limits_{\beta<\delta} h(f_\delta(\beta))\big)$. If
$\lh(t_\delta)=\delta$, then also $t_\delta = \bigcup\limits_{\beta<\delta}
\mrot\big(h(f_\delta(\beta))\big)$ and $\suc_{T_{\delta+1}}(t_\delta)=
\suc_{q_\delta}(t_\delta)$. (Note that in cases (b) and (d) this must
happen.) If $\lh(t_\delta)>\delta$, then $\bigcup\limits_{\beta<\delta}
\mrot\big(h(f_\delta(\beta))\big)= t_\delta\rest \delta\vtl t_\delta$ and 
$\suc_{T_{\delta+1}}(t_\delta\rest \delta)=\suc_{T_\delta} (t_\delta\rest
\delta)$. (This may happen only in cases (c) and (e).)   

Therefore we may conclude that the sequence
$\langle T_\alpha:\alpha<\lambda\rangle$ satisfies the assumptions of Lemma 
\ref{pre2.5}. At limit stages $\delta<\lambda$ this implies that
$T_\delta= \bigcap\limits_{\beta <\delta} T_\beta\in \bbP$ (so the
construction can be carried out) and for $\lambda$ we get
$T^*\stackrel{\rm def}{=} \bigcap\limits_{\alpha<\lambda} T_\alpha\in\bbP$.

\begin{claim}
\label{cl7}
$T^*$ is $(N,\bbP)$--generic in the standard sense.  
\end{claim}

\begin{proof}[Proof of the Claim]
Suppose towards contradiction that there are $\xi_0<\lambda$ and a condition 
$T^+\geq T^*$ such that $T^+$ is incompatible with every member of
$\cI_{\xi_0}\cap N$. Construct inductively a sequence $\langle S_\alpha:
\alpha<\lambda\rangle$ of conditions in $\bbP$ such that 
\begin{itemize}
\item $T^+\leq S_\beta\leq S_\alpha$, $\lh(\mrot(S_\beta))< \lh(\mrot(
  S_\alpha))$ for $\beta<\alpha<\lambda$, and 
\item $S_{\alpha+1}\in\bigcap\limits_{\xi\leq\alpha} \cI_\xi$ and if
  $\alpha$ is limit, then $S_\alpha=\bigcap\limits_{\beta<\alpha} S_\beta$
  (for all $\alpha<\lambda$).  
\end{itemize}
Let $t_\alpha=\mrot(S_\alpha)$. Since $S_\alpha\subseteq T^+\subseteq
T^*\subseteq T_\alpha$ we see that $t_\alpha\in T_\alpha$ and
$(T_\alpha)_{t_\alpha} \leq S_\alpha$. Now we claim that 
\begin{enumerate}
\item[$(\heartsuit)$] there is a limit ordinal $\delta\in\cS[\gamma]$ such
  that  $\xi_0<\delta$ and   
\[h\circ f_\delta=\langle (T_\alpha)_{t_\alpha}:\alpha<\delta\rangle,\quad 
t_\delta=\bigcup_{\alpha<\delta} t_\alpha,\quad \lh(t_\delta)= 
\delta\ \mbox{ and }\ S_\delta\in \bigcap_{\alpha<\delta}\cI_\alpha.\]
\end{enumerate}
This should be clear in cases (b) and (d) (as closed unbounded subsets of
$\lambda$ belong to $D$). In cases (c) and (e) we note that  
\[A:=\{\delta\in \cS[\gamma]: \delta \mbox{ limit and } h\circ
f_\delta=\langle (T_\alpha)_{t_\alpha}:\alpha<\delta\rangle\}\in D^+.\]
Letting $\eta=\bigcup\limits_{\alpha<\lambda} t_\alpha$ we have
$\eta\in\lim_\lambda(S_\alpha)$, so 
\[A_\alpha:=\{\delta<\lambda:|\suc_{S_\alpha} (\eta\rest \delta)|>1\}\in E\]
(for each $\alpha<\lambda$). By our additional assumptions on $E$ we
conclude that  
\[A\cap \mathop{\triangle}\limits_{\alpha<\lambda} A_\alpha \cap \{\delta\in 
\cS[\gamma]\setminus (\xi_0+1): (\forall\alpha<\delta)(\lh(t_\alpha)
<\delta)\}\neq \emptyset.\]
Then any $\delta$ from the set above will witness $(\heartsuit)$.  
\medskip

So let $\delta$ be as given by  $(\heartsuit)$.  Then $h\circ f_\delta\;
R^\pr\; S_\delta$ so by \ref{n1.1}(2) we conclude $q_\delta\in
\bigcap\limits_{ \alpha<\delta}\cI_\alpha$ (in particular
$q_\delta\in\cI_{\xi_0}\cap N$). But $q_\delta\leq (T_{\delta+1})_{t_\delta}  
\leq (T^*)_{t_\delta}\leq (T^+)_{t_\delta}$, contradicting the choice of
$T^+$.   
\end{proof}

Now we are ready to argue that $T^*$ is generic for $\bar{q}$ over $N,h,
\bbP, R^\pr,\bar{f},\cS,\gamma$ and for this let us consider the following
strategy $\st$ of  Generic in the game
$\Game^\cS_\gamma(T^*,N,h,\bbP,R^\pr,\bar{f}, 
\bar{q})$. At a stage $i\in \cS[\gamma]$, when a sequence $\langle
r_j^-,r_j, C_j:j<i\rangle$ has been already constructed, Generic picks
$r^-_i,r_i,C_i$ so that  
\begin{enumerate}
\item[$(\triangle)_4$] if $i$ is limit, then $r^-_i=\bigcap\limits_{j<i}
r^-_j$, $r_i=\bigcap\limits_{j<i} r_j$ and $C_i=\bigcap\limits_{j<i} C_j$, 
\item[$(\triangle)_5$] if $i$ is a successor, say $i=j+1$, then $r^-_i,r_i$  
are chosen so that $r_j\leq r_i$, $\mrot(r_j)\vtl\mrot(r_i^-)=\mrot(r_i)$,
$r^-_j\leq r^-_i$, $T_i\leq r^-_i$ and $r^-_i\leq r_i$. Also, $C_i=
[\lh(\mrot(r_i))+1001,\lambda)$ in cases (b), (d) and $C_i=(\lambda\setminus 
\cS)\cup\{\alpha<\lambda: (\exists t\in r_i)(\lh(t)=\alpha\ \wedge\
|\suc_{r_i}(t)|>1)\}$ in cases (c), (e).      
\end{enumerate}
Why are the choices possible? By $(\triangle)_5$ at previous stages, for
each limit $i<\lambda$ the intersection of conditions played so far is in
$\bbP$ (remember Observation \ref{easyComplete}). Note also that then
$r^-_i\leq r_i$. At a stage $i=j+1$, since $r^-_j\leq r_j$, $T_i\leq T^*\leq
r_j$ and $T^*$ is $(N,\bbP)$--generic, we may find $r^i\in \bbP\cap N$
stronger than both $r^-_j$ and $T_i$ and compatible with $r_j$. Then we may
choose $r_i^*$ stronger than both $r^i$ and $r_j$. Pick $t\in r_i^*$ such
that $|\suc_{r_i^*} (t)|>1$ and $\lh(t)>\lh(\mrot(r^i))$ and set $r^-_i=
(r^i)_t$ and $r_i=(r_i^*)_t$.

Note also that there are always legal moves for Antigeneric (use
Observation \ref{easyComplete} and the choices in $(\triangle)_5$).  

Now suppose that $\langle r^-_i,r_i,C_i:i<\lambda\rangle$ is a play of
$\Game^\cS_\gamma(T^*,N,h,\bbP,R^\pr,\bar{f},\bar{q})$ in which Generic 
follows the strategy $\st$ described above. Let $\delta\in\cS[\gamma] \cap   
\bigcap\limits_{i<\delta} C_i$ be a limit ordinal such that $h\circ
f_\delta$ is an increasing sequence of conditions in $\bbP$ satisfying 
\[h(f_\delta(i))=r_i^-\quad\mbox{ for all successor }i<\delta.\]
Then our choices in $(\triangle)_5+(\triangle)_4$ imply that $i \leq
\lh\big(\mrot(h(f_\delta(i)))\big)<\delta$ and $h(f_\delta(i))\geq T_i$
(for all successor $i<\delta$). Let $t=\bigcup\limits_{i<\delta}
\mrot(r^-_i)$. Since $\lh(t)=\delta\in C_i$ for all $i<\delta$, we know that
$|\suc_{r_i}(t)|>1$ for each $i<\delta$ and therefore
$t=\mrot\big(\bigcap\limits_{i<\delta} r_i\big)= \mrot
\big(\bigcap\limits_{i<\delta} r^-_i\big)$. Now using
$(\triangle)_2+(\triangle)_4$ we conclude $q_\delta=(T_{\delta+1})_t \leq
(T^*)_t\leq (r_\delta)_t=r_\delta$ and $h\circ f_\delta\; R^\pr\; r_\delta$.    
\medskip

(f)\quad Like in the proof of (a) above we note that $\bbS^\psi_E$ is
$({<}\lambda)$--lub--complete and we will use this fact several times
(without saying this). We define a binary relation $R^\pr$ by:

$\bar{p}\; R^\pr\; p$\quad if and only if  
\begin{itemize}
\item $\bar{p}=\langle p_\alpha:\alpha<\delta\rangle\subseteq \bbS^\psi_E$
  is an increasing sequence of conditions such that $\delta<\lambda$ is a
  limit ordinal, and 
\item $p\in\bbS^\psi_E$ is an upper bound of the sequence $\bar{p}$. 
\end{itemize}
Then $R^\pr$ is a $\lambda$--sequential$^+$ purity on $\bbS^\psi_E$. Let 
$N,h,\bar{f},\bar{\cI},\gamma$ be as in the assumptions of Condition
\ref{n1.3}(1)$(\odot)$ and $\bar{q}$ be a candidate. Suppose
$p\in\bbS^\psi_E\cap N$. Choose inductively conditions $p_\alpha\in
\bbS^\psi_E\cap N$ (for  $\alpha<\lambda$) so that:
\begin{enumerate}
\item[$(\blacktriangledown)_1$] $p_0\geq p$ satisfies $\cS\subseteq\dom(p_0)$
  and if $\delta$ is limit, then $p_\delta=\bigcup\limits_{\xi<\delta} p_\xi$.    
\item[$(\blacktriangledown)_2$] If $\alpha=\delta+1$, $\delta\in\cS[\gamma]$
  is a limit  ordinal, and $h\circ f_\delta$ is an increasing sequence of
  conditions from $\bbS^\psi_E\cap N$ such that $p_\delta\leq
  \bigcup\limits_{\beta<\delta} h(f_\delta(\beta))$,  then $p_\alpha=
  \big(p_\delta\rest \alpha\big)\cup \big(q_\delta\rest
  [\alpha,\lambda)\big)$.     
\item[$(\blacktriangledown)_3$] If $\alpha=\delta+1$ and the clause 
$(\blacktriangledown)_2$ does not apply, then $p_\alpha=p_\delta$.
\end{enumerate}
It should be clear that the construction can be completed. We note that if
$\alpha=\delta+1$ satisfies the assumptions of $(\blacktriangledown)_2$,
then (by the choice of $p_0$) $\delta\in \dom(p_\delta)$. Clearly the
sequence $\langle p_\alpha:\alpha<\lambda\rangle$ satisfies the assumptions
of Lemma \ref{fusSil} and hence $p_\lambda\stackrel{\rm def}{=}
\bigcup\limits_{\alpha<\lambda} p_\alpha\in\bbS^\psi_E$.  

We claim that $p_\lambda$ is generic for $\bar{q}$. To see this, consider
the following strategy $\st$ for  Generic in the game
$\Game^\cS_\gamma(T^*,N,h,\bbS^\psi_E,R^\pr,\bar{f}, \bar{q})$. At a stage
$i\in \cS[\gamma]$, when a sequence $\langle r_j^-,r_j, C_j:j<i\rangle$ has
been already constructed, Generic picks $r^-_i,r_i,C_i$ so that   
\begin{enumerate}
\item[$(\blacktriangledown)_4$] if $i$ is limit, then $r^-_i=\bigcup\limits_{j<i}
r^-_j$, $r_i=\bigcup\limits_{j<i} r_j$ and $C_i=\bigcap\limits_{j<i} C_j$, 
\item[$(\blacktriangledown)_5$] if $i$ is a successor, say $i=j+1$, then
  $r^-_i,r_i$ are chosen so that $r_j\leq r_i$, $\rt(r_j)<\rt(r_i^-)$,
  $r^-_j\leq r^-_i$, and $p_i\leq r^-_i\leq r_i$. Also, $C_i=
  [\rt(r_i)+1001,\lambda)$.
\end{enumerate}
Suppose now that $\langle r^-_i,r_i,C_i:i<\lambda\rangle$ is a play of
$\Game^\cS_\gamma(T^*,N,h,\bbS^\psi_E,R^\pr,\bar{f},\bar{q})$ in which
Generic follows the strategy $\st$. Let $\delta\in\cS[\gamma] \cap   
\bigcap\limits_{i<\delta} C_i$ be a limit ordinal such that $h\circ
f_\delta$ is an increasing sequence of conditions in $\bbS^\psi_E$
satisfying  $h(f_\delta(i+1))=r_{i+1}^-$ for all $i<\delta$. Then, by the
description of $\st$, we have that $i \leq \rt(h(f_\delta(i)))<\delta$ and
$p_i\leq h(f_\delta(i))$ (for all successor $i<\delta$). Therefore 
$q_\delta\rest [\delta+1,\lambda)\subseteq p_{\delta+1}\subseteq p_\lambda
\subseteq r_\delta$ and $q_\delta\rest (\delta+1)=
r_\delta\rest(\delta+1)$. Consequently,   $q_\delta\leq r_\delta$ and
clearly $h\circ f_\delta\; R^\pr\; r_\delta$.     
\bigskip

\noindent (g)\quad The forcing notion $\bbP$ is strategically
$({<}\lambda)$--complete by Observation \ref{easyComplete}. Define a binary
relation $R^\pr$ by: 

$\bar{T}\; R^\pr\; T$\quad if and only if  

\begin{itemize}
\item $\bar{T}=\langle T_\alpha:\alpha<\delta\rangle$ is a
  $\leq_{\bbP}$--increasing sequence of conditions from $\bbP$ such that
  $\delta<\lambda$ is a limit ordinal and for $\alpha<\delta$ we have
  $\alpha\leq \lh\big(\mrot(T_\alpha)\big)<\delta$, and 
\item $T\in\bbP$ is a $\leq_{\bbP}$--upper bound of the sequence $\bar{T}$. 
 (Note that the previous bullet implies $\bigcap\limits_{\alpha<\delta}
  T_\alpha\in \bbP$ is the least upper bound of $\bar{T}$.)
\end{itemize}
Clearly $R^\pr$ is a $\lambda$--sequential$^+$ purity on $\bbP$. Let
$N,h,\bar{f},\bar{\cI},\gamma$ be as in the assumptions of Condition
\ref{n1.3}(1)$(\odot)$ and $\bar{q}=\langle q_\delta: \delta\in\cS \mbox{ is
  limit}\;\rangle$ be a candidate. Suppose $T\in\bbP\cap N$. 
To construct a condition $T^*\geq T$ generic for $\bar{q}$ we choose by
induction on $\alpha<\lambda$ conditions $T_\alpha\in \bbP\cap N$ as
follows.     
\begin{enumerate}
\item[$(\nabla)_1$] $T_0\geq T$ satisfies 
\[(\forall t\in T_0)(|\suc_{T_0}(t)|>1\ \Rightarrow \ \lh(t)\notin \cS),\]
and if $\delta$ is a limit ordinal, then $T_\delta=\bigcap\limits_{\xi<\delta}T_\xi$.  
\item[$(\nabla)_2$] Suppose that $\alpha=\delta+1$,
  $\delta\in\cS[\gamma]$ is a limit  ordinal, and $h\circ f_\delta$ is a
  $\leq_\bbP$--increasing sequence of conditions from $\bbP\cap N$ such that
  for all successor $\beta<\delta$ we have 
\[\beta\leq\lh\big(\mrot(h(f_\delta(\beta)))\big)<\delta\ \mbox{ and }\
T_\beta\leq h(f_\delta(\beta)).\]   
Let $t_\delta = \bigcup\limits_{\beta<\delta}
\mrot\big(h(f_\delta(\beta))\big)$  and $T_\alpha^*=\big\{t\in T_\delta:
t_\delta\trianglelefteq t\ \Rightarrow\ t\in q_\delta\big\}$.   
Pick $T_\alpha\in\bbP$ so that $T_\alpha\subseteq T^*_\alpha$ and
$T_\alpha\cap {}^\alpha\lambda= T^*_\alpha\cap {}^\alpha \lambda$.  
\item[$(\nabla)_3$] If $\alpha=\delta+1$ and the clause
$(\nabla)_2$ does not apply, then $T_\alpha=T_\delta$.
\end{enumerate}
Suppose that $\alpha=\delta+1$ satisfies the assumptions of
$(\nabla)_2$. Then $\bigcap\limits_{\beta<\delta}h(f_\delta(\beta))\in
\bbP$ and $T_\delta \leq\bigcap\limits_{\beta<\delta}h(f_\delta(\beta))\leq
q_\delta$ and $t_\delta\trianglelefteq \mrot \big(
\bigcap\limits_{\beta<\delta} h(f_\delta(\beta))\big)\trianglelefteq
\mrot(q_\delta)$. Also  $\lh(t_\delta)=\delta \in \cS$ so, by demand
$(\nabla)_1$, $|\suc_{T_\delta}(t_\delta)|=1$. Hence we may conclude that
the sequence $\langle T_\alpha:\alpha<\lambda\rangle$ satisfies the
assumptions of Lemma \ref{pre2.5}. At limit stages $\delta<\lambda$ this
implies that $T_\delta= \bigcap\limits_{\beta <\delta} T_\beta\in \bbP$ (so
the construction can be carried out) and for $\lambda$ we get
$T^*\stackrel{\rm def}{=} \bigcap\limits_{\alpha<\lambda} T_\alpha\in\bbP$. 

\begin{claim}
\label{cl4}
$T^*$ is $(N,\bbP)$--generic in the standard sense.  
\end{claim}

\begin{proof}[Proof of the Claim]
Assume that $\xi_0<\lambda$ and $T^+\geq_{\bbP} T^*$. Build inductively a
sequence $\langle S_\alpha:\alpha<\lambda\rangle$ of conditions in $\bbP$
such that  
\begin{itemize}
\item $T^+\leq S_\beta\leq S_\alpha$, $\lh(\mrot(S_\beta))< \lh(\mrot(
  S_\alpha))$ for $\beta<\alpha<\lambda$, and 
\item $S_{\alpha+1}\in\bigcap\limits_{\xi\leq\alpha} \cI_\xi$, and if
  $\alpha$ is limit then $S_\alpha=\bigcap\limits_{\beta<\alpha} S_\beta$
  (for all $\alpha<\lambda$).  
\end{itemize}
Let $s_\alpha=\mrot(S_\alpha)$. Since $T_\alpha\supseteq T^+\supseteq
S_\alpha$ we see that $s_\alpha\in T_\alpha$ and $(T_\alpha)_{s_\alpha}\leq
S_\alpha$. For some limit ordinal $\delta\in\cS[\gamma]$ we have
$\xi_0<\delta$ and $\lh(s_\alpha)<\delta$ for $\alpha<\delta$ and $h\circ
f_\delta=\langle (T_\alpha)_{s_\alpha}:\alpha<\delta\rangle$. Then $h\circ
f_\delta\; R^\pr\; S_\delta$ and $S_\delta\in\bigcap\limits_{\alpha<\delta}
\cI_\alpha$. Hence also $q_\delta\in \bigcap\limits_{\alpha<\delta}
\cI_\alpha$ (in particular $q_\delta\in\cI_{\xi_0}\cap N$). Since for
$t_\delta= \bigcup\limits_{\alpha <\delta} s_\alpha$ we have $q_\delta\leq
(T_{\delta+1})_{t_\delta} \leq (T^*)_{t_\delta}\leq (T^+)_{t_\delta}$ we may
conclude $\cI_{\xi_0}\cap N$ is predense above $T^*$. 
\end{proof}

Let us show that $T^*$ is generic for $\bar{q}$ over $N,h,
\bbP, R^\pr,\bar{f},\cS,\gamma$. Consider the following
strategy $\st$ for  Generic in the game
$\Game^\cS_\gamma(T^*,N,h,\bbP,R^\pr,\bar{f}, 
\bar{q})$. At a stage $i\in \cS[\gamma]$, when a sequence $\langle
r_j^-,r_j, C_j:j<i\rangle$ has been already constructed, Generic picks
$r^-_i,r_i,C_i$ so that  
\begin{enumerate}
\item[$(\nabla)_4$] if $i$ is limit, then $r^-_i=\bigcap\limits_{j<i}
r^-_j$, $r_i=\bigcap\limits_{j<i} r_j$ and $C_i=\bigcap\limits_{j<i} C_j$, 
\item[$(\nabla)_5$] if $i$ is a successor, say $i=j+1$, then $r^-_i,r_i$  
are chosen so that $r_j\leq r_i$, $\mrot(r_j)\vtl\mrot(r_i^-)$, $r^-_j\leq
r^-_i$, $T_i\leq r^-_i$ and $r^-_i\leq r_i$. Also, $C_i=
[\lh(\mrot(r_i))+1001,\lambda)$.       
\end{enumerate}
As in the previous part one uses Claim \ref{cl4} to show that the choices
are indeed possible and that there are always legal moves for Antigeneric.

Suppose now that $\langle r^-_i,r_i,C_i:i<\lambda\rangle$ is a play of
$\Game^\cS_\gamma(T^*,N,h,\bbP,R^\pr,\bar{f},\bar{q})$ in which Generic 
follows the strategy $\st$ described above. Let $\delta\in\cS[\gamma] \cap   
\bigcap\limits_{i<\delta} C_i$ be a limit ordinal such that $h\circ
f_\delta$ is an increasing sequence of conditions in $\bbP$ satisfying 
\[h(f_\delta(i))=r_i^-\quad\mbox{ for all successor }i<\delta.\]
Then, by $(\nabla)_5+(\nabla)_4$, we have that $i \leq
\lh\big(\mrot(h(f_\delta(i)))\big)<\delta$ and $h(f_\delta(i))\geq T_i$
(for all successor $i<\delta$). Therefore, letting
$t_\delta=\bigcup\limits_{i<\delta}\mrot(r^-_i)$ and using $(\nabla)_2$ we
conclude $q_\delta\leq (T_{\delta+1})_{t_\delta} \leq (T^*)_{t_\delta}\leq
(r_\delta)_{t_\delta} =r_\delta$ and $h\circ f_\delta\; R^\pr\; r_\delta$.   
\bigskip

\noindent (h)\quad In both cases the arguments are almost identical, so we
will consider $\bbP=\bbH_\lambda$ only. Define a binary relation $R^\pr$
by:\\ 
$\bar{p}\; R^\pr\; p$ \quad if and only if 
\begin{itemize}
\item $\bar{p}=\langle (s_\alpha,g_\alpha):\alpha< \delta\rangle$ is a
  $\leq_{\bbH_\lambda}$--increasing sequence of conditions from
  $\bbH_\lambda$ of limit length $\delta<\lambda$, and 
\item $p=(s,g)\in\bbH_\lambda$  is an upper bound to $\bar{p}$ with
  $s=\bigcup\limits_{\alpha<\delta} s_\alpha$.   
\end{itemize}
Clearly, $R^\pr$ is a $\lambda^+$--sequential purity on
$\bbH_\lambda$. Suppose $N,h,\bar{f},\bar{\cI},\gamma$ are as in 
\ref{n1.3}$(\odot)$ and let $\bar{q}$ be a candidate for these
parameters. We claim that every condition from $\bbH_\lambda\cap N$ is
generic for $\bar{q}$. So let $(s,g)\in \bbH_\lambda\cap N$. Consider the
following strategy $\st$ for Generic in $\Game^\cS_\gamma((s,g),N,h,
\bbH_\lambda,R^{\rm pr}, \bar{f},\bar{q})$. At a stage $i\in \cS[\gamma]$ of
the play, after a sequence $\langle (s_j^-,g_j^-),(s_j,g_j),C_j: j<i\rangle$
was already constructed, Generic is instructed to choose $(s_i^-,g_i^-),
(s_i,g_i)$ and $C_i$ as follows.
\begin{enumerate}
\item[$(\bullet)_1$] If $i=i_0+1$ is a successor, then $g_i^-=g_{i_0}^-$,
  $g_i=g_{i_0}$, $s_i^-=s_i=s_{i_0}\conc \langle
  g_{i_0}(\lh(s_{i_0}))\rangle$ and $C_i=[\lh(s_{i_0})+1001,\lambda)$.
\item[$(\bullet)_2$] If $i\in \cS[\gamma]$ is limit and
  $h(f(j))=(s_j^-,g_j^-)$ for each successor $j<i$, and $q_i=(s^*,g^*)$
  (note that necessarily $s^*=\bigcup\limits_{j<i} s_j^-$), then
  $s_i^-=s_i=s^*$, $C_i=\lambda$ and for $\alpha<\lambda$:
\[g_i(\alpha)=\sup\{g_j(\alpha),g^*(\alpha):j<i\}\ \mbox{ and } \ 
g_i^-(\alpha)=\sup\{g_j^-(\alpha),g^*(\alpha):j<i\}.\] 
\item[$(\bullet)_3$] If $i\in\cS[\gamma]$ is limit but we are not in the
  previous case, then $s_i^-=s_i=\bigcup\limits_{j<i} s^-_j$,
  $g_i(\alpha)=\sup\{g_j(\alpha):j<i\}$,
  $g_i^-(\alpha)=\sup\{g_j^-(\alpha):j<i\}$ for each $\alpha<\lambda$.  
\end{enumerate}
Now check. 
\bigskip

\noindent (i)\quad Define a binary relation $R^\pr$ by:\\
$\bar{p}\; R^{\rm pr}\; p$ \quad if and only if\quad $\bar{p}=\langle
p_\alpha:\alpha< \delta\rangle$ is a $\leq_\bbP$--increasing sequence
of conditions from $\bbP$ which has an upper bound in $\bbP$,
$\delta<\lambda$ is a limit ordinal and $p\in \bbP$ is an upper bound
to $\bar{p}$.  

Plainly, $R^{\rm pr}$ is a $\lambda$--sequential$^+$ purity on
$\bbP$. (Remember $\bbP$ is $({<}\lambda)$--complete.)
Assume $N,h,\bar{f},\bar{\cI},\gamma,\bar{q}$ are as in
\ref{n1.3}(1)$(\odot)$ and let $p\in N$. Let $\st^*\in N$ be a winning
strategy of Complete in the game $\Game^{\lambda+1}_0(\bbP)$. By induction
on $\alpha\leq\lambda$ choose conditions $p^*_\alpha,q^*_\alpha\in \bbP$ so
that

\begin{enumerate}
\item[$(\square)_1$] $\langle  p^*_\alpha,q^*_\alpha:\alpha\leq\lambda
  \rangle$ is a legal play of $\Game^{\lambda+1}_0(\bbP)$ in which Complete
  follows her winning strategy $\st^*$,
\item[$(\square)_2$] $p^*_0=p$, $p^*_\alpha,q^*_\alpha\in N$ for all
  $\alpha<\lambda$, 
\item[$(\square)_3$] $p^*_{\alpha+1}\in\bigcap\limits_{\beta<\alpha}
  \cI_\beta$ for each $\alpha<\lambda$, 
\item[$(\square)_4$] if $\delta\in\cS$ is limit and $q_\delta\geq p^*_\alpha$
  for all $\alpha<\delta$, then $p^*_\delta=q_\delta$.
\end{enumerate}
(Clearly the construction is possible.) We will argue that the condition
$p^*_\lambda$ is generic for the candidate $\bar{q}$ over $N,h,\bbP,R^{\rm
  pr}, \bar{f},\cS,\gamma$. Consider a strategy $\st^+$ of Generic in
$\Game^\cS_\gamma(p^*_\lambda,N,h,\bbP,R^{\rm pr},\bar{f},\bar{q})$ which
instructs her to choose $r^-_i,r_i, C_i$ (for $i\in \cS[\gamma]$) and side
ordinals $\alpha_i<\lambda$ (for $i<\lambda$) so that the following demands
$(\square)_5$--$(\square)_9$ are satisfied.
\begin{enumerate}
\item[$(\square)_5$] $r^-_i\in\bbP\cap N$ and $r_i\in\bbP$. 
\item[$(\square)_6$] $r_i^-\leq r_i$ and if $j<i$ then $r_j\leq  r_i$ and
  $\gamma<\alpha_j<\alpha_i$.   
\item[$(\square)_7$] $\alpha_i<\lambda$ is the smallest ordinal above
  $\{\gamma+1\}\cup \{\alpha_j+1:j<i\}$ such that $r^-_i\leq p_{\alpha_i}^*$ 
  and then $r^-_{i+1}=p^*_{\alpha_i}$, $r_{i+1}=r_i$ and
  $C_{i+1}=[\alpha_i+1001,\lambda)$. 
\item[$(\square)_8$] If $\delta\in\cS[\gamma]$ is limit and $q_\delta\geq
  r^-_i$ for all $i<\delta$, and the family
  $\{q_\delta\}\cup\{r_i:i<\delta\}$ has an upper bound in $\bbP$, then
  $r_\delta^-=q_\delta$ and $r_\delta\in\bbP$ is the $<^*_\chi$--first
  condition stronger than $q_\delta$ and stronger than all $r_i$ for
  $i<\delta$, and $C_i=\lambda$.
\item[$(\square)_9$] If $\delta\in \cS[\gamma]$ is limit but we are not in
  the case of $(\square)_8$, then $(r_\delta^-,r_\delta,C_\delta)$ is the
  $<^*_\chi$--first legal move of Generic after $\langle r^-_i,r_i,C_i:
  i<\delta\rangle$ which satisfies also $r_\delta^-\leq r_\delta$. 
\end{enumerate}
The above conditions fully define a strategy $\st^+$ for Generic. Concerning
clause $(\square)_7$ and the choice of $\alpha_i$, note that by
$(\square)_3$ there is an $\alpha<\lambda$ such that either the conditions
$p_\alpha^*,r^-_i$ are incompatible or $p^*_\alpha\geq 
r^-_i$. Since $r_i\geq p_\lambda^*\geq p_\alpha^*$ and $r_i\geq r^-_i$,
the conditions $p^*_\alpha,r^-_i$ are always compatible, and hence there is 
$\alpha_i<\lambda$ such that $p^*_{\alpha_i}\geq r^-_i$. Note that the
condition $p^*_\lambda$ is $(N,\bbP)$--generic in the standard sense (by
$(\square)_3$) and also for each $j<\lambda$ there is an upper bound to
$\{r_i:i<j\}$ (and this upper bound is $(N,\bbP)$--generic).  Hence,
for both players, there are always legal moves during a play in which
Generic followed $\st^+$ so far (satisfying $r_i^-\leq r_i$). So, in
particular, clause $(\square)_9$ completes the description of $\st^+$. 

Let us argue that $\st^+$ is a winning strategy for Generic. To this end,
assume $\langle r^-_i,r_i,C_i: i<\lambda\rangle$ is a play in which
Generic follows $\st^+$ and let $\langle\alpha_i:i<\lambda\rangle$ be
the side ordinals chosen by her. Suppose $\delta\in\cS[\gamma] \cap
\bigcap\limits_{i<\delta} C_i$ is a limit ordinal such that $h\circ
f_\delta$ is an increasing sequence of conditions from $\bbP\cap N$
satisfying 
\[h\big(f_\delta(i+1)\big)=r^-_{i+1}\quad \mbox{ for all }i<\delta.\] 
Then also $i\leq \alpha_i<\delta$ for each $i<\delta$ so by $(\square)_7$
we have
\[(\forall i<\delta)(\exists\alpha<\delta)(r^-_i\leq p^*_\alpha)\quad \mbox{
  and }\quad (\forall \alpha<\delta)(\exists i<\delta)(p^*_\alpha\leq
r^-_i).\] 
Thus the sequences $\langle p^*_\alpha:\alpha<\delta\rangle$ and $\langle
r^-_i:i<\delta\rangle$ and also $h\circ f_\delta$ have the same upper
bounds. Now, by \ref{n1.1}(2)(a,b), we have $q_\delta\geq h\circ
f_\delta(i+1)=r^-_{i+1}$  for all $i<\delta$, and hence $q_\delta$ is
stronger than all $p^*_\alpha$ for $\alpha<\delta$. By $(\square)_4$ we get
$q_\delta=p_\delta^* \leq  p^*_\lambda\leq r_i$ (for $i<\delta$). Thus
clause $(\square)_8$ applies and $r^-_\delta=q_\delta\leq r_\delta$ so also
$h\circ f_\delta\; R^{\rm pr}\; r_\delta$.  
\end{proof}

\section{The iteration theorem}
Here we will show the main result: pure sequential$^+$ properness is almost
preserved in $\lambda$--support iterations. The proof of the theorem is
somewhat similar to that of \cite[Theorem 3.7]{RoSh:655}, but there are
major differences in the involved concepts.

\begin{theorem}
\label{n3.7} 
Assume Context \ref{context}. Let $\bar{\bbQ}=\langle\bbP_\alpha,
\name{\bbQ}_\alpha:\alpha<\zeta^*\rangle$ be a $\lambda$--support iteration
such that for each $\alpha<\zeta^*$   
\[\forces_{\bbP_\alpha}\mbox{`` $\name{\bbQ}_\alpha$ is purely
  sequentially$^+$ proper over $(D,\cS)$--semi diamonds ''.}\] 
Then 
\begin{enumerate}
\item $\bbP_{\zeta^*}=\lim(\bar{\bbQ})$ is purely sequentially proper over
$(D,\cS)$--semi diamonds.   
\item If, additionally, for each $\alpha<\zeta^*$   
\[\forces_{\bbP_\alpha}\mbox{`` $\name{\bbQ}_\alpha$ is
  $({<}\lambda)$--complete ''}\]  
then $\bbP_{\zeta^*}$ is purely sequentially$^+$ proper over $(D,\cS)$--semi 
diamonds.    
\end{enumerate}
\end{theorem}

\begin{proof}
By adding trivial iterands to the tail of $\bar{\bbQ}$ we may assume that
$\zeta^*\geq\lambda^+$ (this demand is just to avoid slight complications in
enumerations in $(\boxtimes)_1,(\boxtimes)_2$ below). 

Assume that $\chi$ is a large enough regular cardinal and $N\prec
(\cH(\chi),\in,<^*_\chi)$, $|N|=\lambda$, ${}^{<\lambda} N\subseteq N$ and
$\lambda,\bar{\bbQ},D,\cS,\ldots\in N$. For $\xi<\zeta^*$ let
$\name{\st}^0_\xi$ be the $<^*_\chi$--first $\bbP_\xi$--name for a regular
winning strategy of Complete in $\Game^\lambda_0(\name{\bbQ}_\xi)$ (so if
$\xi\in\zeta^*\cap N$ then also $\name{\st}^0_\xi\in N$). To define a
$\lambda$--sequential purity $R^\pr$ on $\bbP_{\zeta^*}$ we first fix  
\begin{enumerate}
\item[$(\boxtimes)_1$] a list $\langle(\name{\tau}_i,\zeta_i):i<\lambda
\rangle$ of all pairs $(\name{\tau},\zeta)\in N$ such that $\zeta\leq
\zeta^*$, $\cf(\zeta)\geq \lambda$ and $\name{\tau}$ is a $\bbP_\zeta$--name
for an ordinal, and 
\item[$(\boxtimes)_2$] an increasing sequence $\bar{w}=\langle w_i:i<\lambda
  \rangle$ of closed subsets of $\zeta^*+1$ such that  
  \begin{enumerate}
  \item[$(\boxtimes)_2^a$] $w_0=\{0,\zeta^*\}$, $\bigcup\limits_{i<\lambda}
    w_i= (\zeta^*+1)\cap N$ and for each $i<\lambda$ we have $|w_i|\leq 2+i$
    and 
  \item[$(\boxtimes)_2^b$] if $i$ is limit then $w_i$ is the closure of
    $\bigcup\limits_{j<i} w_j$, and 
  \item[$(\boxtimes)_2^c$] if $i$ is a successor, say $i=j+1$, then for some
    $\xi\in w_i\cap \zeta_j$ we have $\sup(w_j\cap\zeta_j)<\xi$. 
  \end{enumerate}
\end{enumerate}
We also define a function $\Upsilon:(\zeta^*+1)\cap N\longrightarrow
\lambda$ by  
\begin{enumerate}
\item[$(\boxtimes)_3$] $\Upsilon(\vare)=\min(i<\lambda: \vare\in w_i)$ for
  $\vare\in (\zeta^*+1)\cap N$. 
\end{enumerate}
Then for every $\zeta\in\zeta^*\cap N$ we fix a $\bbP_\zeta$--name
$\name{R}^\pr_\zeta$ for a binary relation such that   
\[\begin{array}{ll}
\forces_{\bbP_\zeta}&\mbox{`` $\name{R}^\pr_\zeta$ is a
  $\lambda$--sequential$^+$ purity on $\name{\bbQ}_\zeta$ such that }\\
&\ \mbox{ if $N[\name{G}_{\bbP_\zeta}]$ is a model of the right form, then }\\
&\ \mbox{ $\name{R}^\pr_\zeta$ witnesses the pure sequential$^+$ properness
  of $\name{\bbQ}_\zeta$ for $N[\name{G}_{\bbP_\zeta}]$ ''.}
\end{array}\]  
Now, for $\zeta\leq\zeta^*$, we define a binary relation $R^\pr\rest\zeta$ by 
\begin{enumerate}
\item[$(\boxtimes)_4$] $\bar{p}\;R^\pr\rest \zeta\; p$ if and only if 
\begin{enumerate}
\item[$(\boxtimes)_4^a$] $\bar{p}=\langle p_\alpha:\alpha<\delta\rangle$ is
  an increasing sequence of conditions from $\bbP_\zeta$ of a limit length 
  $\delta<\lambda$, $p\in\bbP_\zeta$ is an upper bound to $\bar{p}$ and  
\item[$(\boxtimes)_4^b$] for every $\xi\in\zeta\cap N$ such 
that\footnote{the ``$\cdot$'' stands for the ordinal multiplication}  
$\Upsilon(\xi)\cdot \omega\leq\delta$ we have 
\[p\rest\xi\forces_{\bbP_\xi}\mbox{`` }\langle p_{\Upsilon(\xi)+
\alpha}(\xi):\alpha<\delta\rangle\; \name{R}^\pr_\xi\; 
p(\xi)\mbox{ ''.}\]  
\end{enumerate}
\end{enumerate}
The relation $R^\pr$ is $R^\pr\rest\zeta^*$. 

One may wonder why we use $\Upsilon\cdot\omega$ in $(\boxtimes)^b_4$. When
playing the game $\Game^\cS_1$, we will play its variant on each
coordinate from $N\cap \zeta^*$. However, at a stage $\delta$ of the main
game, we will look at coordinates from $w_\delta$ only, so we will start our
coordinate games with some delay. At stage $j=\Upsilon(\xi)+i$, the length
of the play constructed so far at coordinate $\xi$ is $i$. It may happen
that $i\in\cS$ but $\Upsilon(\xi)+i\notin \cS$. Since we do want to control
who is making a move at every coordinate, it makes sense to demand that
$\delta\setminus \Upsilon(\xi)$ has order type $\delta$ whenever $\delta$ is
of any potential importance. This translates to the demand that at
coordinate $\xi\in N\cap\zeta^*$ we play game $\Game^\cS_{\Upsilon(\xi)\cdot
  \omega}$ and relevant $\delta$s are not smaller that  $\Upsilon(\xi)\cdot
\omega$.  

Clearly, for each $\zeta\leq
\zeta^*$, 
\begin{itemize}
\item $R^\pr\rest\zeta$ is a $\lambda$--sequential purity on $\bbP_\zeta$, 
\item if each $\name{\bbQ}_\alpha$ is (forced to be)
  $({<}\lambda)$--complete\footnote{We need this assumption to get bounds on
    coordinates $\xi<\zeta$ for which $(\boxtimes)^b_4$ is not ``active''.},
  then $R^\pr\rest\zeta$ is a   $\lambda$--sequential$^+$ purity, 
\item if $\langle p_\alpha:\alpha<\delta\rangle\, (R^\pr\rest\zeta)\; p$ and
  $\zeta'<\zeta$ then  $\langle p_\alpha\rest \zeta':\alpha<\delta\rangle\;
  (R^\pr\rest\zeta')\; p\rest\zeta'$. 
\end{itemize}

We will show that the relation $R^\pr$ witnesses the pure sequential
properness of $\bbP_{\zeta^*}$ for $N$. So let $h,\bar{\cI},\bar{f},\bar{q}$
and $p$ be as in $(\odot)$ of Definition \ref{n1.3}(1), that is they satisfy
the following conditions $(\boxtimes)_5$--$(\boxtimes)_8$.
\begin{enumerate}
\item[$(\boxtimes)_5$] A function $h:\lambda\longrightarrow N$ is such that
  $\bbP_{\zeta^*}\cap N\subseteq\rng(h)$. 
\item[$(\boxtimes)_6$] A sequence $\bar{f}=\langle f_\delta:\delta\in\cS
  \rangle$ is a $(D,\cS,h)$--semi diamond for $\bbP_{\zeta^*}$ over $N$.
\item[$(\boxtimes)_7$] $\bar{\cI}=\langle\cI_\alpha:\alpha<\lambda\rangle$
  lists all open dense subsets of $\bbP_{\zeta^*}$ from $N$.
\item[$(\boxtimes)_8$] $\bar{q}=\langle q_\delta:\delta\in\cS\mbox{ is limit
}\rangle$ is an $(N,h,\bbP_{\zeta^*},R^\pr,\bar{f},\bar{\cI})$--candidate,
$p\in \bbP_{\zeta^*}\cap N$. 
\end{enumerate}
We also have an ordinal $\gamma<\lambda$, but replacing $\cS$ with 
$\cS[\gamma]$ in the following arguments makes not much difference, 
so we may pretend that $\gamma=1$ and $\cS[\gamma]=\cS$. 

We will define a condition $r\geq p$ such that Generic will have a winning
strategy in the game $\Game^\cS_1(r,N,h,\bbP_{\zeta^*},R^\pr,\bar{f},
\bar{q})$, that is $r$ will be generic for the candidate $\bar{q}$. This
condition will be defined essentially by considering the game
$\Game^\cS_{\Upsilon(\zeta)\cdot\omega}$ on each coordinate $\zeta\in
\zeta^*\cap N$. The proof that it has the required properties will be by
induction on $\zeta$. Therefore, we have to introduce both the restrictions
of our parameters to $\zeta$ (denoted with superscript $[\zeta]$) and the
parameters at coordinate $\zeta$ (denoted with superscript
$\langle\zeta\rangle$). We will also need other auxiliary notions,
definitions and claims. First note that if we modify $f_\delta$ (and
$q_\delta$) for all limit $\delta\in\cS$ such that $h\circ f_\delta$ is not
an increasing sequence of conditions from $\bbP_{\zeta^*}$, then we can make 
the winning for Generic only more difficult. So we may assume that for each
limit $\delta\in\cS$ we have
\begin{enumerate}
\item[$(\boxtimes)_9$] $h\circ f_\delta$ is a
  $\leq_{\bbP_{\zeta^*}}$--increasing sequence of members of
  $\bbP_{\zeta^*}\cap N$.
\end{enumerate}
\medskip

For $\zeta\in N\cap(\zeta^*+1)$ we define 
\begin{enumerate}
\item[$(\boxtimes)_{10}$] $h^{[\zeta]}:\lambda\longrightarrow N$ is a
  function such that for each $\alpha<\lambda$:\\
$h^{[\zeta]}(\alpha)=h(\alpha)\rest\zeta$ provided $h(\alpha)$ is a
function, and $h^{[\zeta]}(\alpha)=\emptyset$ otherwise, 
\item[$(\boxtimes)_{11}$] $\bar{\cI}^{[\zeta]}=\langle \cI_\alpha^{[\zeta]}: 
\alpha<\lambda\rangle$, where $\cI^{[\zeta]}_\alpha=\{p\rest\zeta: p\in
\cI_\alpha\}$.
\end{enumerate}
\medskip

\noindent Clearly $\langle q_\delta\rest \zeta:\zeta\in\cS\rangle$ does not
have to be a $(N,h^{[\zeta]},\bbP_\zeta,R^\pr\rest
\zeta,\bar{f},\bar{\cI}^{[\zeta]})$--candidate. To define suitable
$q^{[\zeta]}$ we need the concept of saturated conditions. Suppose that
$\delta\in\cS$ is a limit ordinal, $\zeta\in w_\delta$ and
$q^*\in\bbP_\zeta$. We say that the condition $q^*$ is {\em saturated over
  $w_\delta,h,\bar{\cI}$ at $\zeta$} if one of the following two
possibilities $(\boxtimes)_{12}^a$ or $(\boxtimes)_{12}^b$ holds:
\begin{enumerate}
\item[$(\boxtimes)_{12}^a$] 
  \begin{enumerate}
\item[(i)] $(h^{[\zeta]}\circ f_\delta)\;
  (R^\pr\rest\zeta)\; q^*$ and $q^*\in\bigcap\limits_{\alpha<\delta}
  \cI_\alpha^{[\zeta]}$, but
\item[(ii)]  if $\zeta'\in w_\delta\setminus (\zeta+1)$ then there is no 
  $q'\in\bbP_{\zeta'}$ such that $q^*\leq q'$ and 
\[(h^{[\zeta']}\circ f_\delta)\; (R^\pr\rest\zeta')\; q'\quad\mbox{ and 
}\quad q'\in\bigcap\limits_{\alpha<\delta} \cI_\alpha^{[\zeta']}.\]  
  \end{enumerate}
\item[$(\boxtimes)_{12}^b$] 
  \begin{enumerate}
\item[(i)] $0<\zeta=\sup(w_\delta\cap\zeta)$ and for each $\xi\in
    w_\delta\cap\zeta$ we have  
\[(h^{[\xi]}\circ f_\delta)\; (R^\pr\rest\xi)\; q^*\rest\xi\quad\mbox{ and 
}\quad q^*\rest\xi\in\bigcap\limits_{\alpha<\delta}
\cI_\alpha^{[\xi]},\quad\mbox{ but }\]  
\item[(ii)]there is no condition $q'\in\bbP_\zeta$ such that $q^*\leq q'$ and 
\[(h^{[\zeta]}\circ f_\delta)\; (R^\pr\rest\zeta)\; q'\quad\mbox{ and 
}\quad q'\in\bigcap\limits_{\alpha<\delta} \cI_\alpha^{[\zeta]}.\]   
  \end{enumerate}
\end{enumerate}

\begin{claim}
\label{cl0}
Suppose $\delta\in\cS$ is a limit ordinal. Then there exist $\zeta\in
w_\delta$ and $q^*\in \bbP_\zeta\cap N$ such that $q^*$ is saturated over
$w_\delta,h,\bar{\cI}$ at $\zeta$. 
\end{claim}

\begin{proof}[Proof of the Claim]
Let $\langle\alpha_\xi:\xi\leq\xi^*\rangle$ be the increasing enumeration 
of $w_\delta$ (so it is a continuous sequence and $\alpha_{\xi^*}=
\zeta^*$ and $\xi^*<|\delta|^+$). We construct the condition $q^*$ as an
upper bound to an increasing sequence $\langle q^+_\xi:\xi<\xi_0\rangle$,
$\xi_0\leq\xi^*$. Conditions $q^+_\xi$ are chosen so that $q^+_\xi\in
\bbP_{\alpha_\xi}\cap N$ and they satisfy $(\boxtimes)^a_{12}$(i). To make
sure that we can get pass the limit steps, at coordinates $\vare\in
N\cap\zeta^*$ which are not relevant for $R^\pr$ (i.e., with
$\Upsilon(\vare)\cdot\omega >\delta$) we employ strategic
completeness. (This forces us to introduce auxiliary conditions $q^*_\xi$.)
We do not have to worry about coordinates relevant for $R^\pr$ (i.e., with
$\Upsilon(\vare)\cdot \omega\leq \delta$) as there the appropriate bounds
exist by \ref{n1.0}(c). The construction may stop at some stage
$\xi_0<\xi^*$ because we arrived to a saturated condition.  If the procedure
does not stop before $\xi^*$, then we get a condition saturated at
$\zeta^*$.  

Let us present this argument with all technical details. So we attempt to
choose inductively a  sequence $\langle q^+_\xi,q^*_\xi:\xi\leq
\xi^*\rangle$ so that  $q^+_\xi,q^*_\xi\in\bbP_{\alpha_\xi}\cap N$ and   
\begin{enumerate}
\item[(i)]  if $\xi'<\xi''\leq\xi^*$ then $q^+_{\xi'}\leq q^*_{\xi'}\leq
  q^+_{\xi''}\leq q^*_{\xi''}$, 
\item[(ii)] if $\vare\in N\cap \zeta^*$ and $\Upsilon(\vare) \cdot
  \omega>\delta$, then for every $\xi\leq\xi^*$,  
\[\begin{array}{ll}
q^*_\xi\rest \vare\forces_{\bbP_\vare}&\mbox{`` }\langle q^+_{\xi'}(\vare),
q^*_{\xi'}(\vare): \xi'\leq\xi\rangle\mbox{ is a legal partial play of }
\Game^\lambda_0(\name{\bbQ}_\vare)\\ 
&\mbox{ in which Complete uses her regular winning strategy 
$\name{\st}^0_\vare$ '',}
\end{array}\]
\item[(iii)] if $\vare\in N\cap\zeta^*$ and $\Upsilon(\vare)\cdot \omega
\leq\delta$, then $q^+_\xi(\vare)=q^*_\xi(\vare)$, 
\item[(iv)] $(h^{[\alpha_\xi]}\circ f_\delta)\; (R^\pr\rest\alpha_\xi)\;
  q^+_\xi$ and $q^+_\xi\in \bigcap\limits_{\alpha<\delta}
  \cI^{[\alpha_\xi]}_\alpha$.  
\end{enumerate}
If we have arrived to a limit stage $\xi_0$ of the construction and we have
defined successfully $\langle q^+_\xi,q^*_\xi:\xi<\xi_0\rangle$ (so that
conditions (i)--(iv) are satisfied for all $\xi<\xi_0$), then, remembering
${}^{<\lambda} N\subseteq N$ and $\delta<\lambda$, we may find a
condition $q^\diamond_{\xi_0}\in\bbP_{\alpha_{\xi_0}}\cap N$ which is above
all $q^+_\xi$ (for $\xi<\xi_0$) and satisfies\footnote{Here we use the
  assumption that $\name{R}^\pr_\vare$ are sequential$^+$ purities rather
  than just sequential purities.} 
$(h^{[\alpha_{\xi_0}]}\circ f_\delta)\; (R^\pr\rest \alpha_{\xi_0})\;
q^\diamond_{\xi_0}$. Then also $q^\diamond_{\xi_0}\rest\alpha_\xi\in
\bigcap\limits_{\alpha<\delta} \cI_\alpha^{[\alpha_\xi]}$ for all $\xi<\xi_0$,
so either $q^\diamond_{\xi_0}$ is saturated by $(\boxtimes)^b_{12}$ (and
then we stop), or  else we may pick $q^+_{\xi_0}\in\bbP_{\alpha_{\xi_0}}
\cap N$ stronger than $q^\diamond_{\xi_0}$ and satisfying the demand in
(iv), and then we pick $q^*_{\xi_0}\in \bbP_{\alpha_{\xi_0}}\cap N$ by
(ii)+(iii).   

Now, if we arrived to a successor stage $\xi=\xi_0+1\leq \xi^*$ and we have
defined $q^+_{\xi_0},q^*_{\xi_0}$, then either $q^*_{\xi_0}$ is saturated by 
$(\boxtimes)^a_{12}$ (and then we stop), or else we may pick $q^+_{\xi_0+1}
\in \bbP_{\alpha_{\xi_0+1}}\cap N$ stronger than $q^*_{\xi_0}$ and such that
(iv) holds, and then we choose $q^*_{\xi_0+1}$ by (ii)+(iii). (We also
stipulate $q^+_0=q^*_0=\emptyset_{\bbP_0}\in\bbP_0$.) 

If we did manage to carry out the construction up to $\xi^*$ and we defined
successfully $q^+_{\xi^*}$, then it is vacuously saturated by
$(\boxtimes)^a_{12}$. 
\end{proof}

Now, for each limit $\delta\in\cS$ we fix a pair $(q^*_\delta,
\zeta^*_\delta)\in N$ so that  
\begin{enumerate}
\item[$(\boxtimes)_{13}^a$] $q_\delta^*$ is saturated over
  $w_\delta,h,\bar{\cI}$ at $\zeta_\delta^*$, and 
\item[$(\boxtimes)_{13}^b$] if the condition $q_\delta$ given by 
$(\boxtimes)_8$ is saturated over $w_\delta,h,\bar{\cI}$ at $\zeta^*$, 
then $q_\delta^*=q_\delta$, $\zeta^*_\delta=\zeta^*$. 
\end{enumerate}
Next for $\zeta\in (\zeta^*+1)\cap N$ we define 
\begin{enumerate}
\item[$(\boxtimes)_{14}$] $\bar{q}^{[\zeta]}=\langle q_\delta^{[\zeta]}:
  \delta\in\cS\mbox{ is limit }\rangle$, where for a limit $\delta\in \cS$
  we set 
\begin{itemize}
\item if $\zeta<\zeta^*_\delta$ then $q_\delta^{[\zeta]}=q^*_\delta\rest\zeta$,  
\item if $\zeta=\zeta^*_\delta$, $q^*_\delta\in\bigcap\limits_{\alpha<\delta}
  \cI_\alpha^{[\zeta]}$ and $(h^{[\zeta]}\circ f_\delta)\;
  (R^\pr\rest\zeta)\; q^*_\delta$, then $q^{[\zeta]}_\delta=q^*_\delta$, 
\item otherwise, $q_\delta^{[\zeta]}$ is the $<^*_\chi$--first condition in
  $\bbP_\zeta$ satisfying the demands in (a)--(c) of \ref{n1.1}(2) with
  $h^{[\zeta]},R^\pr\rest\zeta,\bar{\cI}^{[\zeta]}$ here in place of
  $h,R^\pr,\bar{\cI}$ there.
\end{itemize}
\end{enumerate}

\begin{claim}
\label{cl1}
Let $\zeta\in N\cap (\zeta^*+1)$. Then 
\begin{enumerate}
\item  $h^{[\zeta]}:\lambda\longrightarrow N$ is such that $\bbP_\zeta\cap
  N\subseteq\rng(h^{[\zeta]})$, 
\item $\bar{f}$ is a $(D,\cS,h^{[\zeta]})$--semi diamond sequence for
  $\bbP_\zeta$,
\item $\bar{\cI}^{[\zeta]}$ lists all open dense subsets of $\bbP_\zeta$
  belonging to $N$, 
\item $\bar{q}^{[\zeta]}$ is an $(N,h^{[\zeta]},\bbP_\zeta,R^\pr\rest \zeta,
  \bar{f},\bar{\cI}^{[\zeta]})$--candidate.
\end{enumerate} 
\end{claim}

\begin{proof}[Proof of the Claim]
Straightforward from the definitions. 
\end{proof}

For $\zeta\in\zeta^*\cap N$ we set
\begin{enumerate}
\item[$(\boxtimes)_{15}$] $\bar{f}^{\langle\zeta\rangle}=\langle
f_\delta^{\langle\zeta\rangle}:\delta\in\cS\rangle$,  where for $\delta\in
\cS$ and $\alpha<\delta$ we let
\[f^{\langle\zeta\rangle}_\delta(\alpha)=\left\{
\begin{array}{ll}
f_\delta(\alpha) & \mbox{ if }\delta<\Upsilon(\zeta)\cdot \omega,\\
f_\delta(\Upsilon(\zeta)+\alpha)& \mbox{ if }\delta\geq\Upsilon(\zeta)\cdot
\omega,
\end{array}\right.\]
\end{enumerate}
and then we define $\bbP_\zeta$--names $\name{h}^{\langle\zeta\rangle},
\name{\bar{\cI}}^{\langle \zeta\rangle},\name{\bar{q}}^{\langle \zeta\rangle}$
so that 
\begin{enumerate}
\item[$(\boxtimes)_{16}$] $\name{\bar{\cI}}^{\langle \zeta\rangle}=\langle
  \name{\cI}^{\langle \zeta\rangle}_\alpha:\alpha<\lambda\rangle$ where
  $\name{\cI}^{\langle \zeta\rangle}_\alpha=\{p(\zeta):p\in \cI_\alpha\ \&\
  p\rest\zeta\in \name{G}_{\bbP_\zeta}\}$,
\item[$(\boxtimes)_{17}$] $\name{h}^{\langle\zeta\rangle}$ is a name for a
  function with domain $\lambda$ and such that for each $\gamma<\lambda$, 
  if $h(\gamma)$ is a function, $\zeta\in\dom(h(\gamma))$ and
  $(h(\gamma))(\zeta)$ is a $\bbP_\zeta$--name, then $\name{h}^{\langle
  \zeta\rangle}(\gamma)=(h(\gamma))(\zeta)$, otherwise it is
$\name{\emptyset}_{\name{\bbQ}_\zeta}$,
\item[$(\boxtimes)_{18}$] $\name{\bar{q}}^{\langle\zeta\rangle}= \langle
\name{q}_\delta^{\langle\zeta\rangle}:\delta\in\cS\mbox{ is limit }\rangle$,
where for a limit $\delta\in \cS$, $\name{q}_\delta^{\langle\zeta\rangle}$
is the $<^*_\chi$--first $\bbP_\zeta$--name for a condition in
$\name{\bbQ}_\zeta$ such that:\\   
if $\zeta<\zeta^*_\delta$ and the condition $q_\delta^*(\zeta)=
q_\delta^{[\zeta+1]}(\zeta)$ satisfies the demands in (a)--(c) of
Definition \ref{n1.1}(2) with $\name{h}^{\langle\zeta\rangle},  
f_\delta^{\langle\zeta\rangle},\name{R}^\pr_\zeta,
\name{\bar{\cI}}^{\langle\zeta\rangle}$ here in place of $h, f_\delta,
R^\pr,\bar{\cI}$ there, then $\name{q}_\delta^{\langle\zeta\rangle}= 
q^*_\delta(\zeta)$, otherwise, $\name{q}_\delta^{\langle\zeta\rangle}$ is 
any condition in $\name{\bbQ}_\zeta$ satisfying those demands.
\end{enumerate}
Note that 

\begin{enumerate}
\item[$(\boxtimes)_{19}$] if $\delta\in \cS$ is limit and $\zeta\in 
\zeta^*_\delta\cap N$ and $\Upsilon(\zeta)\cdot \omega\leq\delta$, then 
$q^*_\delta\rest\zeta\forces_{\bbP_\zeta}$``$\name{q}_\delta^{\langle
\zeta \rangle}=q^*_\delta(\zeta)$''. 
\end{enumerate}

\begin{claim}
\label{cl2}
Assume that $\zeta\in N\cap \zeta^*$ and $r\in\bbP_\zeta$ is
$(N,\bbP_\zeta)$--generic in the standard sense. Then the condition $r$
forces (in $\bbP_\zeta$) that:   
\begin{enumerate}
\item $N[\name{G}_{\bbP_\zeta}]$ is a model of the right form,
$\name{\bar{\cI}}^{\langle \zeta\rangle}$ lists all open dense subsets of
$\name{\bbQ}_\zeta$ which belong to $N[\name{G}_{\bbP_\zeta}]$ and the
function $\name{h}^{\langle\zeta \rangle}$ is such that $N[
\name{G}_{\bbP_\zeta}]\cap\name{\bbQ}_\zeta\subseteq \rng(
\name{h}^{\langle\zeta\rangle})$, and   
\item $\bar{f}^{\langle \zeta\rangle}$ is a $(D,\cS, \name{h}^{\langle
\zeta\rangle})$--semi diamond for $\name{\bbQ}_\zeta$ over
$N[\name{G}_{\bbP_\zeta}]$, and    
\item $\name{\bar{q}}^{\langle \zeta\rangle}$ is an
  $(N[\name{G}_{\bbP_\zeta}],h^{\langle\zeta \rangle},\name{\bbQ}_\zeta,
  \bar{f}^{\langle\zeta\rangle}, \name{\bar{\cI}}^{\langle\zeta\rangle}
  )$--candidate.  
\end{enumerate}
\end{claim}

\begin{proof}[Proof of the Claim]
(2)\quad Assume that $r^*\in\bbP_\zeta$, $r^*\geq r$, and
$\name{\bar{q}}'=\langle\name{q}_\alpha':\alpha<\lambda\rangle$ is a 
$\bbP_\zeta$--name for an increasing sequence of conditions from
$\name{\bbQ}_\zeta\cap N[\name{G}_{\bbP_\zeta}]$ (so $\name{q}'_\alpha$ is
a name for an object in $N[\name{G}_{\bbP_\zeta}]$, but it does not have to
belong to $N$). Suppose also that $\name{A}_\xi$ (for $\xi<\lambda$) are
$\bbP_\zeta$--names for members of $D\cap\bV$.   

Construct inductively a sequence $\langle r_i^-,\name{q}''_i,r_i,r_i^+,A_i:
i<\lambda\rangle$ so that 
\begin{enumerate}
\item[(i)] $r_i^-\in \bbP_\zeta\cap N$, $r_i,r_i^+\in \bbP_\zeta$, $r^*\leq
  r_0$, $r^-_i\leq r_i\leq r^+_i$, $r^-_i\leq r^-_j$ for $i<j<\lambda$, 
\item[(ii)] for each $\vare<\zeta$ and $j<\lambda$: 
\[\begin{array}{ll}
r^+_j\rest \vare\forces_{\bbP_\vare}&\mbox{`` }\langle r_i(\vare), r_i^+
(\vare): i\leq j\rangle\mbox{ is a legal partial play of }
\Game^\lambda_0(\name{\bbQ}_\vare) \\   
&\mbox{ in which Complete uses her regular winning strategy 
$\name{\st}^0_\vare$ '',}  
\end{array}\]

\item[(iii)] $\name{q}_i''\in N$ is a $\bbP_\zeta$--name for a condition in
  $\name{\bbQ}_\zeta$ and $r_i\forces_{\bbP_\zeta}$`` $\name{q}'_i=
  \name{q}''_i$ '',  
\item[(iv)] $A_i\in D\cap\bV$ and  $r_i\forces_{\bbP_\zeta}$`` $\name{A}_i=A_i$ '',
\item[(v)] $r^-_j\forces_{\bbP_\zeta}$`` the sequence $\langle\name{q}_i'':
i\leq j\rangle$ is $\leq_{\name{\bbQ}_\zeta}$--increasing ''.  
\end{enumerate}
[Why is the construction possible? First, after arriving to a step
$j<\lambda$ we note that the sequence $\langle r_i,r^+_i:i<j\rangle$ has an
upper bound by (ii). So we may choose a condition $r_j'$ stronger than all
$r_i,r^+_i$ (for $i<j$) and deciding the values of $\name{q}'_j,\name{A}_j$,
say $r_j'\forces_{\bbP_\zeta}$`` $\name{q}'_j=\name{q}''_j\ \&\
\name{A}_j=A_j$ '', where $\name{q}_j''\in N$ and $A_j\in D\cap \bV$. Since
$r_j'\geq r^-_i$ (for all $i<j$) and $r'_j\forces_{\bbP_\zeta}$`` the
sequence $\langle\name{q}_i'': i\leq j\rangle$ is increasing '', and since
$r'_j$ is  $(N,\bbP_\zeta)$--generic, we may choose conditions $r_j^-\in
N\cap \bbP_\zeta$ and $r_j\in\bbP_\zeta$ so that $r^-_i\leq r^-_j\leq r_j$
for $i\leq j$ and $r_j'\leq r_j$ and $r^-_j\forces_{\bbP_\zeta}$`` the
sequence  $\langle\name{q}_i'': i\leq j\rangle$ is increasing ''. Finally,
$r^+_j\in \bbP_\zeta$ is defined essentially by (ii).]

\noindent Now define
\[\name{q}^*_i=\left\{\begin{array}{ll}
\name{\emptyset}_{\name{\bbQ}_\zeta}&\mbox{ if }i<\Upsilon(\zeta),\\ 
\name{q}_\alpha''&\mbox{ if }i=\Upsilon(\zeta)+\alpha.
\end{array}\right.\] 
Clearly, $\langle {r_i^-}^\frown\langle\name{q}_i^*\rangle:i<
\lambda\rangle$ is an increasing sequence of conditions from 
$\bbP_{\zeta^*}\cap N$. Therefore, as $D$ is normal and $A_i\in D$ and
$\bar{f}$ is a $(D,\cS,h)$--semi diamond, we may find a limit ordinal
$\delta\in\cS\cap\mathop{\triangle}\limits_{i<\lambda} A_i$ such that
$\delta>\Upsilon(\zeta)\cdot\omega$ and $\langle h\circ f_\delta(i):
i<\delta\rangle=\langle {r_i^-}^\frown\langle \name{q}_i^* \rangle: 
i<\delta\rangle$. Then for each $\alpha<\delta$ we have 
\[\name{h}^{\langle\zeta\rangle}\circ f_\delta^{\langle\zeta\rangle}(\alpha)
= \name{h}^{\langle\zeta\rangle}\big(f_\delta(\Upsilon(\zeta)+\alpha)\big) =
h\big(f_\delta(\Upsilon(\zeta)+\alpha)\big)(\zeta)=
\name{q}^*_{\Upsilon(\zeta)+\alpha} =\name{q}''_\alpha.\] 
Also, $r_\delta$ is stronger than all $r_i$ (for $i<\delta$) so it forces
that $(\forall i<\delta)(\name{A}_i=A_i\ \&\
\name{q}_i'=\name{q}_i'')$. Hence 
\[r_\delta\forces\mbox{`` }\delta\in \cS\cap\mathop{\triangle}_{\xi<\lambda}
\name{A}_\xi\quad \&\quad \name{h}^{\langle \zeta\rangle}\circ
f_\delta^{\langle\zeta\rangle}=\name{\bar{q}}'\rest\delta\mbox { ''.}\] 
The rest should be clear. 
\medskip

\noindent (1,3)\quad Straightforward. 
\end{proof}
\medskip

Remember that $\langle(\name{\tau}_i,\zeta_i):i<\lambda\rangle$ was fixed in
$(\boxtimes)_1$, $\langle w_i:i<\lambda\rangle$ was chosen in
$(\boxtimes)_2$, the function $\Upsilon$ was defined in $(\boxtimes)_3$, and
$\bar{q},p$ are from $(\boxtimes)_8$. Also, $q^*_i$ are the saturated conditions
picked in $(\boxtimes)_{13}$. 

\begin{claim}
\label{cl5}
There is a sequence $\langle p_i,p_i^+: i<\lambda\rangle\subseteq N$
such that for all $i<\lambda$  we have:    
\begin{enumerate}
\item[$(\boxtimes)_{20}^a$] $(p_i,w_i), (p^+_i,w_i)\in \bbP_{\zeta^*}^{\rm
    RS}\cap N$,  
\item[$(\boxtimes)_{20}^b$] $(p,\{0,\zeta^*\})\leq' (p_j,w_j)\leq'
  (p_j^+,w_j)\leq' (p_i,w_i)$ for all $j<i$, and\\
if  $\vare\in\zeta^*\cap N$ and $\Upsilon(\vare)>i$, or $\vare\in
\zeta^*\setminus N$, and if $G\subseteq\bbP_{\zeta^*}$ is
generic over $\bV$ such that $(p_i^+,w_i)\in' G$, then   
\[\begin{array}{ll}
(p_i^+,w_i)^G\rest \vare\forces_{\bbP_\vare}&\mbox{`` the sequence }
\langle (p_j,w_j)^G(\vare),(p_j^+, w_j)^G(\vare): j\leq i\rangle\\
&\mbox{ is a legal partial play of } \Game^\lambda_0(\name{\bbQ}_\vare)\\
&\mbox{ in which Complete uses her regular strategy $\name{\st}^0_\vare$ '',} 
\end{array}\]
\item[$(\boxtimes)_{20}^c$]  if $\vare\in N\cap\zeta^*$, then
  $p_{\Upsilon(\vare)}(\vare)=p_j(\vare)=p_j^+(\vare)$ for
  all $j\geq \Upsilon(\vare)$,   
\item[$(\boxtimes)_{20}^d$] if $i\in\cS$ is limit and $\vare\in\dom(q_i^*)$,
$\Upsilon(\vare)\geq i$,  then $p_i(\vare)$ is such that for every generic  
$G\subseteq\bbP_{\zeta^*}$ over $\bV$ with $(p_i,w_i)\in' G$, and two 
successive members $\vare',\vare''$ of the set $w_i$ such that 
$\vare'\leq \vare<\vare''$ we have: 
\begin{enumerate}
\item[{}] if $\{p_j(\vare)[G\cap\bbP_{\vare'}][G\cap \bbP_\vare]:
j<i\}\cup\{q_i^*(\vare)[G\cap \bbP_\vare]\}$ has an upper 
bound in $\name{\bbQ}_\vare[G\cap\bbP_\vare ]$, then 
$p_i(\vare)[G\cap\bbP_{\vare'}][G\cap \bbP_\vare]$ is such an upper bound, 
\end{enumerate}
\item[$(\boxtimes)_{20}^e$] for some $\xi\in w_{i+1}\cap \zeta_i\setminus
  (\sup(w_i\cap\zeta_i)+1)$ we have $p_{i+1}\rest \xi=p_i^+\rest \xi$ and
  for some a $\bbP_\xi$--name $\name{\tau}\in N$, for every generic
  $G\subseteq\bbP_{\zeta^*}$ over $\bV$ with $(p_{i+1},w_{i+1})\in' G$, 
we have that $\name{\tau}_i[G\cap \bbP_{\zeta_i}]=
  \name{\tau}[G\cap\bbP_\xi]$.  
\end{enumerate}  
\end{claim}

\begin{proof}[Proof of the Claim]
By induction on $i<\lambda$ we choose $p_i,p_i^+,\cA_i,\cA_i^+,f_i,f_i^+$
such that the following demands $(*)_1$--$(*)_6$ are satisfied. 
\begin{enumerate}
\item[$(*)_1$] $(p_i,w_i),(p_i^+,w_i)\in \bbP^{\rm RS}_{\zeta^*}\cap N$
  and $(\cA_i,f_i,w_i),(\cA_i^+,f_i^+,w_i)\in N$ are their standard
  representations, respectively, and $(p_0,w_0)=(p,\{0,\zeta^*\})$. 
\item[$(*)_2$] If $i<j<\lambda$ then $(\cA_i,f_i,w_i)\preccurlyeq
  (\cA_i^+,f_i^+,w_i) \preccurlyeq (\cA_j,f_j,w_j)$. 
\item[$(*)_3$] If $j<\lambda$, $\vare\in \zeta^*\setminus w_j$, $s_j\in
  \cA_j^+$, $f^+_j(s_j)\leq s_j$ and for $i\leq j$ the conditions
  $r_i\in\cA_i$, $s_i\in \cA_i^+$ are such that $r_i\leq s_i\leq s_j$, then 
\[\begin{array}{ll}
f^+_j(s_j)\rest\vare\forces_{\bbP_\vare}&\mbox{`` }\langle f_i(r_i)(\vare),
f^+_i(s_i)(\vare): i\leq j\rangle\mbox{ is a legal partial play of }
\Game^\lambda_0(\name{\bbQ}_\vare)\\ 
&\mbox{in which Complete uses her regular winning strategy
}\name{\st}^0_\vare \mbox{ ''.}
\end{array}\] 
\item[$(*)_4$] If $\vare\in w_j$, $j<\lambda$ and $r\in
  \cA_j$, $r'\in \cA_j^+$ and $r''\in \cA_{\Upsilon(\vare)}$ 
  are such that $r''\leq r\leq r'$, then 
\[f_j(r)(\vare)=f^+_j(r')(\vare)=f_{\Upsilon(\vare)}(r'')(\vare)
\quad \mbox{ and }\quad p_j(\vare)=p^+_j(\vare)=
p_{\Upsilon(\vare)}(\vare).\] 
\item[$(*)_5$] If $j\in\cS$ is limit and $\vare\in\dom(q^*_j)\setminus
  \bigcup\limits_{i<j} w_i$, $r\in \cA_j$ and $r_i\in \cA_i$ are such that
  $r_i\leq r$ for $i<j$ and if $f_j(r)\leq r$ then 
\[\begin{array}{ll}
\forces_{\bbP_\vare}&\mbox{`` if }\{f_i(r_i)(\vare):i<j\}\cup
\{q^*_j(\vare)\} \mbox{ has an upper bound in }\name{\bbQ}_\vare\\
&\ \mbox{ then $f_j(r_j)(\vare)$ is such an upper bound ''.}
\end{array}\]
\item[$(*)_6$] Let $\xi=\sup(w_{i+1}\cap \zeta_i)$. Then $p_{i+1}\rest
  \xi=p_i^+\rest \xi$ and for some $\bbP_\xi$--name $\name{\tau}\in N$ we
  have: 

if $r\in\cA_{i+1}$ and $f_{i+1}(r)\leq r$, then
$r\forces\name{\tau}=\name{\tau}_i$. 
\end{enumerate}

As declared in $(*)_1$ we set $p_0=p$ and we choose $p_0^+\in
\bbP_{\zeta^*}\cap N$ essentially by $(*)_3$. Thus we let
$\dom(p_0^+)=\dom(p_0)$, $p_0(0)=p_0^+(0)$ and for $\vare\in
\dom(p^+_0)\setminus \{0\}$ we let $p^+_0(\vare)$ be the $<^*_\chi$--first
$\bbP_\vare$--name such that  
\[p^+_0\rest \vare\forces\mbox{`` } p^+_0(\vare) \mbox{ is the answer to
  $p_0(\vare)$ given by the strategy $\name{\st}^0_\vare$ ''.}\]
Next we set $\cA_0=\cA^+_0=\{\emptyset_{\bbP_{\zeta^*}}\}$ and
$f_0(\emptyset_{\bbP_{\zeta^*}})=p_0$, $f_0^+(\emptyset_{\bbP_{\zeta^*}})=
p_0^+$.  
\medskip

Suppose we have chosen $p_i,p_i^+,\cA_i,\cA_i^+,f_i,f_i^+\in N$ for $i<j$ so
that the relevant demands in $(*)_1$--$(*)_6$ are satisfied.
\medskip

{\sc Case:}\qquad $j$ is a successor ordinal, say $j=i_0+1$.\\
Let $\xi=\sup(w_j\cap \zeta_{i_0})$ and $\zeta=\min(w_{i_0}\setminus
\zeta_{i_0})$. By $(\boxtimes)^c_2$ we know $\sup(w_{i_0}\cap\zeta_{i_0})
<\xi$. Working in $N$ apply Proposition \ref{X.7} to
$\xi,\zeta,w_{i_0}, w_j, \name{\tau}_{i_0}$ and $(p_{i_0}^+,w_{i_0})$ to get
$p_j$ and $\name{\tau}$ such that  
\begin{enumerate}
\item[$(\otimes)_1$] $(p_j,w_j)\in \bbP^{\rm RS}_{\zeta^*}\cap N$ and
  $(p^+_{i_0}, w_{i_0})\leq' (p_j,w_j)$,
\item[$(\otimes)_2$] $p_j\rest (\xi\cup [\zeta,\zeta^*))=p_{i_0}^+ \rest
  (\xi\cup [\zeta,\zeta^*))$,
\item[$(\otimes)_3$] $\name{\tau}\in N$ is a $\bbP_\xi$--name,
\item[$(\otimes)_4$] if $G\subseteq \bbP_{\zeta^*}$ is generic over $\bV$
  and $(p_j,w_j)\in' G$, then $\name{\tau}[G]=\name{\tau}_{i_0}[G]$. 
\end{enumerate}
By Proposition \ref{X.6} (used in $N$) we may choose a representation
$(\cA_j,f_j,w_j)\in N$ for $(p_j,w_j)$ such that $(\cA_{i_0}^+,f_{i_0}^+,
w_{i_0})\preccurlyeq (\cA_j,f_j,w_j)$. Necessarily, 
\begin{enumerate}
\item[$(\otimes)_5$] if $r\in\cA_j$ and $f_j(r)\leq r$, then $r\forces
  \name{\tau}=\name{\tau}_{i_0}$.  
\end{enumerate}
Since
\begin{enumerate}
\item[$(\otimes)_6$] if $\Upsilon(\vare)<j$ (i.e., $\vare\in w_{i_0}$), then
  either $\vare<\xi$ or $\vare\geq \zeta$, 
\end{enumerate}
we have $p_j(\vare)=p_{i_0}^+(\vare)$ whenever
$\Upsilon(\vare)<j$. Therefore if $\Upsilon(\vare)<j$, $r\in \cA_j$, $r^+\in
\cA_{i_0}^+$, $r'\in\cA_{\Upsilon(\vare)}$ and $r''\in
\cA_{\Upsilon(\vare)}^+$ are such that $r'\leq r''\leq r^+\leq r$, then
($p_j(\vare)= p^+_{\Upsilon(\vare)}(\vare)=p_{\Upsilon(\vare)}(\vare)$ and) 
\[f_j(r)(\vare)=f_{i_0}^+(r^+)(\vare)= f^+_{\Upsilon(\vare)}(r'')(\vare) =
f_{\Upsilon(\vare)}(r')(\vare).\] 
To define $\cA_j^+,f_j^+$ and $p^+_j$ we first note that 
\begin{enumerate}
\item[$(\otimes)_7$] if $r\in\cA_j$, $f_j(r)\leq r$ and $r_i\in \cA_i$,
  $s_i\in\cA_i^+$ (for $i<j$), are such that $r_i\leq s_i\leq r$, and
  $\vare\in \zeta^*\setminus w_j$, then  
\[\begin{array}{ll}
f_j(r)\rest\vare\forces_{\bbP_\vare}&\mbox{`` }\langle f_i(r_i)(\vare),
f^+_i(s_i)(\vare): i<j\rangle\conc \langle f_j(r)(\vare)\rangle\mbox{ is a
  partial play of}\\
& \Game^\lambda_0(\name{\bbQ}_\vare)\ \mbox{ in which Complete uses her
  strategy }\name{\st}^0_\vare \mbox{ ''.}
\end{array}\] 
\end{enumerate}
Let $r\in\cA_j$. We define $g(r)\in\bbP_{\zeta^*}$ so that
$\dom(g(r))=\dom(f_j(r))$ and 
\begin{enumerate}
\item[$(\otimes)_8^a$] if $\vare\in \dom(g(r))$ satisfies $\Upsilon(\vare)>j$,
  then $g(r)(\vare)$ is the $<^*_\chi$--first $\bbP_\vare$--name for a
  condition in $\name{\bbQ}_\vare$ such that it is forced that 
  \begin{enumerate}
  \item[(a)] if  $\langle f_i(r_i)(\vare),f^+_i(s_i)(\vare): i<j\rangle\conc
    \langle f_j(r)(\vare)\rangle$ is a legal partial play of
    $\Game^\lambda_0(\name{\bbQ}_\vare)$ in which Complete follows her
    regular winning strategy $\name{\st}^0_\vare$, then $g(r)(\vare)$ is the
    answer to this partial play given to Complete by $\name{\st}^0_\vare$; 
\item[(b)]  if the assumptions of (a) are not satisfied then
  $g(r)(\vare)=f_j(r)(\vare)$. 
  \end{enumerate}
\item[$(\otimes)_8^b$] For $\vare\in \dom(g(r))$ with $\Upsilon(\vare)\leq
  j$ we set  $g(r)(\vare)=f_j(r)(\vare)$. 
\end{enumerate}

Clearly, $g$ is a well defined function from $\cA_j$ to $\bbP_{\zeta^*}$ and
$g\in N$. We note that $g(r)\geq f_j(r)$ for $r\in\cA_j$ and
\begin{enumerate}
\item[$(\otimes)_9$] if $\alpha<\beta$ are two successive elements of $w_j$,
  $r^0,r^1\in\cA_j$ are such that $r^0\rest \alpha$ and $r^1\rest \alpha$
  are compatible, then $g(r^0)\rest \beta=g(r^1)\rest\beta$. 
\end{enumerate}
[Why? Let $r^{\ell,+}_i\in\cA^+_i$, $r^\ell_i\in\cA_i$ for $i<j$ be such
that $r^\ell_i\leq r^{\ell,+}_i\leq r^\ell$. Then $r^0_i\rest\alpha$ and 
$r^1_i\rest\alpha$ are compatible and so are $r^{0,+}_i\rest\alpha$ and
$r^{1,+}_i\rest \alpha$. Therefore, by $(*)_1$ and \ref{X.1}(iv),
$f_i(r^0_i)\rest \beta= f_i(r^1_i)\rest\beta$ and $f_i^+(r^{0,+}_i)\rest
\beta= f_i^+(r^{1,+}_i) \rest\beta$. Since the definition of
$g(r^\ell)(\vare)$ involves only  $\langle f_i(r^\ell_i)(\vare),
f_i^+(r^{\ell,+}_i)(\vare): i<j\rangle$ and $f_j(r^\ell)(\vare),
\name{\st}^0_\vare, \Upsilon$, we conclude that
$g(r^0)(\vare)=g(r^1)(\vare)$ for all $\vare<\beta$.] 

Pick a maximal antichain $\cA^+_j\in N$ of $\bbP_{\zeta^*}$ such that 
\begin{enumerate}
\item[$(\otimes)^a_{10}$] $(\forall s\in \cA^+_j)(\exists r\in \cA_j)(r \leq
  s)$, 
\item[$(\otimes)^b_{10}$]  if $s\in \cA^+_j$, $r\in \cA_j$ and $r\leq s$,
  then either $g(r)\leq s$ or for some $\vare\in \dom(s)$ we have $g(r)\rest
  \vare\leq s\rest \vare$ and
\[s\rest\vare \forces_{\bbP_\vare}\mbox{`` the conditions
  $g(r)(\vare),s(\vare)$ are incompatible in $\name{\bbQ}_\vare$ ''.}\]  
\end{enumerate}
Define $f^+_j:\cA^+_j\longrightarrow \bbP_{\zeta^*}$ by $f^+_j(s)=g(r)$
where $r\in \cA_j$ is the unique member satisfying $r\leq s$. Plainly, the
function $f^+_j$ belongs to $N$ and it satisfies condition (v) of
Observation \ref{X.2A} (by $(\otimes)_9$). Hence $(\cA^+_j,f^+_j,w_j)$ is a
standard representation of some $(p^+_j,w_j)\in \bbP^{\rm RS}_{\zeta^*}\cap
N$. Note also that if $s\in \cA^+_j$ and $f^+_j(s)\leq s$, then letting
$r\in\cA_j$ be such that $s\geq r$ we see that $s\geq g(r)\geq f_j(r)$. Thus
$r,f_j(r)$ are compatible and hence $r\geq f_j(r)$. This in turn implies (by
$(\otimes)_7$) that for each $\vare\in \dom(g(r))$ with $\Upsilon(\vare)>j$
the condition $f_j(r)\rest \vare$ forces the assumptions of the case (a) of 
$(\otimes)_8$. Consequently the demand $(*)_3$ (for $f^+_j(s)$) holds. One
easily verifies that also the other relevant demands from $(*)_1$--$(*)_6$
are satisfied by $\cA_j,f_j,p_j,\cA_j^+,f^+_j,p^+_j$.  
\medskip

{\sc Case:}\qquad $j$ is a limit ordinal.\\
If $j\in\cS$ then let $q^*=q^*_j$, otherwise
$q^*=\emptyset_{\bbP_{\zeta^*}}$. Let $\cA^0_j\in N$ be a maximal antichain
of $\bbP_{\zeta^*}$ such that 
\[(\forall s\in \cA^0_j)(\forall i<j)(\exists  r\in \cA_i)(r\leq s).\]
We are going to define a function $g:\cA^0_j\longrightarrow
\bbP_{\zeta^*}$. Let $r\in \cA^0_j$ and for $i<j$ let $r_i\in\cA_i$,
$s_i\in\cA^+_i$ be the unique elements such that $r_i\leq s_i\leq r$. We
define $g(r)$ so that $\dom(g(r))=\bigcup\limits_{i<j}\dom(f_i(r_i))
\cup\dom(q^*)$ and   
\begin{enumerate}
\item[$(\otimes)^a_{11}$] if $\vare\in \dom(g(r))$ and $\Upsilon(\vare)\geq
  j$, then $g(r)(\vare)$ is the $<^*_\chi$--first $\bbP_\vare$--name for a
  condition in $\name{\bbQ}_\vare$ for which it is forced that 
  \begin{enumerate}
\item[(a)] if $\{f_i(r_i)(\vare):i<j\}\cup \{q^*(\vare)\}$ has an upper
    bound in $\name{\bbQ}_\vare$, then $g(r)(\vare)$ is such an upper bound; 
\item[(b)]  if the assumption of (a) is not satisfied but
  $\{f_i(r_i)(\vare):i<j\}$ has an upper  bound in $\name{\bbQ}_\vare$, then
  $g(r)(\vare)$ is such an upper bound;  
\item[(c)] otherwise, $g(r)(\vare)=\name{\emptyset}_{\name{\bbQ}_\vare}$. 
  \end{enumerate}
\item[$(\otimes)^b_{11}$] For $\vare\in\dom(g(r))$ with $\Upsilon(\vare)<j$
  we set $g(r)(\vare)=f_{\Upsilon(\vare)}(r_{\Upsilon(\vare)})(\vare)$. 
\end{enumerate}
Clearly, $g:\cA^0_j\longrightarrow\bbP_{\zeta^*}$ is a well defined
function, $g\in N$. 
\begin{enumerate}
\item[$(\otimes)_{12}$] Assume $r\in \cA^0_j$, $r_i\in\cA_i$, $r_i\leq r$
  for $i<j$. 
  \begin{enumerate}
\item[(a)] If $f_i(r_i)\leq r_i$ for all $i<j$, then $f_i(r_i)\leq g(r)$
    for all $i<j$.
\item[(b)] If $f_i(r_i),r_i$ are incompatible for some $i<j$, then $r$ and
  $g(r)$ are incompatible. 
  \end{enumerate}
\end{enumerate}
[Why? Clause (a) follows immediately from the inductive hypothesis $(*)_3$
and the definition of $g$. To show clause (b) assume that $\vare<\zeta^*$ and
$i_0<j$ are such that 
\begin{itemize}
\item $f_i(r_i)\rest\vare\leq r_i\rest\vare$ for all $i<j$, and 
\item $r_{i_0}\rest \vare\forces$`` $f_{i_0}(r_{i_0})(\vare),
  r_{i_0}(\vare)$ are incompatible '' 
\end{itemize}
(remember \ref {X.1}(iii)).  By Observation \ref{X.2A}(vi) for each $i<j$
  there is $s_i\in\cA_i$ such that $r_i\rest\vare$ and $s_i\rest \vare$ are
  compatible and $f_i(s_i)\leq s_i$. Then also
  $f_i(r_i)\rest(\vare+1)=f_i(s_i)\rest (\vare+1)\leq s_i\rest (\vare+1)$
  and for each $\alpha\leq\vare$ satisfying $\Upsilon(\alpha)\geq j$ we have  
\[\begin{array}{ll}
f_i(r_i)\rest\alpha\forces_{\bbP_\alpha} &\mbox{`` }\langle
f_{i'}(r_{i'})(\alpha):i'<i\rangle\mbox{ are innings of Incomplete in a play
  of}\\
&\Game^\lambda_0(\name{\bbQ}_\alpha)\ \mbox{ in which Complete follows
}\name{\st}^0_\alpha\mbox{ ''.} 
\end{array}\]
(Note that if $i'<i$ and $t\in \cA_{i'}$ is such that $t\leq s_i$, then
$t\rest\vare$ and $r_{i'}\rest \vare$ are compatible, so $f_{i'}(t)\rest
(\vare+1)= f_{i'}(r_{i'})\rest (\vare+1)$.) Consequently, by induction on
$\alpha\leq\vare+1$ we may show that $f_i(r_i)\rest\alpha\leq
g(r)\rest\alpha$ for all $i<j$ and $\alpha\leq \vare$. In particular,
$f_{i_0}(r_{i_0})\rest (\vare+1)\leq g(r)\rest (\vare+1)$ and hence, by the
choice of $i_0$ and $\vare$, $g(r)$ and $r_{i_0}$ are incompatible. Hence
also $g(r)$ and $r$ are incompatible.]

In the same way as for $(\otimes)_9$ we may argue that 
\begin{enumerate}
\item[$(\otimes)_{13}$]  if $\alpha<\beta$ are two successive elements of $w_j$,
  $r^0,r^1\in\cA_j^0$ are such that $r^0\rest \alpha$ and $r^1\rest \alpha$
  are compatible, then $g(r^0)\rest \beta=g(r^1)\rest\beta$. 
\end{enumerate}
Now pick a maximal antichain $\cA_j\in N$ of $\bbP_{\zeta^*}$ such that 
\begin{enumerate}
\item[$(\otimes)^a_{14}$] $(\forall s\in\cA_j)(\exists r\in \cA^0_j)(r\leq
  s)$ and 
\item[$(\otimes)^b_{14}$] if $s\in\cA_j$, $r\in\cA^0_j$ and $r\leq s$ then
  either $g(r)\leq s$ or for some $\vare\in\dom(s)$ we have $g(r)\rest
  \vare\leq s\rest\vare$ and $s\rest\vare\forces_{\bbP_\vare}$``
  $g(r)(\vare),s(\vare)$ are incompatible in $\name{\bbQ}_\vare$ ''. 
\end{enumerate}
Define $f_j:\cA_j\longrightarrow\bbP_{\zeta^*}$ by $f_j(s)=g(r)$ where $r\in
\cA^0_j$ is the unique member such that $r\leq s$ for $s\in\cA^0_j$. Using
$(\otimes)_{12}+(\otimes)_{13}$ and Proposition \ref{X.3} we may easily
argue that $(\cA_j,f_j,w_j)\in N$ is a standard representation of some
$(p_j,w_j)\in \bbP^{\rm RS}_{\zeta^*}\cap N$ such that
$(\cA_i,f_i,w_i)\preccurlyeq (\cA_i^+,f_i^+,w_i)\preccurlyeq
(\cA_j,f_j,w_j)$ for all $i<j$. Finally, exactly as in the successor case
($(\otimes)_7$--$(\otimes)_{10}$ and later) we define $\cA_j,f_j^+$ and
$p^+_j$. We easily verify that the demands $(*)_1$--$(*)_6$ are satisfied. 

This completes the construction of $\langle p_i,p_i^+,\cA_i,\cA^+_i,f_i,
f_i^+: i<\lambda\rangle$. It should be self evident that the conditions
$(*)_1$--$(*)_6$ imply the requirements
$(\boxtimes)^a_{20}$--$(\boxtimes)^e_{20}$. 
\end{proof}

Recalling Remark \ref{RSremark}, we define a condition $r\in\bbP_{\zeta^*}$ by
declaring that its domain (support) is $\dom(r)=\zeta^*\cap N$ and for each
$\zeta\in\zeta^*\cap N$ 
\[\begin{array}{r}
r\rest\zeta\forces\mbox{`` }r(\zeta)\geq p_{\Upsilon(\zeta)}(\zeta)\mbox{
  is generic for $\name{\bar{q}}^{\langle \zeta\rangle}$ over }\\
  N[\name{G}_{\bbP_\zeta}], \name{h}^{\langle\zeta\rangle},
  \name{\bbQ}_\zeta, \name{R}^\pr_\zeta,\bar{f}^{\langle\zeta\rangle}, \cS,
  \Upsilon(\zeta)\cdot\omega\mbox{ ''. }
\end{array}\]  

\begin{claim}
\label{cl3}
For every $\zeta\in (\zeta^*+1)\cap N$, Generic has a winning strategy in
the game $\Game^\cS_1(r\rest\zeta,N,h^{[\zeta]},\bbP_\zeta,R^\pr\rest\zeta,
\bar{f},\bar{q}^{[\zeta]})$.  
\end{claim}

\begin{proof}[Proof of the Claim]
We prove the claim by induction on $\zeta\in (\zeta^*+1)\cap N$, so suppose
that $\zeta\in (\zeta^*+1)\cap N$ and we know that for all $\xi\in
N\cap\zeta$ the condition $r\rest\xi$ is generic for $\bar{q}^{[\xi]}$
over $N,h^{[\xi]},\bbP_\xi,R^\pr\rest\xi,\bar{f},\cS,1$. Note,
the inductive hypothesis implies that:
\begin{enumerate}
\item[$(\boxplus)_0$] $r(\xi)$ is well-defined for $\xi\in N\cap\zeta$ (by
  Claim \ref{cl2}), and so 
\item[$(\boxplus)_1$] for each $\xi\in N\cap\zeta$ we may fix a
  $\bbP_\xi$--name $\name{\st}_\xi$ such that the condition $r\rest\xi$
  forces that it is a winning strategy of Generic in the game
  $\Game^\cS_{\Upsilon(\xi)\cdot\omega}  (r(\xi),N[\name{G}_{\bbP_\xi}],
  \name{h}^{\langle \xi\rangle},  \name{\bbQ}_\xi,\name{R}^\pr_\xi,
  \bar{f}^{\langle \xi\rangle},  \name{\bar{q}}^{\langle\xi\rangle})$, and  
\item [$(\boxplus)_2$] $(p_i^\zeta,w^\zeta_i)\leq'
(r\rest\zeta,\{0,\zeta\})$ for all $i<\lambda$, where $p^\zeta_i=p_i
\rest\zeta$ and $w^\zeta_i=(w_i\cap\zeta)\cup\{\zeta\}$.
\end{enumerate}

\begin{subclaim}
\label{subclaim}
Assume that $(p^*_i,v_i)\in \bbP_\zeta^{\rm RS}\cap N$ (for $i<\delta<
\lambda$), and $p^*\in\bbP_\zeta\cap N$, $r'\in\bbP_\zeta$ are such that  
\[r\rest \zeta\leq r',\quad p^*\leq r'\quad\mbox{and}\quad (\forall
j<i<\delta)((p^*_j,v_j)\leq' (p^*_i,v_i) \leq' (r',\{0,\zeta\})).\] 
Then there are conditions $p^+\in N\cap\bbP_\zeta$ and $r^+\in
\bbP_\zeta$ such that 
\[p^*\leq p^+\leq r^+,\qquad r'\leq r^+\ \mbox{ and }\ (\forall
i<\delta)((p^*_i,v_i)\leq' (p^+,\{0,\zeta\})).\]   
\end{subclaim}

\begin{proof}[Proof of the Subclaim]
By the inductive hypothesis (of Claim \ref{cl3}) we know that $r\rest\xi$
is $(N,\bbP_\xi)$--generic for all $\xi\in\zeta\cap N$. Therefore,
if $\zeta$ is a successor or a limit ordinal of cofinality $\geq\lambda$,
then we immediately get that $r\rest\zeta$ is $(N,\bbP_\zeta)$--generic
(remember \ref{n1.2}(5) and use $(\boxtimes)_1+
(\boxtimes)_{20}^e$). Hence either $r'$ is $(N,\bbP_\zeta)$--generic
or $\zeta$ is a limit ordinal of cofinality $\cf(\zeta)<\lambda$ and
for all $\xi<\zeta$ the condition $r'\rest \xi$ is
$(N,\bbP_\xi)$--generic. In either case Lemma \ref{RSbounds}(3)
applies.   
\end{proof}

Remembering that $\bar{\gamma}=\langle\gamma_\alpha:\alpha<\lambda\rangle$ 
was defined in Context \ref{context}(4) we set 
\[E_0^\zeta \stackrel{\rm def}{=} \{\delta<\lambda:\delta=\gamma_\delta>0
\mbox{ and }\delta>\Upsilon(\zeta)\mbox{ and }(\forall\alpha<\delta)
(\alpha\cdot \omega<\delta)\}\]  
(plainly, $E_0^\zeta\in D$). 
\medskip

We will give a strategy $\st$ for Generic in $\Game^\cS_1(r\rest\zeta, 
N, h^{[\zeta]},\bbP_\zeta,R^\pr\rest\zeta,\bar{f},
\bar{q}^{[\zeta]})$. Essentially, for each $\vare\in\zeta\cap N$, Generic is
going to play a suitable game at the coordinate $\vare$, but at each time
she is playing the game on less than $\lambda$ coordinates. Her actions will
be described on the intervals $[\gamma_\alpha, \gamma_{\alpha+1})$
separately. We may assume that, for $\xi\in N\cap\zeta$, the sets
$\name{C}_i\in  D^{\bV[\name{G}_{\bbP_\xi}]}$ given to Generic by her
winning strategy  $\name{\st}_\xi$ are of the form
$\mathop{\triangle}\limits_{\alpha<\lambda}  \name{A}_\alpha$, where
$\name{A}_\alpha\in D^\bV$.  To keep the ``past and future plays'' under
control, in addition to the  innings $(r^-_i,r_i,C_i)$ of the game, Generic
will construct aside   
\begin{enumerate}
\item[$(\boxplus)_3^\alpha$] $s_\alpha,z_\alpha,z_\alpha^\diamond ,r^*_j,r^+_j,
\name{r}_i^\ominus(\vare),\name{r}_i^\oplus(\vare),\name{A}_i^\xi(\vare), 
\name{\tau}_{i,\vare},A^\xi_{i,\vare}$ for $j<\gamma_{\alpha+1}$, 
$\xi<\lambda$ and $\vare\in N\cap\zeta$ with
$\Upsilon(\vare)<\gamma_{\alpha+1}$ and $i$ such that
$\Upsilon(\vare)+i<\gamma_{\alpha+1}$.  
\end{enumerate}
The primary roles played by these objects are as follows:
\begin{itemize}
\item $\name{r}_i^\ominus(\vare),\name{r}_i^\oplus(\vare), \mathop{\triangle}\limits_{\xi<\lambda} 
\name{A}_i^\xi(\vare)$ are the innings in the stage $i$ of the game played at coordinate
$\vare$,
\item $s_\alpha\in \bbP_\zeta$ are conditions deciding
  $\name{r}_i^\ominus(\vare), \name{A}_i^\xi(\vare)$ and
  $\name{\tau}_{i,\vare},A^\xi_{i,\vare}$ are the  values forced; the
  conditions $s_\alpha$ will be used as $r_j$'s too, 
\item $r^*_j,r^+_j\in \bbP_\zeta$ form a strategic completeness play and
  are used to guarantee that the sequences $\langle r_j:
  j<\gamma_\alpha\rangle$ have upper bounds, 
\item $z_\alpha,z_\alpha^\diamond\in N\cap \bbP_\zeta$ are conditions
  forming  a strategic completeness play and are used to guarantee that the  
  sequences $\langle r^-_j(\vare):j<\gamma_\alpha\rangle$ (are forced to) have
  upper bounds. 
\end{itemize}
Thus we require that the following demands $(\boxplus)_4$--$(\boxplus)_{10}$ are 
satisfied in addition to the rules of the game. 
\begin{enumerate}
\item[$(\boxplus)_4$] $s_\alpha,r^*_j,r^+_j\in\bbP_\zeta$,
  $z_\alpha,z_\alpha^\diamond\in N\cap \bbP_\zeta$, and $r_j^+\leq
  s_\alpha$, $r_i^*\leq r_i^+\leq r_j^*$ for all $i<j<\gamma_{\alpha+1}$. 
\item[$(\boxplus)_5$] If $\gamma_\alpha\notin \cS$ then $r_{\gamma_\alpha}
  \leq r^*_{\gamma_\alpha}$ and if $\gamma_\alpha\in\cS$ then
  $r_{\gamma_\alpha}=r^*_{\gamma_\alpha}$. Also, $s_\alpha\in
  \bigcap\limits_{i<\gamma_{\alpha+1}} \cI_i^{[\zeta]}$ and if
  $\gamma_\alpha<j<\gamma_{\alpha+1}$, then $r_j=s_\alpha$
  and 
\[C_j=E_0^\zeta\cap\bigcap\{A^\xi_{i,\vare}:\xi<\gamma_{\alpha+1}\ \&\
  \vare\in N\cap\zeta\ \&\ \Upsilon(\vare)+i<\gamma_{\alpha+1}\}.\]    
\item[$(\boxplus)_6$] If $\vare\in \zeta\setminus\{\vare'\in N\cap\zeta:
  \Upsilon(\vare')<\gamma_{\alpha+1}\}$,  $j<\gamma_{\alpha+1}$, then    
\[\begin{array}{ll}
r^+_j\rest\vare\forces_{\bbP_\vare}&\mbox{`` }\langle
r^*_i(\vare),r_i^+(\vare): i\leq j\rangle\mbox{ is a legal partial play of }
\Game^\lambda_0(\name{\bbQ}_\vare) \\  
&\mbox{ in which Complete uses her regular winning strategy 
$\name{\st}^0_\vare$ ''.}   
\end{array}\]
\item[$(\boxplus)_7$] If $\vare\in N\cap\zeta$ and $\Upsilon(\vare)+i
<\gamma_{\alpha+1}$, then $\name{r}_i^\ominus(\vare)$ is a
$\bbP_\vare$--name for a condition in  $N[\name{G}_{\bbP_\vare}] \cap 
\name{\bbQ}_\vare$, $r^\oplus_i(\vare)$ is a $\bbP_\vare$--name for a
condition in $\name{\bbQ}_\vare$ and $\name{A}^\xi_i(\vare)$ are
$\bbP_\vare$--names for elements of $D\cap\bV$ (for $\xi<\lambda$) such
that   
\[\begin{array}{ll}
r\rest\vare\forces_{\bbP_\vare}&\mbox{`` }\langle
\name{r}^\ominus_i(\vare),\name{r}_i^\oplus(\vare), \mathop{\triangle}
\limits_{\xi<\lambda}\name{A}_i^\xi(\vare): \Upsilon(\vare)+i<
\gamma_{\alpha+1}\rangle\mbox{ is a partial play}\\  
&\ \mbox{ of the game } \Game^\cS_{\Upsilon(\vare)\cdot\omega} (r(\vare),
N[\name{G}_{\bbP_\vare}],\name{h}^{\langle\vare\rangle},
\name{\bbQ}_\vare, \name{R}^\pr_\vare,\bar{f}^{\langle\vare\rangle},
\name{\bar{q}}^{\langle\vare\rangle}) \\   
&\ \mbox{ in which Generic uses her winning strategy $\name{\st}_\vare$  
  ''.}   
\end{array}\]
\item[$(\boxplus)_8$] If $\vare\in N\cap\zeta$ and $\gamma_\alpha\leq
  j=\Upsilon(\vare)+i<\gamma_{\alpha+1}$, then $\name{\tau}_{i,\vare}\in N$
  is a $\bbP_\vare$--name for a condition in $\name{\bbQ}_\vare$,
  $A^\xi_{i,\vare}\in D$ for $\xi<\gamma_{\alpha+1}$, and   
\[s_\alpha\rest\vare\forces_{\bbP_\vare}\mbox{`` }\name{A}^\xi_i(\vare)=
A^\xi_{i,\vare}\mbox{ and if $j\in\cS$ then } r^-_j(\vare)=
\name{r}_i^\ominus(\vare)=\name{\tau}_{i,\vare}\mbox{ ''.} \] 
Also,  $r^*_j\rest\vare\forces r^*_j(\vare)=r_j^+(\vare)=
\name{r}^\oplus_i(\vare)$.    
\item[$(\boxplus)_9$] If $\gamma_\alpha\notin \cS$, then
  $(p_j^\zeta,w_j^\zeta)\leq' (r_{\gamma_\alpha+1}^-,\{0,\zeta\})$ for all
  $j<\gamma_{\alpha+1}$ (where $p^\zeta_j=p_j\rest\zeta$ and
  $w^\zeta_j=(w_j\cap \zeta)\cup\{\zeta\}$, see $(\boxplus)_2$). 
\item[$(\boxplus)_{10}$] If $\gamma_\alpha\notin\cS$, then
  $r^-_{\gamma_\alpha}\leq z_\alpha\leq z_\alpha^\diamond\leq
  r^-_{\gamma_\alpha+1}$ and $z^\diamond_\beta\leq z_\alpha$ for
  $\beta<\alpha$, $\gamma_\beta\notin\cS$, and for each $\vare\in 
N\cap\zeta$ we have 
\[\begin{array}{ll}
z^\diamond_\alpha\rest\vare\forces_{\bbP_\vare}&\mbox{`` }\langle
z_\beta(\vare),z^\diamond_\beta(\vare): \beta\leq \alpha,\
\gamma_\beta\notin \cS\rangle\mbox{ is a legal partial play of}\\  
&\Game^\lambda_0(\name{\bbQ}_\vare)\mbox{ in which Complete uses her winning
  strategy $\name{\st}^0_\vare$   ''.}   
\end{array}\]
\end{enumerate}
\medskip

Suppose that the players arrived to a stage $\gamma_\alpha$ of the play and
the sequence $\langle r^-_i,r_i,C_i:i<\gamma_\alpha\rangle$ has been
constructed and the objects listed in $(\boxplus)_3^\beta$ have been
chosen for all $\beta<\alpha$. The procedure applied now by Generic 
depends on whether the inning at $\gamma_\alpha$ is given by Generic or
Antigeneric, so we will have two cases below.  
\medskip

\noindent {\sc Case 1:}\quad $\gamma_\alpha\notin\cS$.\\ 
The inning $(r_{\gamma_\alpha}^-,r_{\gamma_\alpha},C_{\gamma_\alpha})$ is
chosen by Antigeneric, but we have to argue first that 
\begin{enumerate}
\item[$(\boxplus)_{11}$] there exists a legal move for Antigeneric. 
\end{enumerate}
For this we note that there is a condition $r'\in\bbP_\zeta$ stronger than
all $r_j$ for $j<\gamma_\alpha$. [Why? If $\alpha=\beta+1$ then $s_\beta$
works, see $(\boxplus)_5$. If $\alpha$ is limit, then look at the conditions
$r^*_j,r^+_j$ for $j<\gamma_\alpha$: by $(\boxplus)_6$--$(\boxplus)_8$ the
sequence $\langle r^*_j:j<\gamma_\alpha\rangle$ has an upper bound and by
$(\boxplus)_5$ this bound will be above all $r_j$ for $j<\gamma_\alpha$.]
Then the condition $r'$ is stronger than $r\rest\zeta$ and stronger
than all $r^-_j$ (for $j<\gamma_\alpha$) and we may use Subclaim
\ref{subclaim} to pick conditions $r^-\in \bbP_\zeta\cap N$ and
$r^*\in\bbP_\zeta$ so that  
\[r^-\leq r^*,\quad r'\leq r^*\ \mbox{ and }\ (\forall
j<\gamma_\alpha)(r^-_j\leq r^-).\]
Then $(r^-,r^*,\lambda)$ is a legal inning of Antigeneric at this stage. 
\medskip

So, let $(r_{\gamma_\alpha}^-,r_{\gamma_\alpha},C_{\gamma_\alpha})$ be the
inning played by Antigeneric at stage $\gamma_\alpha$. By Subclaim
\ref{subclaim} Generic may pick conditions $r^-\in\bbP_\zeta\cap N$ and
$r^+\in \bbP_\zeta$ so that  
\begin{enumerate}
\item[$(\boxplus)_{12}$] $r^-_{\gamma_\alpha}\leq r^-\leq r^+$,
  $r_{\gamma_\alpha}\leq r^+$ and $(\forall j<\gamma_{\alpha+1})((p^\zeta_j,
  w^\zeta_j)\leq' (r^-,\{0,\zeta\}))$.  
\end{enumerate}
Then $z^\diamond_\beta\leq r^-_{\gamma_\beta+1}\leq r^-_{\gamma_\alpha}\leq
r^- \leq r^+$ (for $\beta<\alpha$, $\gamma_\beta\notin\cS$), so by the
argument as in \ref{subclaim} we may apply Lemma \ref{emplystrat} for
$\langle z_\beta,z_\beta^\diamond:\beta<\alpha\ \&\
\gamma_\beta\notin\cS\rangle$ and $r^-,r^+$. This will give us conditions
$z_\alpha,z_\alpha^\diamond\in N\cap \bbP_\zeta$ and $r^*_{\gamma_\alpha}\in
\bbP_\zeta$ such that $r^-\leq z_\alpha\leq z_\alpha^\diamond\leq
r^*_{\gamma_\alpha}$, $r^+\leq r^*_{\gamma_\alpha}$ and for each $\vare\in
\zeta\cap N$ the condition $z^\diamond_\alpha\rest \vare$ forces that ``
the sequence $\langle z_\beta(\vare),z^\diamond_\beta(\vare): \beta\leq
\alpha,\ \gamma_\beta\notin \cS\rangle$ is a legal partial play of 
$\Game^\lambda_0(\name{\bbQ}_\vare)$ in which Complete uses her regular
winning strategy $\name{\st}^0_\vare$ ''. Now, for each $\vare\in N\cap
\zeta$ and $i$ with $\Upsilon(\vare)+i <\gamma_{\alpha+1}$ Generic picks
$\name{r}^\ominus_i(\vare),\name{r}^\oplus_i(\vare), \name{A}^\xi_i(\vare)$
so that  (for $\Upsilon(\vare)+i<\gamma_\alpha$ they are the ones chosen at
previous stages and)
\[\begin{array}{ll}
r\rest\vare\forces_{\bbP_\vare}&\mbox{`` }\langle
\name{r}^\ominus_i(\vare),\name{r}_i^\oplus(\vare), \mathop{\triangle}
\limits_{\xi<\lambda}\name{A}_i^\xi(\vare):\Upsilon(\vare)+i<
\gamma_{\alpha+1} \rangle\mbox{ is a partial play of}\\  
&\ \ \Game^\cS_{\Upsilon(\vare)\cdot\omega}(r(\vare),
N[\name{G}_{\bbP_\vare}],\name{h}^{\langle\vare\rangle}, \name{\bbQ}_\vare,
\name{R}^\pr_\vare,\bar{f}^{\langle \vare\rangle},
\name{\bar{q}}^{\langle\vare\rangle}) \\    
&\ \mbox{ in which Generic uses her winning strategy $\name{\st}_\vare$ 
  ''}   
\end{array}\]
and 
\begin{itemize}
\item if $\Upsilon(\vare)+i=\gamma_\alpha$, then\footnote{Note that
necessarily $i$ is {\em not\/} a successor in this case. Hence if 
$\gamma_\alpha<\Upsilon(\vare)\cdot\omega$, then $i\notin\cS[
\Upsilon(\vare)\cdot\omega]$ and otherwise
$i=\gamma_\alpha\notin\cS$. Consequently, in the game
$\Game^\cS_{\Upsilon(\vare)\cdot\omega}$ the inning at $i$ belongs to
Antigeneric.}     
\[r^*_{\gamma_\alpha}\rest\vare\forces_{\bbP_\vare}\mbox{`` } 
\name{r}^\ominus_i(\vare)=z^\diamond_\alpha(\vare)\ \&\ \name{r}^\oplus_i
(\vare)= r^*_{\gamma_\alpha}(\vare)\ \&\ \name{A}_i^\xi(\vare)=
C_{\gamma_\alpha}\mbox{ '',}\]    
\item if $\gamma_\alpha<\Upsilon(\vare)<\gamma_{\alpha+1}$, then   
\[r^*_{\gamma_\alpha}\rest\vare\forces_{\bbP_\vare}\mbox{`` } 
\name{r}^\ominus_0(\vare)=z^\diamond_\alpha(\vare)\ \&\ 
\name{r}^\oplus_0(\vare)=r^*_{\gamma_\alpha}(\vare)\ \&\ 
\name{A}_0^\xi(\vare)=C_{\gamma_\alpha}\mbox{ ''.}\]  
\end{itemize}
Next, Generic picks $r^+_{\gamma_\alpha}\in\bbP_\zeta$ so that
$\dom(r^+_{\gamma_\alpha})=\dom(r^*_{\gamma_\alpha})$ and the demands of
$(\boxplus)_6$ (for $j=\gamma_\alpha$) are satisfied and
$r^+_{\gamma_\alpha}(\vare)= r^*_{\gamma_\alpha}(\vare)$ whenever $\vare\in
N\cap\zeta$, $\Upsilon(\vare)<\gamma_{\alpha+1}$. For $\gamma_\alpha<j
<\gamma_{\alpha+1}$ she chooses $r^*_j,r^+_j$ so that $\dom(r^*_j)=
\dom(r^+_j)= \dom(r^*_{\gamma_\alpha})$  and $(\boxplus)_6$ holds (for
$\vare\in\zeta\setminus \{\vare'\in N\cap\zeta:\Upsilon(\vare')<
\gamma_{\alpha+1}\}$) and for $\vare\in N\cap\zeta$ with
$\Upsilon(\vare)<\gamma_{\alpha+1}$ we have 
\begin{itemize}
\item if $j<\Upsilon(\vare)$, then $r^*_j(\vare)= r^+_j(\vare)=
  r^*_{\gamma_\alpha}(\vare)$, and
\item if $j=\Upsilon(\vare)+i$, then $r_j^*\rest\vare \forces$``
  $r^*_j(\vare)=r^+_j(\vare)= \name{r}^\oplus_i(\vare)$ ''. 
\end{itemize}
By $(\boxplus)_6$ and $(\boxplus)_7$ the
sequence $\langle r^+_j:j<\gamma_{\alpha+1}\rangle$ has an upper bound in
$\bbP_\zeta$, so Generic may choose $s_\alpha \in\bbP_\zeta$ and
$\name{\tau}_{i,\vare}\in N$, $A^\xi_{i,\vare} \in D$ (for $\vare\in
N\cap\zeta$, $\Upsilon(\vare)+i<\gamma_{\alpha+1}$ and
$\xi<\gamma_{\alpha+1}$) such that 
\begin{itemize}
\item $s_\alpha$ is stronger than all $r^+_j$ (for $j<\gamma_{\alpha+1}$)
  and $s_\alpha\in\bigcap\limits_{j<\gamma_{\alpha+1}}\cI_j^{[\zeta]}$, and 
\item $s_\alpha\rest\vare\forces_{\bbP_\vare}$`` $\name{r}^\ominus_i(\vare)
  =\name{\tau}_{i,\vare}\ \&\ \name{A}^\xi_i(\vare)=A^\xi_{i,\vare}$ ''.  
\end{itemize}
She also defines conditions $r^-_j\in \bbP_\zeta\cap N$ (for
$\gamma_\alpha<j< \gamma_{\alpha+1}$) by declaring that
$\dom(r^-_j)=\dom(z^\diamond_\alpha)\cup\{\vare'\in
N\cap\zeta:\Upsilon(\vare')<\gamma_{\alpha+1}\}$ and  
\begin{itemize}
\item if $\vare\in\dom(r^-_j)\setminus \{\vare'\in N\cap\zeta:
  \Upsilon(\vare')<\gamma_{\alpha+1}\}$ or $\vare\in N\cap\zeta$, $j
  <\Upsilon(\vare)<\gamma_{\alpha+1}$, then $r^-_j(\vare)=z_\alpha^\diamond
  (\vare)$, and  
\item if $\vare\in N\cap\zeta$ and $j=\Upsilon(\vare)+i^*<
  \gamma_{\alpha+1}$, then 
\[\begin{array}{ll}
r^-_j\rest\vare\forces_{\bbP_\vare}&\mbox{`` if }\langle
\name{\tau}_{i,\vare}:\gamma_\alpha\leq \Upsilon(\vare)+i<\gamma_{\alpha+1}
\rangle \mbox{ is an increasing sequence }\\
&\ \mbox{ of conditions above $z^\diamond_\alpha(\vare)$, then } 
r^-_j(\vare)= \name{\tau}_{i^*,\vare},\\
&\ \mbox{ else $r^-_j(\vare)=z^\diamond_\alpha(\vare)$ ''.}
\end{array}\]
\end{itemize}
Then for $\gamma_\alpha<j<\gamma_{\alpha+1}$ Generic plays $r^-_j$ chosen
above, $r_j=s_\alpha$ and $C_j=E_0^\zeta\cap\bigcap\{A^\xi_{i,\vare}:\xi<
\gamma_{\alpha+1}\ \&\ \vare\in N\cap\zeta\ \&\ \Upsilon(\vare)+i<
\gamma_{\alpha+1}\}$. Observe that the rules of the game and the demands
$(\boxplus)_4$--$(\boxplus)_{10}$ are obeyed by the choices above.
\bigskip

\noindent {\sc Case 2:}\quad $\gamma_\alpha\in\cS$ (this may happen only if
$\alpha$ is limit).\\ 
Generic's choices are similar to those from the previous case, except that 
$r_{\gamma_\alpha},r^-_{\gamma_\alpha}$ need to be treated differently, and
this influences the choice of $s_\alpha$ too. 

By $(\boxplus)_7$ at previous stages, for each $\vare\in N\cap\zeta$ with
$\Upsilon(\vare)<\gamma_\alpha$ Generic may pick $\name{r}^\ominus_i(\vare),
\name{r}^\oplus_i(\vare)$ and $\name{A}^\xi_i(\vare)$ (for
$\gamma_\alpha\leq \Upsilon(\vare)+i<\gamma_{\alpha+1}$) so that
$(\boxplus)_7$ still holds. By $(\boxplus)_6+(\boxplus)_8$ she may 
choose a condition $r'\in\bbP_\zeta$ stronger than all $r_j$ for
$j<\gamma_\alpha$ and such that $r'(\vare)= \name{r}^\oplus_i
(\vare)$ when $\vare\in N\cap\zeta$, $\Upsilon(\vare)+i=\gamma_\alpha$,
$0<i$. Now, for $\vare\in \zeta\cap N$ with $\gamma_\alpha\leq
\Upsilon(\vare)< \gamma_{\alpha+1}$ we have that 
\[\begin{array}{ll}
r'\rest\vare\forces_{\bbP_\vare}&\mbox{`` the condition $r'(\vare)$ is
  $(N[\name{G}_{\bbP_\vare}],\name{\bbQ}_\vare)$--generic and}\\
&\ \ r'(\vare)\geq r^-_j(\vare)\mbox{ for all $j<\gamma_\alpha$ ''.}
\end{array}\]
Therefore, by Observation \ref{tostart}, Generic may pick $\bbP_\vare$--names
$\name{r}^\ominus_0(\vare), \name{r}^\oplus_0(\vare)$ for conditions in
$\name{\bbQ}_\vare$ such that 
\[r\rest\vare\forces_{\bbP_\vare}\mbox{`` }\name{r}^\ominus_0(\vare) \in
N[\name{G}_{\bbP_\vare}]\ \&\ \name{r}^\ominus_0(\vare)\leq
\name{r}^\oplus_0(\vare) \ \&\ r(\vare)\leq \name{r}^\oplus_0(\vare)\mbox{
  '',} \] 
and
\[r'\rest\vare\forces_{\bbP_\vare}\mbox{`` }r'(\vare)\leq
\name{r}^\oplus_0(\vare)\ \&\ (\forall j<\gamma_\alpha)(r^-_j(\vare)\leq
\name{r}^\ominus_0(\vare))\mbox{ ''.}\] 
Then for $\vare\in \zeta\cap N$, $\xi<\lambda$ and $i$ such that
$\gamma_\alpha\leq\Upsilon(\vare)<\Upsilon(\vare)+i<\gamma_{\alpha+1}$
Generic picks $\bbP_\vare$--names $\name{r}_i^\ominus(\vare),
\name{r}_i^\oplus(\vare)$ and $\name{A}^\xi_i(\vare)$  so that
$(\boxplus)_7$ holds.   

Now she declares that $\dom(r_{\gamma_\alpha})=\dom(r_{\gamma_\alpha}^*)=
\dom(r')$ and for $\vare\in\dom(r_{\gamma_\alpha})$ she sets 
\[r_{\gamma_\alpha}(\vare)=r_{\gamma_\alpha}^*(\vare)=
\left\{\begin{array}{ll}
\name{r}^\oplus_0(\vare)&\mbox{ if }\vare\in
N\cap\zeta,\ \gamma_\alpha\leq\Upsilon(\vare)<\gamma_{\alpha+1},\\  
r'(\vare) &\mbox{ otherwise.}
\end{array}\right.\]
Then $r^+_{\gamma_\alpha},r^*_j,r^+_j$ (for $\gamma_\alpha<j
<\gamma_{\alpha+1}$) are defined exactly as in the previous case. Like
before, the sequence $\langle r^+_j:j<\gamma_{\alpha+1}\rangle$ has an upper
bound in $\bbP_\zeta$, so Generic may find a condition $s'\in\bbP_\zeta$, 
names $\name{\tau}_{i,\vare}\in N$ for conditions in $\name{\bbQ}_\vare$ and
sets $A^\xi_{i,\vare}\in D$ (for $\vare\in N\cap\zeta$, $\gamma_\alpha\leq
\Upsilon(\vare)+i<\gamma_{\alpha+1}$ and $\xi<\gamma_{\alpha+1}$) such that  
\begin{itemize}
\item $r^+_j\leq s'$ for all $j<\gamma_{\alpha+1}$ and $s'\in
  \bigcap\limits_{j<\gamma_{\alpha+1}}\cI_j^{[\zeta]}$, and  
\item $s'\rest\vare\forces_{\bbP_\vare}$`` $\name{r}^\ominus_i(\vare)
  =\name{\tau}_{i,\vare}\ \&\ \name{A}^\xi_i(\vare)=A^\xi_{i,\vare}$ ''.  
\end{itemize}
Clearly $r^-_j\leq r_j\leq r'\leq r_{\gamma_\alpha}\leq s'$ for
$j<\gamma_\alpha$, so Generic may use Subclaim \ref{subclaim}  
to pick conditions $s_\alpha\in\bbP_\zeta$ and $r^-\in N\cap\bbP_\zeta$ such  
that $r^-\leq s_\alpha$, $s'\leq s_\alpha$ and $(\forall j<\gamma_\alpha)
(r^-_j\leq r^-)$. Moreover, Generic may modify $r^-$ so that, additionally,
for each $\vare\in\dom(r^-)$ we have 
\[\begin{array}{ll}
\forces_{\bbP_\vare}&\mbox{`` if the sequence }\langle r^-_j(\vare):j<
\gamma_\alpha\rangle\mbox{ has an upper bound in }\name{\bbQ}_\vare\\
&\ \mbox{ then }r^-(\vare)\mbox{ is such a bound ''.}
\end{array}\] 
Then, for $\gamma_\alpha\leq j<\gamma_{\alpha+1}$, Generic sets
$\dom(r^-_j)=\dom(r^-)$ and  
\begin{itemize}
\item if $\vare\in\dom(r^-_j)\setminus \{\vare'\in N\cap\zeta:
  \Upsilon(\vare')< \gamma_{\alpha+1}\}$, then $r^-_j(\vare)=r^-(\vare)$,
  and   
\item if $\vare\in N\cap\zeta$, $\Upsilon(\vare)<\gamma_{\alpha+1}$ and
  $i^*$ is $0$ if $j\leq\Upsilon(\vare)$ and $i^*$ is such that
  $j=\Upsilon(\vare)+i^*$ if $\Upsilon(\vare)<j<\gamma_{\alpha+1}$, then
  $r^-_j(\vare)$ is the $<^*_\chi$--first name for a condition in
  $\name{\bbQ}_\vare$ such that 
\[\begin{array}{ll}
r^-_j\rest\vare\forces_{\bbP_\vare}&\mbox{`` if }\langle
\name{\tau}_{i,\vare}:\gamma_\alpha\leq \Upsilon(\vare)+i<\gamma_{\alpha+1}
\rangle \mbox{ is an increasing sequence of}\\
&\mbox{ conditions, and $\name{\tau}_{0,\vare}$ is above all
  $r_{j'}^-(\vare)$ for $j'<\gamma_\alpha$,}\\
&\mbox{ then } r^-_j(\vare)=\name{\tau}_{i^*,\vare},\ \  \mbox{ and else
  $r^-_j(\vare)=r^-(\vare)$ ''.}  
\end{array}\]
\end{itemize}
Since in the current case $\alpha$ is a limit ordinal, $(\boxplus)_{10}$
implies that for each $\vare\in N\cap \zeta$ and a condition
$z\in\bbP_\vare$ stronger than all $r^-_{j'}\rest\vare$ (for
$j'<\gamma_\alpha$) we have
\[z\forces_{\bbP_\vare}\mbox{`` the sequence }\langle r^-_{j'}(\vare):
j'<\gamma_\alpha\rangle\mbox{ has an upper bound in $\name{\bbQ}_\vare$
  '',}\] 
and thus $z\forces_{\bbP_\vare}$`` $r^-_{j'}(\vare)\leq r^-(\vare)$ for all
$j'<\gamma_\alpha$ '' (remember the additional demand on
$r^-$). Consequently, the conditions $r^-_j$ (for $\gamma_\alpha\leq
j<\gamma_{\alpha+1}$) are well defined, they belong to $N$ and if
$j'<\gamma_\alpha$ and $\gamma_\alpha\leq j_1<j_2<\gamma_{\alpha+1}$ then
$r^-_{j'}\leq r^-_{j_1}\leq r^-_{j_2}\leq s_\alpha$. Finally, for
$\gamma_\alpha\leq j<\gamma_{\alpha+1}$ Generic plays $r^-_j$ chosen above,
$r_{\gamma_\alpha}$ and $r_j=s_\alpha$ if $j>\gamma_\alpha$, and
$C_j=E_0^\zeta\cap\bigcap\{A^\xi_{i,\vare}:\xi<\gamma_{\alpha+1}\ \&\
\vare\in N\cap\zeta\ \&\ \Upsilon(\vare)+i<\gamma_{\alpha+1}\}$. Easily the
rules of the game\footnote{Note: in this case it is not required that
  $r^-_{\gamma_\alpha}\leq r_{\gamma_\alpha}$ as $\gamma_\alpha\in \cS\cap
  \cR$.},   and the demands $(\boxplus)_4$--$(\boxplus)_{10}$ are obeyed by
the choices above. 

\begin{subclaim}
\label{sub}
The strategy $\st$ described above is a winning strategy for Generic in
$\Game^\cS_1(r\rest\zeta,N[\name{G}_{\bbP_\zeta}], h^{[\zeta]},
\bbP_\zeta,R^\pr\rest\zeta, \bar{f},\bar{q}^{[\zeta]})$.   
\end{subclaim}

\begin{proof}[Proof of the Subclaim]
Suppose that $\langle r_j^-,r_j,C_j:j<\lambda\rangle$ is a result of a 
play of $\Game^\cS_1(r\rest\zeta,N[\name{G}_{\bbP_\zeta}], h^{[\zeta]},
\bbP_\zeta,R^\pr\rest\zeta, \bar{f},\bar{q}^{[\zeta]})$ in which Generic 
follows the strategy $\st$. By what was said in the description of the 
strategy (specifically in $(\boxplus)_{11}$) both players had always 
legal moves, so the play lasted really $\lambda$ steps. Let 
\[s_\alpha,z_\alpha,z^\diamond_\alpha,r^*_j,r^+_j,
\name{r}_i^\ominus(\vare),\name{r}_i^\oplus(\vare),
\name{A}_i^\xi(\vare),\name{\tau}_{i,\vare}, A^\xi_{i,\vare}\]  
be the objects written aside by Generic (so they satisfy the demands
$(\boxplus)_4$--$(\boxplus)_{10}$).   

We will argue that condition $(\circledast)$ of Definition \ref{n1.1}(3)
holds (i.e., Generic wins).
\medskip

Assume that a limit ordinal $\delta\in\cS\cap\bigcap\limits_{j<\delta}C_j$
is such that\footnote{Note: $h^{[\zeta]}\circ f_\delta$ is an increasing
sequence of conditions by $(\boxtimes)_9$} $h^{[\zeta]}(f_\delta(j))=r^-_j$
for each successor $j<\delta$. Then also $\delta\in E^\zeta_0$, so
$(\forall \alpha<\delta)(\alpha\cdot\omega<\delta)$ and $\Upsilon(\zeta)<
\delta=\gamma_\delta=\sup(\gamma_\alpha:\alpha<\delta\ \&\ \gamma_\alpha 
\notin\cS)$. Therefore, by $(\boxplus)_9$, we have that 
\[\big(\forall i<\delta\big)\big(\exists j<\delta\big)\big(
(p^\zeta_i,w^\zeta_i)\leq'(r^-_j,\{0,\zeta\})\big),\]
and consequently if $\zeta'=\min(\zeta^*_\delta,\zeta)$, then 
$(p^{\zeta'}_i,w^{\zeta'}_i)\leq' (q^*_\delta\rest\zeta',\{0,\zeta'\})$
for all $i<\delta$. (Remember, $q^*_\delta$ is an upper bound to
$h^{[\zeta^*_\delta]} \circ f_\delta$, see $(\boxtimes)_{12}+
(\boxtimes)_{13}$.) Hence, by $(\boxplus)^d_{20}+(\boxplus)_2$,  
\begin{enumerate}
\item[$(\boxplus)_{13}$] if $\vare\in\zeta\cap\zeta^*_\delta\cap N$, 
$\Upsilon(\vare)\geq \delta$ and $z\in\bbP_\vare$ is stronger than both 
$q^*_\delta\rest\vare$ and $r\rest\vare$, then $z\forces q^*_\delta(\vare)
\leq r(\vare)$.
\end{enumerate}
Let $\bar{\beta}=\langle\beta(\vare)\leq\vare^*\rangle$ be the increasing
(continuous) enumeration of $w_\delta\cap (\zeta+1)$. Then $\beta(0)=0\in
w_0$ and $\beta(\vare^*)=\zeta\in w_{i^*}$, where $i^*=\Upsilon(\zeta)
<\delta$. Since  $r_\delta$ is stronger than all $r^-_j,r_j$ (for
$j<\delta)$ we also have that for each $\vare\leq\vare^*$: 
\begin{enumerate}
\item[$(\boxplus)^a_{14}$] $r_\delta\rest \beta(\vare)$ is an upper bound to
  $h^{[\beta(\vare)]}\circ f_\delta$, and 
\item[$(\boxplus)^b_{14}$] $r_\delta\rest\beta(\vare)\in
  \bigcap\limits_{\alpha<\delta} \cI^{[\beta(\vare)]}_\alpha$ (by
  $(\boxplus)_5$), and
\item[$(\boxplus)^c_{14}$] if $\beta(\vare)\in w_i\cap\zeta$ for some
  $i<\delta$, then $\Upsilon\big(\beta(\vare)\big)\cdot \omega<\delta$ and
  $\delta\in\bigcap\{A^\xi_{j,\beta(\vare)}:\xi,j<\delta\}$, and hence (by 
  $(\boxplus)_8$) the condition $r_\delta\rest\beta(\vare)$ forces (in
  $\bbP_{\beta(\vare)}$) that: 
\[\begin{array}{l}
\delta\in 
\mathop{\triangle}\limits_{j<\lambda} \mathop{\triangle}\limits_{\xi<
  \lambda}\name{A}_j^\xi(\beta(\vare))\cap\cS[\Upsilon(\beta(\vare))\cdot
\omega] \mbox{  and}\\ 
\name{h}^{\langle\beta(\vare)\rangle}\circ f_\delta^{\langle
  \beta(\vare)\rangle}\mbox{is an increasing sequence of conditions in
  $\name{\bbQ}_{\beta(\vare)}$ and}\\
\name{h}^{\langle\beta(\vare)\rangle}\big(f_\delta^{\langle\beta(\vare)
\rangle}(j)\big) =h\big(f_\delta\big(\Upsilon(\beta(\vare))+j\big) \big)  
\big(\beta(\vare)\big) =r^-_{\Upsilon(\beta(\vare))+j}\big(\beta(\vare)
\big)=\\
 \name{r}^\ominus_j\big(\beta(\vare)\big)\mbox{ for all successor 
$j<\delta$.} 
\end{array}\]
\end{enumerate}
By induction on $\vare\leq\vare^*$ we are going to show that 
\begin{enumerate}
\item[$(\boxplus)_{15}^a$] $(h^{[\beta(\vare)]}\circ f_\delta)\;
  (R^\pr\rest\beta(\vare))\; (r_\delta\rest\beta(\vare))$, and 
\item[$(\boxplus)_{15}^b$] $\beta(\vare)\leq\zeta^*_\delta$ and
  $q_\delta^{[\beta(\vare)]}= q^*_\delta \rest \beta(\vare)\leq
  r_\delta\rest\beta(\vare)$. 
\end{enumerate}

To start we note that $\langle \name{r}^\ominus_j(0), \name{r}^\oplus_j(0),
\mathop{\triangle}\limits_{\xi<\lambda}\name{A}_j^\xi(0): j<\delta\rangle$
is a legal partial play of $\Game^\cS_0(r(0), N,\name{h}^{\langle 
  0\rangle},\bbQ_0,\name{R}^\pr_0,\bar{f}^{\langle   0\rangle},
\name{\bar{q}}^{\langle 0\rangle})$ in which Generic uses her 
winning strategy $\name{\st}_0$ (and all names involved are actually objects
from $\bV$). Also $\delta\in \mathop{\triangle}\limits_{j<\lambda}
\mathop{\triangle}\limits_{\xi<\lambda}\name{A}_j^\xi(0)$,
$\name{r}^\ominus_j(0)= r^-_j(0)$ for all $j\in\delta\cap\cS$ (by 
$(\boxplus)_8$) and $r^-_j(0)=\name{h}^{\langle 0\rangle}(f_\delta(j))$ for
all successor $j<\delta$. Therefore, as the play on the coordinate $\vare=0$ 
is won by Generic, $r_\delta(0)=r^*_\delta(0)=\name{r}^\oplus_\delta(0)\geq 
\name{q}^{\langle 0\rangle}_\delta$ and $(\name{h}^{\langle 0\rangle}\circ 
f_\delta)\;\name{R}^\pr_0\; r_\delta(0)$. Also, by
$(\boxplus)^a_{14}$, $r_\delta\rest\beta(1)$ is an upper bound to
$h^{[\beta(1)]}\circ f_\delta$ and hence $\big(h^{[\beta(1)]}\circ f_\delta
\big)\; \big(R^\pr\rest \beta(1)\big)\; \big(r_\delta\rest \beta(1)\big)$. And
since $r_\delta\rest\beta(1)\in\bigcap\limits_{\alpha<\delta}
\cI_\alpha^{[\beta(1)]}$ (see $(\boxplus)^b_{14}$) we may conclude that
$\zeta^*_\delta\geq \beta(1)$ and therefore $q_\delta^{[\beta(1)]}=q^*_\delta
\rest \beta(1)$ (as $\beta(1)$ is not a limit of members of $w_\delta$). Now
by induction on $\xi\leq\beta(1)$ we argue that $q^*_\delta\rest \xi\leq
r_\delta\rest\xi$: it follows from $(\boxtimes)_{19}$ and what has been said
above that $q^*_\delta(0)=\name{q}^{\langle 0\rangle}_\delta\leq
r_\delta(0)$. Suppose $0<\xi<\beta(1)$, $\xi\in N$, and assume that we have
shown $q^*_\delta\rest\xi\leq r_\delta\rest\xi$. Then, by $(\boxplus)_{13}$,
$r_\delta\rest\xi\forces_{\bbP_\xi} q^*_\delta(\xi)\leq r(\xi)\leq
r_\delta(\xi)$. This concludes the arguments for $(\boxplus)_{15}$ when
$\vare=1$. 

Suppose we have justified $(\boxplus)_{15}$ for $\vare_0<\vare^*$ and
$\Upsilon(\beta(\vare_0))<\delta$. By $(\boxplus)_7$, the condition
$r_\delta\rest\beta(\vare_0)$ forces that 
\[\langle\name{r}^\ominus_i( \beta(
\vare_0)), \name{r}^\oplus_i(\beta(\vare_0)), \mathop{\triangle}\limits_{\xi
<\lambda}\name{A}_i^\xi(\beta(\vare_0)): \Upsilon(\beta(\vare_0))+i<
\delta\rangle\]
is a legal partial play of the game 
\[\Game^\cS_{\Upsilon(\vare
  (\beta_0))} (r(\beta(\vare_0)), N[\name{G}_{\bbP_{\beta(\vare_0)}}],
\name{h}^{\langle \beta(\vare_0)\rangle},\name{\bbQ}_{\beta(\vare_0)}, 
\name{R}^\pr_{\beta(\vare_0)}, \bar{f}^{\langle \beta(\vare_0)\rangle},
\name{\bar{q}}^{\langle \beta(\vare_0)\rangle})\]
in which Generic uses the winning strategy
$\name{\st}_{\beta(\vare_0)}$. Therefore, using $(\boxplus)_{14}^c$ and
$(\boxplus)_5+(\boxplus)_8$, we get that $r_\delta\rest\beta(\vare_0)$
forces (in $\bbP_{\beta(\vare_0)}$) that   
\[\mbox{`` } r_\delta(\beta(\vare_0)) = \name{r}^\oplus_\delta
(\beta(\vare_0))\geq\name{q}_\delta^{\langle \beta(\vare_0)\rangle}\ \mbox{
  and }\  (\name{h}^{\langle \beta(\vare_0)\rangle}\circ f_\delta^{\langle 
  \beta(\vare_0) \rangle})\; \name{R}^\pr_{\beta(\vare_0)}\;
  r_\delta(\beta(\vare_0))\mbox{ ''.}\]  
Since $\name{h}^{\langle \beta(\vare_0)\rangle}\circ f_\delta^{\langle
    \beta(\vare_0) \rangle} (j)=h(f_\delta(\Upsilon(\beta(\vare_0))+j)
  (\beta(\vare_0))$ and $r_\delta\rest \beta(\vare_0+1)$ is an upper bound
  to $h^{[\beta(\vare_0)]}\circ f_\delta$ (by $(\boxplus)^a_{14}$), we may
  use our inductive hypothesis of $(\boxplus)^a_{15}$ for $\vare_0$ to
  conclude that $(h^{[\beta(\vare_0+1)]}\circ f_\delta)\; (R^\pr\rest
  \beta(\vare_0+1))\; (r_\delta\rest \beta(\vare_0))$ (i.e.,
  $(\boxplus)^a_{15}$ for $\vare=\vare_0+1$). By $(\boxplus)^b_{14}$ and the
  inductive hypothesis of $(\boxplus)^b_{15}$ for $\vare_0$ we see now that
  $\zeta^*_\delta\geq \beta(\vare_0+1)$ and also, as $\beta(\vare_0+1)$ is
  not the limit of elements of $w_\delta$, $q_\delta^{[\beta(\vare_0+1)]}=
  q^*_\delta\rest \beta(\vare_0+1)$. Like before we may use
  $(\boxplus)_{13}+(\boxtimes)_{19}$ to show inductively that for all
  $\xi\leq \beta(\vare_0+1)$ we have $q^*_\delta\rest\xi\leq
  r_\delta\rest\xi$, getting $(\boxplus)^b_{15}$ for $\vare=\vare_0+1$. 

Suppose that we have shown $(\boxplus)_{15}$ for $\vare_0<\vare^*$ and
$\Upsilon(\beta(\vare_0)) =\delta$ (so $\beta(\vare_0)$ is the limit of
points from $\bigcup\limits_{i<\delta} w_i$). The assumption of
$(\boxplus)^a_{15}$ for $\vare_0$ implies immediately that
$(h^{[\beta(\vare_0+1)]}\circ f_\delta)\; (R^\pr\rest \beta(\vare_0+1))\;
(r_\delta\rest \beta(\vare_0))$ as there are no new ``active'' coordinates
here (remember $(\boxtimes)^b_4$ and $(\boxplus)^a_{14}$). The same way as
before we argue that $(\boxplus)^b_{15}$ holds for $\vare=\vare_0+1$ as
well. 

Now suppose that $\vare\leq\vare^*$ is a limit ordinal and that we have
shown $(\boxplus)_{15}$ for all $\vare'<\vare$. Since
$\beta(\vare)=\sup(\beta(\vare'):\vare'<\vare)$, the assumption of
$(\boxplus)^a_{15}$ for $\vare'<\vare$ implies that  $(h^{[\beta(\vare)]}
\circ f_\delta)\; (R^\pr\rest \beta(\vare))\; (r_\delta\rest \beta(\vare))$.
The assumption of $(\boxplus)^b_{15}$ for $\vare'<\vare$ implies that
$\zeta^*_\delta\geq \beta(\vare)$ and $q^*_\delta\rest \beta(\vare) \leq
r_\delta\rest \beta(\vare)$. Hence, remembering $(\boxplus)^b_{14}$, we also
conclude that $q^*_\delta\rest\beta(\vare)\in \bigcap_{\alpha<\delta}
\cI_\alpha^{[\beta(\vare)]}$ and therefore $q_\delta^{[\beta(\vare)]}=
q^*_\delta\rest \beta(\vare)$. Consequently, the proof of $(\boxplus)_{15}$
is completed. 
\medskip

Finally, $(\boxplus)_{15}$ for $\vare^*$ implies that Generic won the
considered play, finishing the proof of the Subclaim. 
\end{proof}
\end{proof}

To conclude the Theorem we will argue that the strategy $\st$ for 
Generic in $\Game^\cS_1(r,N,h,\bbP_{\zeta^*},R^\pr,
\bar{f},\bar{q}^{[\zeta^*]})$ described in the proof of Claim 
\ref{cl3} for $\zeta=\zeta^*$ is also a winning strategy for Generic in
$\Game^\cS_1(r,N,h,\bbP_{\zeta^*},R^\pr, \bar{f},\bar{q})$. Note that if
$\zeta=\zeta^*$ and in the proof of Subclaim \ref{sub} we look at
$\vare=\vare^*$ (i.e., $\beta(\vare)=\zeta^*$), then we obtain that for
relevant $\delta$:
\begin{enumerate}
\item[$(\alpha)$]  $r_\delta$ is an upper bound to $h\circ f_\delta$ (by
  $(\boxplus)^a_{14}$), 
\item[$(\beta)$] $r_\delta\in \bigcap\limits_{\alpha<\delta}\cI_\alpha$ (by
  $(\boxplus)^b_{14}$), 
\item[$(\gamma)$] $(h\circ f_\delta)\; R^\pr\; r_\delta$ (by
  $(\boxplus)^a_{15}$). 
\end{enumerate}
Therefore, $q_\delta$ must have the same properties $(\alpha)$--$(\gamma)$,
and by $(\boxtimes)^b_{13}$ we get $\zeta^*_\delta=\zeta^*$ and
$q^*_\delta=q_\delta$. Now $(\boxplus)^b_{15}$ implies
$q_\delta^{[\zeta^*]}=q_\delta$ and we see that Generic wins indeed. 

Thus the condition $r$ is generic for $\bar{q}$ over
$N,h,\bbP_{\zeta^*},R^\pr,\bar{f}$ and this completes the proof of the
Theorem.
\end{proof}

\subsection*{Acknowledgements}
We would like to thank the referees for their valuable comments which helped
to improve the manuscript.

Both authors acknowledge support from the United States-Israel
Binational Science Foundation (Grant no. 2010405). This is publication
1001 of the second author.


\end{document}